\documentclass{article}

\usepackage[pages=all, color=black, position={current page.south}, placement=bottom, scale=1, opacity=1, vshift=5mm]{background}
\SetBgContents{}
\usepackage[margin=1in]{geometry} % full-width

\usepackage{todonotes}
\setuptodonotes{inline}
\usepackage{multirow}
\usepackage{comment}
\usepackage{amsmath}
\usepackage{amsthm}
\usepackage{amssymb}
\usepackage{booktabs}
\usepackage[utf8]{inputenc}
\usepackage{hyperref}
\hypersetup{
	unicode,
    pdfauthor={Christoph Plate, Julius Martensen, Sebastian Sager},
    pdftitle={OEDforUDE},
    pdfkeywords={Optimal Experimental Design, Hybrid Models, Optimal Control, Universal Differential Equations},
	pdfproducer={LaTeX},
	pdfcreator={pdflatex}
}

\theoremstyle{plain}
\newtheorem{theorem}{Theorem}
\newtheorem{corollary}[theorem]{Corollary}
\newtheorem{lemma}[theorem]{Lemma}

\theoremstyle{definition}
\newtheorem{definition}[theorem]{Definition}

\newtheorem{algo}{Algorithm}
\DeclareMathOperator{\tr}{tr}
\usepackage{mathrsfs} % for \mathscr command
\usepackage{caption}
\usepackage{subcaption}
\usepackage{bbm}

\newcommand{\myheader}[1]{\medskip \noindent \textbf{#1.} }
\newcommand{\R}{\mathbb R}
\newcommand{\tf}{t_{\mathrm{f}}}
\newcommand{\myU}{U}
\newcommand{\trace}{\text{trace}}

\newcommand{\svdu}[1]{{\text{SVD}_{#1}}}
\newcommand{\psvdu}[1]{\hat{p}\text{-}\svdu{#1}}
\newcommand{\tsvdu}[1]{\theta\text{-}\svdu{#1}}

\title{Optimal Experimental Design for Universal Differential Equations}
\author{Christoph Plate$^{1,2}$ \and Carl Julius Martensen$^1$ \and Sebastian Sager$^1$}

\date{
	$^1$Otto von Guericke University Magdeburg \\ \texttt{\{christoph.plate, carl.martensen, sager\}@ovgu.de}\\%
	$^2$Max Planck Institute for Dynamics of Complex Technical Systems Magdeburg\\
	\texttt{\{plate\}@mpi-magdeburg.mpg.de}\\%
	\today
}

\begin{document}
\maketitle

\begin{abstract}
\noindent
Complex dynamic systems are typically either modeled using expert knowledge in the form of differential equations or via data-driven universal approximation models such as artificial neural networks (ANN).
While the first approach has advantages with respect to interpretability, transparency, data-efficiency, and extrapolation, the second approach is able to learn completely unknown functional relations from data and may result in models that can be evaluated more efficiently.
To combine the complementary advantages, universal differential equations (UDE) have been suggested. They replace unknown terms in the differential equations with ANN.
Such hybrid models allow to both encode prior domain knowledge, such as first principles, and to learn unknown mechanisms from data.
Often, data for the training of UDE can only be obtained via costly experiments.
We consider optimal experimental design (OED) for planning of experiments and generating data needed to train UDE.
The number of weights in the embedded ANN usually leads to an overfitting of the regression problem. 
To make the OED problem tractable for optimization, we propose and compare dimension reduction methods that are based on lumping of weights and singular value decomposition of the Fisher information matrix (FIM), respectively.
They result in lower-dimensional variational differential equations, which are easier to solve and yield regular FIM.
Our numerical results showcase the advantages of OED for UDE, such as increased data-efficiency and better extrapolation properties.
\end{abstract}

%\tableofcontents

%\todo{Vorschläge für target-journals: Computers \& Chemical Engineering (wie Misener und Biegler, IF 4.2) oder Optimization and Engineering (IF 2.1, da sollten wir gute Chancen haben) oder sogar Automatica (IF 6.1, wenn wir noch ein paar theoretische Resultate finden).}

%\todo{Roter Faden: Formulieren von OED für UDEs. Lumping parameter Trick. Implementierung. Zwei Beispiele (Lotka und urethan), die Vorteile a) der Versuchsplanung allgemein, b) Vorteil der SVD zur Modellreduktion und c) Performanz der lumping parameter Ansätze zeigen.} 

\section{Introduction}
\label{sec:intro}

%%% General introduction: unknown models and expensive data in engineering
\myheader{Model identification and usage}
The engineering sciences have been successfully using concepts from modeling, simulation, and optimization for many decades. 
Depending on particular applications, optimization has been used for different model-based design and control purposes \cite{vassiliadis2021optimization}.
The focus was either on descent-based nonlinear programming for single-objective optimization \cite{Biegler2010},
derivative-free algorithms for multi-objective optimization \cite{debBook},
stable and efficient control concepts \cite{Diehl2002b},
or on mixed-integer programming, e.g., in superstructure optimization \cite{mencarelli2020review}.
In most parts of the extensive literature on process optimization, the underlying mathematical models have been derived from domain knowledge and consisted of systems of differential equations for transient processes, or of algebraic equations for (assumed) steady states. Shortcomings in the modeling process have been accounted for by methods from uncertainty quantification and robust optimization.
The emergence of powerful data-driven universal approximation models is about to bring a revolution to the field, with many opportunities for research breakthroughs in the intersection of engineering and machine learning \cite{Schweidtmann2021}.
Designing mathematical models and training them has matured over the last decades. Especially for particular machine learning tasks, such as the training of large language models or image processing models, impressive improvements could be achieved. The situation is different, though, in the context of engineering, due to comparatively sparse and expensive-to-observe data and challenging multi-scale modeling tasks with unknown functional relations. 
In addition, engineering simulation and optimization tasks are typically computationally expensive. Surrogate models for process optimization are surveyed with respect to accuracy and computational effort in \cite{Misener2023}. 

Relevant for many processes with partially unknown mechanisms is thus model identification \cite{Brunton2019-wp,camps2023discovering}, the symbolic reconstruction of equations from data.% (different from system identification). 
A number of different frameworks have been considered in the literature, including grammar-based methods, sparse regression, neural networks, Gaussian process regression, Bayesian approaches, and symbolic regression \cite{Ghadami2022-fj,Brunton2019-wp,Dzeroski2008-ja,camps2023discovering,metzcar2024review}. Symbolic regression can, e.g., be based on the sparse identification of nonlinear dynamics (SINDy) algorithm \cite{Brunton2016discovering, mangan2016inferring}, genetic algorithms \cite{Reuter2022}, or mixed-integer optimization \cite{Cozad2018-lb,Kim2023-bt,Bertsimas2023-ge,Austel2017-xt}. Yet, there are many open research questions in model discovery and because of the numerical challenges of inferring symbolic expressions from data, often an intermediate training of a hybrid model is favorable \cite{Martensen2024}.

%%% Thus: UDE
\myheader{Universal differential equations}
We think that combining data-driven surrogate models and differential equations is a very promising approach. The relationship between differential equations and artificial neural networks (ANN) has been studied in several papers. E.g., it is well known that residual ANN can be interpreted as explicit Euler discretizations of ordinary differential equations \cite{Haber2017StableArchitecturesDeep}. Also, several approaches have been suggested how the two concepts can be combined in hybrid models. In addition to structure-based methods, such as physics-informed, Hamiltonian, and Lagrangian neural networks (PINN, HNN, LNN respectively) \cite{Raissi2019}, hybrid models that combine a first-principle model with an embedded machine learning model are of particular interest. A notable example is the framework of universal differential equations (UDE) \cite{Rackauckas2020}. Unlike PINNs, HNNs and LNNs, this approach uses automatic differentiation through the underlying numerical solver of the differential equation, allowing efficient training of the surrogates. ANN are known to approximate continuous functions on compact sets to any desired accuracy (universal approximation theorem) if they are large enough and use discriminatory activation functions \cite{Cybenko1989, Hornik1991}. Consequently, embedding them in differential equation models results in a trainable model capable of learning aspects of the underlying dynamics, assuming sufficient data quality and an appropriate network architecture. Training these universal approximators within the differential equation entails a) differentiating the solution of the differential equation with respect to the neural network parameters, typically achieved via the adjoint method as discussed in \cite{Rackauckas2020}, and b) updating the parameters using this information, e.g., by employing a variant of steepest descent or stochastic gradient methods like \texttt{Adam} \cite{Kingma2014}.
To formally introduce UDE, assume we have a standard ordinary differential equation (ODE) initial value problem on a fixed time horizon $\mathcal{T} = [0, \tf]$ 
\begin{align}\label{dxdt}
	\dot x (t) = f(x(t),u(t),p), \quad x(0) = x_0(p)
\end{align}
with parameters $p \in \mathbb{R}^{n_p}$ and sufficiently smooth right-hand side function $f: \mathbb{R}^{n_x} \times \mathbb R^{n_u} \times \mathbb{R}^{n_p} \mapsto \mathbb{R}^{n_x}$ and controls $u \in \mathcal{U}$. 
Now, the right-hand side function $f$ may be comprised of three parts: First, a function $\hat f : \mathbb{R}^{n_x} \times \mathbb{R}^{n_u} \times \mathbb{R}^{n_{\hat p}} \mapsto \mathbb{R}^{n_{\hat f}}$, whose structure is known a-priori from domain knowledge (e.g., from first-principles) and which may contain parameters $\hat p \in \mathbb{R}^{n_{\hat p}}$, e.g., physical constants. 
Second, a universal approximator model ${\myU: \mathbb{R}^{n_x} \times \mathbb{R}^{n_\theta} \mapsto \mathbb{R}^{n_{U}}}$ with parameters $\theta \in \mathbb{R}^{n_\theta}$ that may capture unknown parts of the model. 
Third, a function $f_S : \mathbb{R}^{n_{\hat f}} \times \mathbb{R}^{n_U} \mapsto \mathbb{R}^{n_x}$ which represents the composition of the known and unknown parts of the UDE. 
With $p = (\hat p, \theta)$, the initial values $x(0) = x_0(\hat p)$ are independent of $\theta$ and the right-hand side is

\begin{align}\label{f_hat}
	f(x(t),u(t),p) = f_S\left(\hat{f}(x(t),u(t),\hat{p}), \myU(x(t), \theta)\right)
\end{align} 
i.e., the function $\hat f$ is complemented by an ANN, e.g., as a replacement of a submodel.
The ANN itself has $J \in \mathbb{N}$ layers. As a reminder, in the simplest case, each of these layers consists of a linear transformation composed with a (typically nonlinear) activation function that is applied elementwise, i.e., 
\begin{align}
	x^{(j)} = s^{(j)}\left( W_j x^{(j-1)} + b_j  \right), \quad j \in [J].
\end{align}
Here, $x^{(0)} = x(t)$ is the input to the feed-forward ANN, $W_j \in \mathbb{R}^{d_j\times d_{j-1}}, b_j \in \mathbb{R}^{d_j}$ are the weights and biases, $s^{(j)}$ the activation function and $d_j \in \mathbb{N}$ the dimension of the $j$-th layer, respectively. We denote with $\theta^{(j)} = \text{vec}(W_j, b_j) \in \mathbb{R}^{d_{j}\cdot (d_{j-1}+1)}$ the weights and biases of the $j$-th layer stacked together as one vector and with  $\theta = \cup_{j \in [J]} \theta^{(j)} \in \mathbb{R}^{n_\theta}$ the weights and biases of all layers. 
The prediction function $\myU$ is given by
\begin{align} \label{eq_Utheta}
	\myU(x(t), \theta) = x^{(J)}.
\end{align}
Note that this formal definition also allows, without loss of generality, to consider several ANNs at once by stacking the different output vectors in a slightly modified definition \eqref{eq_Utheta}.

The parameters $p$ of our model comprise a) the weights and biases $\theta \in \mathbb{R}^{n_\theta}$ of the embedded ANN and b) the model parameters $\hat p$ of the first-principle model $\hat f$.
Note that a generalization towards partial differential algebraic equations on the one hand, and towards more involved ANN topologies on the other hand can be done in a similar way. We restrict ourselves to the case of UDE composed of ODE and feed-forward ANN mainly in the interest of notational simplicity.

UDEs as a hybrid model betweeen differential equations and data-driven universal approximators as suggested in \cite{Rackauckas2020,morGoyB23,Kamp2023} have been applied in areas as diverse as
chemical engineering \cite{Martensen2023a},
glacier ice flow modeling \cite{Bolibar2023a},
modeling of COVID-19 regional transmission \cite{Campos2023},
neuroscience \cite{ElGazzar2024},
climate modeling \cite{Gelbrecht2021,Ramadhan2023},
hydrology \cite{Hoge2022}, and
pharmacology \cite{Valderrama2024}.
Using the general definition~\eqref{f_hat} from above, UDEs contain the important special cases of ODEs (i.e., no embedded ANN) and of neural ODEs \cite{Chen2018b} (i.e., no domain knowledge in $\hat f$) as often used for theoretical studies, e.g., to investigate lifted approaches \cite{Dupont2019AugmentedNeuralODEs} or connections to optimal control theory and the turnpike phenomenon \cite{Ruiz2023}.

%%% Active learning
\myheader{Active learning and optimal experimental design}
%As it is common for universal approximators in supervised-learning, there is the need to collect training data in some way to use it for training purposes in order to achieve a satisfactory accuracy and to be able to subsequently use the model for simulation and optimization purposes or further analysis. The choice where to collect the data has a huge impact on the models predictive capabilities and its accuracy. It is intuitively clear that in regions of the domain where the model behaves highly nonlinear more training samples should be collected than in regions without high rates of change. The problem of determining the right set of training samples is typically called active learning. 
Experimental training data may be sparse and difficult to obtain. Designing information-rich experiments has been studied for the two extreme cases of UDEs mentioned above: the training of ANNs and the parameter estimation in differential equations. In the first case, the machine learning community uses the approaches and terminology of \textit{active learning}, in the second case the statistical approach \textit{optimal experimental design} is often preferred.
Active learning is a large and very active research field. 
While it is intuitively clear that regions of the training domain where the model behaves highly nonlinear should be preferred for selecting training samples, it is challenging to design algorithms that do this automatically.
There are several contributions on how to select training samples for ANN or other surrogate models, see, e.g., \cite{cohn1994improving,settles2009active,Eason2014} for a general introduction to active learning and further references, \cite{ren2021survey} for a survey focussing on deep learning, \cite{prince2004does} for a critical study of the effectiveness of active learning, \cite{fang2017learning} for connections to reinforcement learning, and \cite{di2024active} for connections to Bayesian Optimization.
Also, in the context of model order reduction the concept of active learning has been applied successfully \cite{morBenGW15,morGoyB23,morGoyB24,zhuang2023active}. % morBenGQetal21,morGoyB22a,

%%% Special case: only differential equations (OED)
Learning model parameters for differential equations is traditionally related to the statistical approach of optimal experimental design (OED).
OED is a systematic approach for obtaining useful data for parameter estimation, model identification, and model discrimination. 
The estimated model parameters are random variables and are hence endowed with a confidence region. 
The size of this confidence region depends on how the process was controlled and when and what was measured. 
OED optimizes controls and sampling decisions in the sense of minimizing the size of this confidence region, as discussed in several textbooks \cite{Fedorov1972a,Atkinson1992,Pukelsheim1993,Kitsos2013}.
See, e.g. \cite{Koerkel2002,Sager2013,Schenkendorf2018,Wang2022,ratze2023optimal,bubel2024sequential}, for an extension of the concept to dynamic systems (ODEs) and applications in chemical engineering.

The online version of OED is conceptually closely related to (online) active learning and data assimilation.
In sequential OED, experimental designs are calculated, experimental data is obtained, and model parameters are estimated iteratively \cite{Koerkel1999,Kreutz2009}.
Sequential OED has been extended by online OED in which the experiment is directly re-designed when a new measurement is taken \cite{Stigter2006, Galvanin2009,Barz2013, Qian2014, Lemoine2016}.
OED has also been applied in the context of dual control, in which sequentially optimal control, experimental design, and parameter estimation problems are solved \cite{Jost2017} or exploration versus exploitation is balanced using policy gradients \cite{shen2023bayesian}.
%
%%% Nothing general yet for OED for UDE
While some software solutions for OED of dynamic systems have been published, e.g., \cite{Martensen2024c, Wang2022}, there are no general purpose methods available to treat complex differential equation models. Furthermore, OED for the concurrent estimation of model parameters of an UDE and the weight vectors of embedded submodels is completely open. 
While there are some contributions for the training of ANN using OED \cite{Cohn1996,Crombecq2011,Martens2015}, 
%The idea of using optimizing ANN with approximated Fisher information (natural gradient) has been investigated in \cite{Martens2015}.
%\todo{Cite more? Where has OED been applied in the context of neural networks, or other ML models?}, 
our paper is to our knowledge the first that proposes OED for UDE, especially with the ansatz to formulate and solve OED as a specifically structured mixed-integer optimal control problem, as suggested in \cite{Sager2013}.

%%% Dimension reduction
\myheader{Ill-posed inverse problems and dimension reduction}
Ill-posed OED problems have been discussed in the literature \cite{barz2015nonlinear,xu1998truncated}. A variety of approaches have been proposed to overcome the issue of structural and practical non-identifiability, which corresponds to non-invertible or ill-conditioned matrices. 
Popular are truncated singular value decomposition (SVD) and column subset selection \cite{eswar2024bayesian}, which can be seen as special cases of model order reduction \cite{morBenGQetal21}. They are based on identifying lower-dimensional approximations of relevant matrices.
A particular challenge is the determination of the threshold value for truncating SVD \cite{falini2022review} with interesting suggestions such as the L-curve approach \cite{xu1998truncated}.
We will have a look at truncated SVD in the context of UDEs, where, due to the typical overparametrization of ANN, the ill-posedness is not the exception, but the typical situation.

%\cite{thibeault2024low} The low-rank hypothesis of complex systems

%%% Theoretical advances
\myheader{Information gain}
There are two different types of decisions that influence the impact of observations on the uncertainty of a follow-up parameter estimation or ANN training problem. First, experimental degrees of freedom that specify the setup of the experiment. We shall refer to these degrees of freedom as \textit{controls}. Second, the decision what and when to measure. We shall refer to these decisions as \textit{sampling}. Measuring as much as possible provides a maximum of information about unknown parameters. But typically measurements or data processing are costly. Therefore, the maximum amount of measurements is limited or measurements are penalized in the objective function via associated costs.
In previous work, via studying the necessary conditions of optimality of OED problems with respect to sampling decisions, the concept of \textit{global information gain} was introduced \cite{Sager2013}. It allows an a posteriori analysis and interpretation of information-rich time intervals.
So far, it has never been studied in the context of dimension reduction approaches.

%%% Contributions and Outline
\myheader{Contributions and outline}
We see several contributions to the literature, to our knowledge all presented for the first time.
First, we provide a self-contained, compact summary of fundamental approaches to OED and UDE, respectively, and the formal definition of a fusion of these concepts.
Second, we describe algorithmic implementations of OED for UDE and the obtained results for several benchmark problems, focussing among others on the interplay of mechanistic model parameters and ANN weights, on the individual impact of controls and sampling on a posteriori uncertainty, or on the question how the data can be used for the training of UDE.
Third, we propose a new approach to dimension reduction for UDE models that is based on lumping the sensitivities of differential states with respect to ANN weights and biases. This new approach does not require costly singular value decompositions, sometimes performs competetively, has a transparent interpretation, and might be exploited in the future in specific settings such as the training of sparse ANN.
Fourth, we propose to use the global information gain function as a new criterion to select the truncation hyperparameter in truncated SVD. This contribution is independent of UDE and applicable to ill-posed OED in general.

The paper is organized as follows. 
We start with some definitions and the formulation of the OED problem for ODE as an optimal control problem in Section~\ref{sec:oed}. 
In Section~\ref{sec:methodology} we apply this setting to the case of UDE and discuss ways to reduce the computational complexity. 
Numerical results for benchmark problems are presented in Section~\ref{sec:results} that illustrate the advantages of this approach when compared to equidistant sampling of longitudinal data.
We conclude with a discussion in Section~\ref{sec:discussion}.

\section{Optimal experimental design as an optimal control problem}
\label{sec:oed}

For models described by ODEs, the OED problem can be formulated as a specifically structured mixed-integer optimal control problem (MIOCP) \cite{Sager2013}. 
A variety of approaches have been suggested for OED, e.g., formulations based on
the variance-covariance matrix of the linearized model or robust formulations \cite{Koerkel2004},
sigma points or other approaches allowing to calculate approximations of global sensitivities \cite{Schenkendorf2018},
polynomial chaos \cite{streif2014optimal},
Gaussian process surrogates \cite{olofsson2018design},
adaptive discretizations \cite{schmid2024adaptive},
nonlinear preconditioners \cite{mommer2015nonlinear},
transport maps \cite{koval2024tractable},
or Bayesian optimization \cite{greenhill2020bayesian}.
While we think that the proposed methodology for UDE of the following section can be applied to most of these approaches in a straightforward way, we focus here on the linearized variant of OED.

In this approach, we add the so-called sensitivities $G = \mathrm{d}x / \mathrm{d} p$ of the states $x$ with respect to the model parameters $p$ to the system \eqref{dxdt} of differential equations.
The variational differential equations (VDE) for $G = \mathrm{d}x / \mathrm{d} p: [0, \tf] \mapsto \R^{n_x \times n_p}$ can be derived by taking derivatives of the system solution 
$$x(t) = x_0 + \displaystyle \int_{0}^{t} f(x(\tau), u(\tau), p) \; \mathrm{d} \tau$$ 
with respect to time $t$ and parameters $p$, resulting in the matrix-valued ODE initial value problem
\begin{align} \label{eq_VDE}
\dot{G}(t) & = f_x(x(t), u(t), p)G(t) + f_p(x(t),u(t),p), \quad \quad G(0) = \frac{\partial x_0}{\partial p}.
\end{align}
Here, $f_x$ and $f_p$ denote the partial derivatives of $f$ with respect to $x$ and $p$. 
Note that usually only a subset of all parameters is included here, while others are considered to be known and fixed. Without loss of generality, such fixed constants can be considered as part of the function $f$.
Based on the sensitivities $G$, also the FIM $F(\tf)$ can be modeled in the control problem.
The inverse $F(\tf)^{-1}$ is the variance-covariance matrix of the linearized least squares parameter estimation problem.
The goal of OED is to reduce the volume of the ellipsoid associated with the variance-covariance matrix. Since matrices can not be totally ordered, the experimental design community defined a variety of scalar objective functions $\phi: \R^{n_p \times n_p} \mapsto \R$. 
Typical examples are 
the trace (A-criterion and weighted sum of half-axes), $\phi_A(F(\tf)) = \frac{1}{n_p} \tr(F(\tf)^{-1})$,
the determinant (D-criterion and proxy for the volume), $\phi_D(F(\tf)) = \det(F(\tf)^{-1})^{\frac{1}{n_p}}$, and
the maximal eigenvalue (E-criterion and maximal half-axis), $\phi_E(F(\tf)) = \max \{ \lambda : \lambda \text{ is eigenvalue of } F(\tf)^{-1} \}$.
The degrees of freedom of the OED problem involve initial values $x_0$, control functions $u$ to excite the system, and sampling functions (when and what to measure) $w$.
The OED problem for the ODE \eqref{dxdt} formulated as a MIOCP is

\begin{align}
	\begin{array}{clcll}
		\displaystyle \min_{x, x_0, G, F, z, w, u} & \phi(F(\tf)) & \\[1.5ex]
		\mbox{s.t.} & \dot{x}(t) & = & f(x(t), u(t), p) &\\
		& \dot{G}(t) & = & f_x(x(t), u(t), p)G(t) + f_p(x(t),u(t),p), \\
		& \dot{F}(t) & = & \sum_{i=1}^{n_y} w_i(t) (h^i_x(x(t)) G(t))^\top (h^i_x(x(t)) G(t)), \\
		& \dot{z}(t) & = & w(t),\\
		& x(0) & = & x_0, \quad G(0) = \frac{\partial x_0}{\partial p}, \quad F(0) = 0, \quad z(0) = 0,\\
		& u(t) & \in & \mathcal{U},\\
		& w(t) & \in & \mathcal{W},\\
		& z_i(\tf)   & \leq & M_i,
	\end{array}
	\tag{OED}
	\label{OED}
\end{align}
with (differential) states $x\left(t\right) : \mathcal{T} \mapsto \mathbb{R}^{n_x}$ and their initial condition $x_0 \in \mathbb{R}^{n_x}$, model parameters $p \in \mathbb{R}^{n_p}$, time $t \in \mathcal{T}$, and controls $u\left(t\right) : \mathcal{T} \mapsto \mathbb{R}^{n_u}$. 
In addition to \eqref{dxdt} we also require an observed function $y = h(x)$ with $h : \mathbb{R}^{n_x} \mapsto \mathbb{R}^{n_y}$ and $y: \mathcal T \mapsto \mathbb{R}^{n_y}$.
The first and second constraint in \eqref{OED} denote the time evolution of the dynamical system and of the sensitivities $G : \mathcal{T} \mapsto \mathbb{R}^{n_x \times n_p}$ of $x$
with respect to $p$, respectively. 
The evolution of the symmetric matrix $F : \mathcal{T} \mapsto \mathbb{R}^{n_p \times n_p}$ is given by the weighted sum of observability Gramians $h^i_x\left(x(t)\right) G(t),~ i=1, \dots, n_y$ for each observed function of states. The weights $w_i\left(t\right) \in \{0, 1\}, ~ i = 1, \dots, n_y$ are the (binary) sampling decisions, 
where $w_i(t) = 1$ denotes the decision to perform a measurement at time $t$. The objective $\phi(F(\tf))$ of Mayer type is a suitable OED criterion as discussed above. 
In this formulation, upper bounds $M_i$ are provided as maximum amounts (in continuous time) of measurements. 
A detailed discussion of discrete-time and continuous-time measurements in OED, alternative formulations penalizing measurements in the objective function (with a hyperparameter that models the trade-off between information gain and experimental costs of measurements), and a detailed analysis of solution structures is given in \cite{Sager2013}.

The special features of \eqref{OED} with respect to standard optimal control problems, such as the non-dependency of some differential states ($x$) on others ($G$ and $F$), the inverse in the objective function, the integrality of sampling functions, or the positive-semidefiniteness of the right hand-side of $\dot{F}$, can and should be exploited in efficient algorithms and software \cite{Koerkel2004,Sager2013,Wang2022,Martensen2024c}.

A pretty straightforward result will become useful later, when the dimension of sensitivities shall be reduced.
\begin{lemma} \label{lem:sensitivitiesA}
Let a well-posed VDE \eqref{eq_VDE} with $(x,u,p)$ be given and let $G(\cdot)$ be a solution.
Let $A \in \R^{n_p \times n_A}$ be a matrix. Then the solution $G_s: [0, \tf] \mapsto \R^{n_x \times n_A}$ of the VDE
\begin{align} \label{eq_VDEA}
\dot{G_A}(t) & = f_x(x(t), u(t), p) G_A(t) + f_p(x(t),u(t),p) \; A, \quad \quad G_A(0) = \frac{\partial x_0}{\partial p} A
\end{align}
exists uniquely and is given by
\begin{align}
  G_A(t) &= G(t) \; A. \label{eq_GA}
\end{align}
\end{lemma}
\begin{proof}
Unique existence follows from the Picard-Lindelöf theorem, multiplication with $A$ does not affect the Lipschitz continuity of $f_x$ and $f_p$ implied by the assumed well-posedness.
That $G_A(t) = G(t) \; A$ is a solution of \eqref{eq_VDEA} can be seen by multiplying \eqref{eq_VDE} with $A$ from the right hand side and using $\dot{G_A} = \dot{(G A)} = \dot{G} A$.
\end{proof}

\section{The extension towards universal differential equations}
\label{sec:methodology}

%For reasons of brevity we will drop the explicit time dependency from here on. 

We consider OED problems of type \eqref{OED} for UDE of type \eqref{f_hat}.
The combined parameter estimation and training problem involves model parameters $\hat{p} \in \mathbb{R}^{n_{\hat{p}}}$ and weights and biases $\theta \in \mathbb{R}^{n_\theta}$ of the embedded ANN.
Qualitatively, they can both be captured in one large vector $p \in \R^{n_p}$ without any changes to the setting of \eqref{OED}.
Quantitatively though, the number $n_\theta$ of weights $\theta$ is typically very large in comparison to the number of model parameters $\hat p$. 
This poses two main issues: 
First, also the number of variational differential equations for the sensitivities $G(\cdot)$ is very large, resulting in a drastic increase of computation time.
Second, ANN are often overparametrized. This structural and practical non-identifiability of the weights leads to singular Fisher information matrices (FIM), violating a major assumption of OED contrasting it to most methods from active learning which are typically based on some kind of space-filling sampling of the training domain. 
We are interested in experimental designs for the estimation of both kinds of parameters, $\hat p$ and $\theta$.
Therefore, the challenge is to find ways to reduce the dimension of $G$ and hence of the FIM $F(\tf)$ such that the approximation of the FIM is invertible, accurate, and comparatively fast-to-evaluate.
%
%As the treatment of model parameters $\hat p$  in the VDE \eqref{eq_VDE} remains unchanged in our setting, we focus on the treatment of weights and biases $\theta$ and the computation and approximation of their sensitivities, which will be referred to as $G_c(\cdot) \in \mathbb R^{n_x \times {n_\theta}}$  in the following. The (usually ill-posed) problem \eqref{OED} of computing designs to accurately estimate $\theta$ will be denoted as (OED-NN-c).

% -------------
\subsection{Dimension reduction methods}

In OED optimization problems, all free parameters $p$ are fixed to an initial guess $\bar{p}$ that serves as a linearization point of the nonlinear model. The follow-up parameter estimation of $p$ with new experimental data is only accounted for by taking sensitivities with respect to $p$ into account via the VDE \eqref{eq_VDE}. 
This implies that the general approach is still valid, if model evaluations based on $\bar p$ are combined with a reduced number or approximations of sensitivities.
We propose and investigate several approaches to do this and hence overcome the issue of overfitting in the context of OED. To be able to distinguish different settings and approaches conveniently in the following, we introduce the compact notation $w-u-a$ with options
\begin{align} 
w & \in \left\{ w^*, w0 \right\}, \label{wua_w} \\
u & \in \left\{ u^*, u0 \right\}, \label{wua_u} \\
a & \in \left\{ c, \svdu{n_s}, \psvdu{n_s}, o, l, ll, \tsvdu{n_s} \right\}. \label{wua_a} 
\end{align}
The $w$ refers to the sampling function in \eqref{OED}. For $w^*$, the sampling decisions when to measure have been optimized, for $w0$ an equidistant grid has been used to collect measurements.
The $u$ refers to the control functions in \eqref{OED}. For $u^*$, they have been degrees of freedom, for $u0$ the particular choice of $u(t) = 0 \; \forall \; t \in \mathcal{T}$ was evaluated.
The $a$ specifies either the $a=c$ complete (and usually ill-posed) use of all sensitivities for all $n_\theta$ weights and biases or a particular dimension reduction approach. These options for $a$ are explained in more detail in the remainder of this subsection.
An example notation is $w^*$-$u0$-$\svdu{3}$, an experimental design that uses optimized samplings, no controls, and a singular value decomposition with 3 singular values. The approach $a$ will also be used as a subscript for sensitivities $G$ and FIM $F$.

\myheader{Singular value decomposition ($a = \svdu{n_s}$ and $a = \psvdu{n_s}$)}
Motivated by reduced-order modeling approaches \cite{morBenGQetal21} and previous application in OED \cite{barz2015nonlinear}, we propose to use the truncated SVD of the FIM to select the most relevant weights and biases. The basic idea is to compute the SVD of the FIM using the VDE \eqref{eq_VDE} for $G_c =\mathrm{d}x / \mathrm{d} \theta: [0, \tf] \mapsto \R^{n_x \times n_\theta}$ and then truncate the SVD with the $n_s$ largest singular values. The relevant weights and biases can be determined by identifying the nonzero elements in the corresponding singular vectors. Only these enter the VDE in \eqref{OED}, leading to a significant complexity reduction for $n_s \ll n_p$ with sensitivities $G_{\svdu{n_s}}: \mathcal{T} \mapsto R^{n_x \times n_s}$ and $F_{\svdu{n_s}}(\tf) \in \mathbb{R}^{n_s \times n_s}$.

We investigate two cases: one in which all free model parameters, weights, and biases are used for the truncated SVD, and one in which the model parameters $\hat{p}$ are kept in addition to the SVD of the block matrix for all ANN weights and biases. We refer to the first approach as $a = \svdu{n_s}$, the second one as $a = \psvdu{n_s}$.

\myheader{Outer parameter for the ANN ($a = o$)}
We introduce a dummy parameter $p_{o}=1.0$ to which the ANN is formally multiplied. We consider
\begin{align}\label{f_hato}
	f(x(t),u(t), (p, p_o)) = f_S\left(\hat{f}(x(t),u(t),\hat{p}), p_o \cdot \myU(x(t), \theta)\right)
\end{align} 
which is formally equivalent to \eqref{f_hat}, but allows to evaluate sensitivities with respect to the additional parameter $p_o$.
Thus, $p_o$ can be interpreted as a proxy for capturing the overall influence of the ANN, similar to a model parameter multiplied to a known and fixed algebraic term in the right hand side. While the sensitivities with respect to this artificial parameter are different in dimension and value from the sensitivities with respect to the weights of the ANN, resulting experimental designs might still be a good choice.
As a major advantage, for one embedded ANN this approach reduces the dimension of the variational differential equations $G_o : \mathcal{T} \mapsto R^{n_x \times 1}$ and of the FIM $F_o(\tf) \in \mathbb{R}^{1\times 1}$ drastically.

\myheader{Lumping of all weights ($a=l$)}
Similar to $a = o$, we introduce a parameter $p_{l}=1.0$ for the embedded ANN. Now, however, the parameter is multiplied ``inside'' the ANN to all ${n_\theta}$ weights and biases, i.e., 
\begin{align}\label{f_hatl}
	f(x(t),u(t), (p, p_l)) = f_S\left(\hat{f}(x(t),u(t),\hat{p}), \myU(x(t), \theta_l)\right)
\end{align} 
for
		\begin{align} \label{eq_thetal}
			\theta_l &= \theta \cdot p_l =: \Xi_l \cdot p_l
		\end{align}
with $\Xi_l := \theta \in \R^{n_\theta}$. Thus, $p_l$ can be interpreted as a proxy for capturing the overall influence of all weights and biases as it is propagated through the network $\myU$ during inference. The dimension reduction is similar to the approach $a=o$ above with $G_l : \mathcal{T} \mapsto R^{n_x \times 1}$ and $F_l(\tf) \in \mathbb{R}^{1\times 1}$.

\myheader{Layerwise lumping of weights ($a=ll$)}
Taking this idea a step further, we introduce scaling parameters for subsets of weights and biases. A natural heuristic choice for such subsets are the weights and biases of each layer $j \in [J]$. Formally this can be done by introducing a vector of parameters $p_{ll} = [1.0,\dots,1.0] \in \mathbb R^J$ with elements that are formally multiplied to the weights and biases of the individual layers of the ANN,
		\begin{align} \label{eq_thetall}
			\theta_{ll} &= \left[\theta \odot \mathbbm{1}_{\theta_i \in  \theta^{(1)} \forall i \in [{n_\theta}]}, \dots, \theta \odot \mathbbm{1}_{ \theta_i \in \theta^{(J)} \forall i \in [{n_\theta}]}  \right] \cdot p_{ll} =: \Xi_{ll} \cdot p_{ll},
		\end{align}
		where $\odot$ denotes the Hadamard product of element-wise multiplication and $\Xi_{ll} \in \mathbb R^{{n_\theta} \times J}$. The entries of $\Xi_{ll}$ are hence given by
		\begin{align*}
			\Xi_{ll,i,j}=\begin{cases}
				\theta_i, & \text{if $\theta_i$ belongs to layer $j$, for $i \in [{n_\theta}], j \in [J]$}\\
				0, & \text{otherwise}
			 \end{cases}		
		\end{align*}
This approach yields sensitivities $G_{ll}(\cdot) \in \mathbb{R}^{n_x \times J}$ and a FIM $F_{ll}(\tf) \in \mathbb{R}^{J \times J}$.

\myheader{Weighted singular value decomposition ($a = \tsvdu{n_s}$)}
The ideas from above can also be combined. As an alternative to $a=ll$ we identify subsets of weights for artificial scaling via a truncated SVD, i.e., by choosing the indices of positive entries in a column of the decomposition matrix. As in other SVD approaches, the dimension is reduced to $n_s$.

Note that also a combination of $a = \tsvdu{n_s}$ and $a = \psvdu{n_s}$ is possible. We do not discuss it further, because the numerical results were almost identical to $a = \psvdu{n_s}$.

\myheader{Illustration for prototypical example}
Figure~\ref{fig:lem1} illustrates the sensitivities $G_c, G_o, G_l, G_{ll}, G_{\svdu{4}}$, and $G_{\tsvdu{4}}$ for an ANN $\myU(x(t), \theta)$ embedded in the Lotka Volterra equations 
$x_1(t) =  x_1(t) - \myU(x(t), \theta) - \hat p_2 u(t) x_1(t)$ and $\dot x_2(t) = -x_2(t) + \myU(x(t), \theta) - \hat p_4 u(t) x_2(t)$, see also Section~\ref{sec_lotka}.
The ANN $\myU(x(t), \theta)$ for the interaction term in $n_x = 2$ differential states has $J=4$ layers, $n_\theta = 81$ weights, and different nonlinear activation functions.
Shown are sensitivities for the $n_s=4$ largest singular values.
The sensitivities thus map $\mathcal T$ to $\R^{2\times81}, \R^{2 \times 1}, \R^{2 \times 1},\R^{2 \times 4}$, $\R^{2 \times 4}$, and $\R^{2 \times 4}$, respectively. 
The different dimensions indicate the potential of dimension reduction in terms of computational speedup.
\begin{figure}[h!!!]
	\centering
	\includegraphics[width=.8\linewidth]{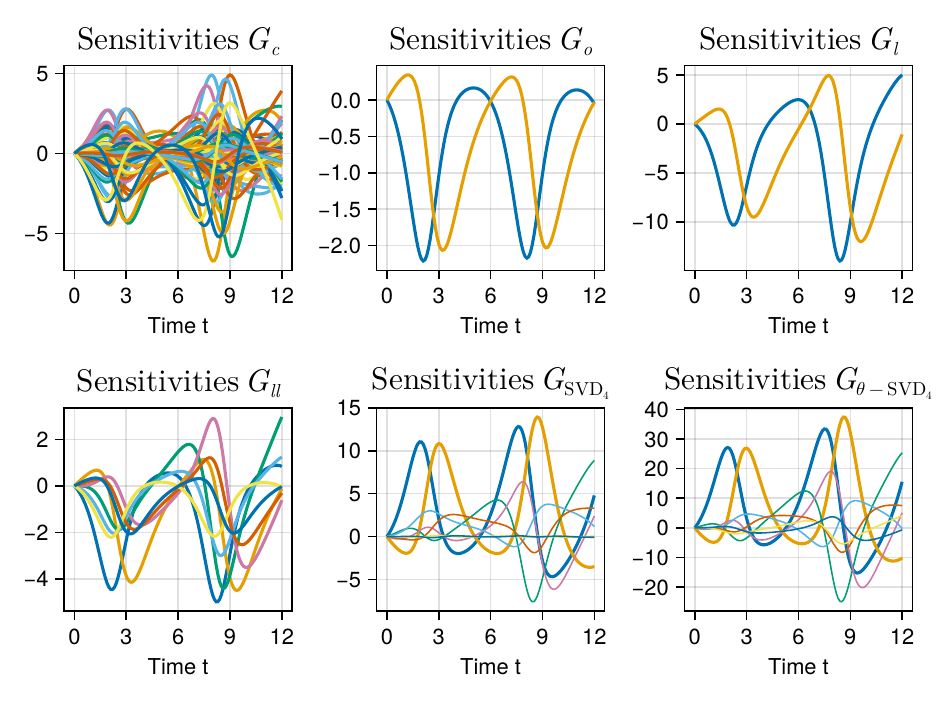}
	\caption{Illustration of sensitivities for ANN embedded in the Lotka-Volterra equations. Although they have different dimensions, the time intervals of interest (those with high absolute values) are similar.}
	\label{fig:lem1}
\end{figure}

\begin{figure}[t]
	\centering
	\includegraphics[width=.72\linewidth]{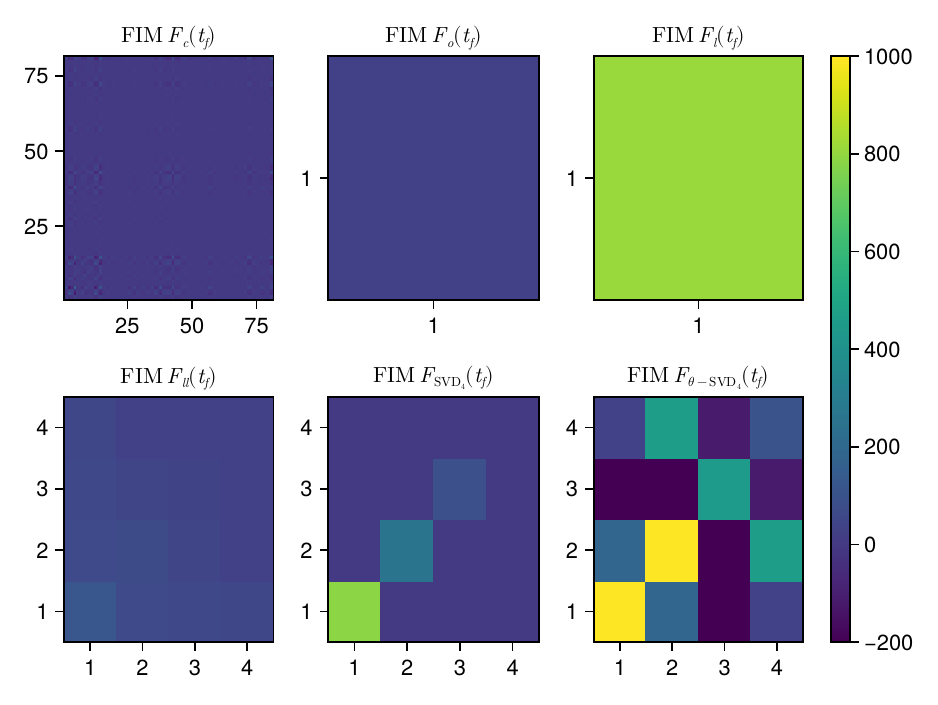}
	\caption{Comparison of FIM $F(\tf)$ corresponding to the sensitities from Figure~\ref{fig:lem1}. The FIM $F_c(\tf)$ for all $2 n_\theta$ sensitivities is indefinite, while all dimension reduction approaches lead to positive definite FIM.}
	\label{fig:F_tf_indefinite}
\end{figure}
Figure \ref{fig:F_tf_indefinite} shows the corresponding FIM. Here one observes the potential of dimension reduction in terms of regularity of the FIM. 

% -------------
\subsection{Calculation of sensitivities and Fisher information matrix}

After the conceptual introduction of the different approaches to reducing the dimension, we have a closer look at how the sensitivies can be calculated and how they are related.
%For notational simplicity, we consider the case without model parameters, i.e., $n_{\hat p} = 0$. 
For the block matrix sensitivities $G(t) = \left[G_{\hat p}(t) \; G_{c}(t)\right]$ of the differential states $x$ with respect to all parameters $p = (\hat p, \theta)$ application of the general VDE \eqref{eq_VDE} and of the particular UDE structure \eqref{f_hat} yields
\begin{align} 
 \frac{\mathrm{d}}{\mathrm{d} t} \left[G_{\hat p}(t) \; G_{c}(t)\right] &= \frac{\partial f}{\partial x}(x(t),u(t),p) \; \left[G_{\hat p}(t) \; G_{c}(t)\right] + \frac{\partial f}{\partial p}(x(t),u(t),p) \label{eq_Gca} \\ 
%			   &= \left(\frac{\partial f_S}{\partial \hat f} \frac{\partial \hat f}{\partial x} + \frac{\partial f_S}{\partial \myU} \frac{\partial \myU}{\partial x} \right) \; \left[G_{\hat p}(t) \; G_{c}(t)\right] + \left[\frac{\partial f_S}{\partial \hat p} ~ \frac{\partial f_S}{\partial \theta} \right] \\ 
			   &= \left(\frac{\partial f_S}{\partial \hat f} \frac{\partial \hat f}{\partial x} + \frac{\partial f_S}{\partial \myU} \frac{\partial \myU}{\partial x} \right) \; \left[G_{\hat p}(t) \; G_{c}(t)\right] + \left[\frac{\partial f_S}{\partial \hat f}\frac{\partial \hat f}{\partial \hat p} ~ \frac{\partial f_S}{\partial \myU}\frac{\partial \myU}{\partial \theta} \right] \label{eq_Gcb}
\end{align}
with initial values $G_{\hat p}(0) = \mathrm{d} x(0) / \mathrm{d} {\hat p} \in \mathbb R^{n_x \times n_{\hat p}}$ and $G_c(0) = 0 \in \mathbb R^{n_x \times n_\theta}$. Note that we omitted arguments in \eqref{eq_Gcb} for notational simplicity.
Formulation \eqref{eq_Gcb} generalizes pure first-principle models with $\partial_{\myU} f_S = 0, \partial_{\hat f} f_S = 1$ and neural ordinary differential equations with $\partial_{\myU} f_S = 1, \partial_{\hat f} f_S = 0$. 

Integrating the third differential equation in \eqref{OED} for the particular case of \eqref{f_hat} gives the FIM
\begin{align}
F(\tf) &= 
  \begin{pmatrix}
  F_{\hat p \hat p}(\tf) & F_{\hat p c}(\tf) \\
  F_{\hat p c}^T(\tf) & F_{c c}(\tf) 
  \end{pmatrix} := \\
 & 
  \begin{pmatrix}
     \displaystyle \int_0^{\tf} \sum_{i=1}^{n_y} w_i(t) 
     G_{\hat p}^T(t) h^{iT}_x(x(t)) h^i_x(x(t)) G_{\hat p}(t) \; \mathrm{d}t
  & 
     \displaystyle \int_0^{\tf} \sum_{i=1}^{n_y} w_i(t) 
     G_{\hat p}^T(t) h^{iT}_x(x(t)) h^i_x(x(t)) G_{c}(t) \; \mathrm{d}t
  \\
     \displaystyle \int_0^{\tf} \sum_{i=1}^{n_y} w_i(t) 
     G_{c}^T(t) h^{iT}_x(x(t)) h^i_x(x(t)) G_{\hat p}(t) \; \mathrm{d}t
  & 
     \displaystyle \int_0^{\tf} \sum_{i=1}^{n_y} w_i(t) 
     G_{c}^T(t) h^{iT}_x(x(t)) h^i_x(x(t)) G_{c}(t) \; \mathrm{d}t
  \end{pmatrix} \label{eq_Fhybrid}
\end{align}

Using the notation from above, we obtain the following result for the weight lumping approaches.
\begin{lemma}
	\label{lem:sensitivitiesLumped}
Let $\myU : \mathbb R^{n_x} \times \mathbb{R}^{n_\theta} \mapsto \mathbb R^{n_{\hat{x}}}$ be an ANN with weights $\theta \in \mathbb R^{n_\theta}$. 
Let $\Xi_l$ and $\Xi_{ll}$ be defined as in \eqref{eq_thetal} and \eqref{eq_thetall}, respectively, and
$\mathbbm{1}_J$ be the all ones vector of dimension $J$.
Then for $t$ almost everywhere in $\mathcal T$
	\begin{align}
		G_l(t) &= G_c(t) \cdot \Xi_l, \label{eq_lem1a} \\
		G_{ll}(t) &= G_c(t) \cdot \Xi_{ll},\label{eq_lem1b} \\
		G_{l}(t) &= G_{ll}(t) \cdot \mathbbm{1}_{J}. \label{eq_lem1c}
	\end{align}
\end{lemma}

\begin{proof}
We start by looking at the VDE \eqref{eq_Gcb} derived above.
 We observe that $f_x$, $\frac{\partial f_S}{\partial \hat f}\frac{\partial \hat f}{\partial \hat p}$ and $\frac{\partial f_S}{\partial \myU}$ are independent of the particular approach $a=c$, $a=l$, and $a=ll$. 
%Therefore, if we can show the properties for $\frac{\partial \myU}{\partial \theta}$, they will carry over via linearity of $G(t)$ to the right hand side of \eqref{eq_Gcb}.
Application of the chain rule for the setting $p_{l} = 1.0$ and $\theta_l = \theta \cdot p_{l}$ yields
\begin{align*}
	\frac{\partial \myU}{\partial p_{l}}(x(t),\theta_l) 
	 &= \frac{\partial \myU}{\partial \theta}(x(t),\theta_l) \cdot  \frac{\partial \theta_l}{\partial p_{l}} \\
	 &= \frac{\partial \myU}{\partial \theta}(x(t),\theta) \cdot \theta \\
	 &= \frac{\partial \myU}{\partial \theta}(x(t),\theta) \cdot \Xi_l.
\end{align*}
Therefore and with the trivial property $G_l(0) = 0 = G_c(0) \cdot \Xi_l$, we can apply Lemma~\ref{lem:sensitivitiesA} with the choice $A = \Xi_l = \theta \in \mathbb R^{n_\theta \times 1}$ and obtain equation \eqref{eq_lem1a}.
%Inserting $\frac{\partial \myU}{\partial p_{l}} = \frac{\partial \myU}{\partial \theta}(x(t),\theta) \Xi_l \in \mathbb R^{n_x \times 1}$ for $\frac{\partial \myU}{\partial \theta}$ in \eqref{eq_Gcb} and the trivial property $G_l(0) = 0 = G_c(0) \cdot \Xi_l$ allow to conclude \eqref{eq_lem1a} from the existence and uniqueness of solutions via the Picard-Lindelöf theorem.

Equation \eqref{eq_lem1b} follows similarly for the setting $p_{ll} = [1.0,\dots,1.0] \in \mathbb R^J$ and $\theta_{ll} = \Xi_{ll} \cdot p_{ll}$ via
\begin{align*}
	\frac{\partial \myU}{\partial p_{ll}}(x(t),\theta_{ll}) 
	 &= \frac{\partial \myU}{\partial \theta}(x(t),\theta_{ll}) \cdot \frac{\partial \theta_{ll}}{\partial p_{ll}} = \frac{\partial \myU}{\partial \theta}(x(t),\theta) \cdot \Xi_{ll}
\end{align*}
with $\Xi_{ll} \in \mathbb R^{n_x \times J}$.
Equation \eqref{eq_lem1c} follows from $\Xi_l = \Xi_{ll} \mathbbm{1}_{J}$.
\end{proof}
The relations between the corresponding FIM follow directly.
\begin{corollary} \label{cor:FisherLumped}
With the notation from above for approach $a \in \{c, l, ll \}$ we have
\begin{align}
F_l(\tf) &= 
 \begin{pmatrix}
 I & 0 \\
 0 & \Xi_l
 \end{pmatrix}^T 
 \; F(\tf) \;
 \begin{pmatrix}
 I & 0 \\
 0 & \Xi_l
 \end{pmatrix}
\label{cor_Fisherl} \\
F_{ll}(\tf) &= 
 \begin{pmatrix}
 I & 0 \\
 0 & \Xi_{ll}
 \end{pmatrix}^T 
 \; F(\tf) \;
 \begin{pmatrix}
 I & 0 \\
 0 & \Xi_{ll}
 \end{pmatrix}
\label{cor_Fisherll}
\end{align}
\end{corollary}
\begin{proof}
Follows from multiplication using $F(\tf)$ from \eqref{eq_Fhybrid} as
\begin{align*}
 \begin{pmatrix}
 I & 0 \\
 0 & \Xi_l
 \end{pmatrix}^T 
 \; F(\tf) \;
 \begin{pmatrix}
 I & 0 \\
 0 & \Xi_l
 \end{pmatrix}
&=
 \begin{pmatrix}
 I & 0 \\
 0 & \Xi_l
 \end{pmatrix}^T 
 \; \begin{pmatrix}
  F_{\hat p \hat p}(\tf) & F_{\hat p c}(\tf) \\
  F_{\hat p c}^T(\tf) & F_{c c}(\tf)
 \end{pmatrix} \;
 \begin{pmatrix}
 I & 0 \\
 0 & \Xi_l
 \end{pmatrix}
 =
 \begin{pmatrix}
  F_{\hat p \hat p}(\tf) & F_{\hat p c}(\tf) \Xi_l \\
  \Xi_l^T F_{\hat p c}^T(\tf) & \Xi_l^T F_{c c}(\tf) \Xi_l
 \end{pmatrix}
\end{align*}
 and comparing the result to inserting (\ref{eq_lem1a}) in the definition of the FIM. The argument is identical for \eqref{cor_Fisherll}.
\end{proof}

% -----------------------
We now consider approximations of sensitivities and FIM in the SVD approach.
\begin{lemma} \label{lem:sensitivitiesSVD}
Let $\myU : \mathbb R^{n_x} \times \mathbb{R}^{n_\theta} \mapsto \mathbb R^{n_{\hat{x}}}$ be an ANN with weights $\theta \in \mathbb R^{n_\theta}$. 
Then there exists for any number $1 \le n_s \le n_\theta$ of positive singular values of $F_{c c}(\tf)$ a matrix $V \in \mathbb R^{n_\theta \times n_s}$ such that 
\begin{align}
F_\svdu{n_s}(\tf) &= 
 \begin{pmatrix}
 I & 0 \\
 0 & V
 \end{pmatrix}^T 
 \; F(\tf) \;
 \begin{pmatrix}
 I & 0 \\
 0 & V
 \end{pmatrix}
= 
 \begin{pmatrix}
 F_{\hat p \hat p}(\tf) & F_{\hat p c}(\tf) V\\
 V^T F_{\hat p c}^T(\tf) & D
 \end{pmatrix}
 \label{eq_FisherSVD} 
\end{align}
where $D = V^T F_{c c}(\tf) V \in \R^{n_s \times n_s}$ is a diagonal, positive definite matrix with eigenvalues of $F_{c c}(\tf)$ on the main diagonal.
Additionaly, we have for $t$ almost everywhere in $\mathcal T$
	\begin{align}
		G_\svdu{n_s}(t) &= G_{c}(t) \cdot V. \label{eq_sensitivitySVD} 
	\end{align}
\end{lemma}
\begin{proof}
We consider the decomposition of the lower right submatrix $F_{c c}(\tf)$ of $F(\tf)$
\begin{align}
	F_{c c}(\tf) &= \int_{0}^{\tf} \sum_{i=1}^{n_y} (h_x^i(x(t))G_c(t))^T (h_x^i(x(t))G_c(t)) \; \mathrm{d} t = \hat U \hat D {\hat V}^T
	\label{eq:F_svd}
\end{align}
with $\hat U, \hat V \in \mathbb R^{n_\theta \times n_\theta}$ unitary matrices and $\hat D \in \mathbb R^{n_\theta \times n_\theta}$ diagonal containing the singular values of $F_{c c}(\tf)$.  
As $F_{c c}(\tf)$ is symmetric and positive semidefinite by construction, $\hat U = \hat V$ can be chosen and singular values and eigenvalues of $F_{c c}(\tf)$ coincide \cite[Corollary 2.5.11]{horn2012matrix}. We obtain
\begin{align}
			\begin{split}
			\hat D 	&= {\hat V}^T F_{cc}(\tf) \hat V \\
			  	&= {\hat V}^T \sum_{i=1}^{n_y} \int_{0}^{\tf} (h_x^i(x(t)) G(t))^T (h_x^i(x(t))  G(t))  \; \mathrm{d} t \; \hat V \\
				&= \sum_{i=1}^{n_y} \int_{0}^{\tf} (h_x^i(x(t))  G(t) \hat V)^T (h_x^i(x(t)) G(t) \hat V) \; \mathrm{d} t
			\end{split}
			\label{eq:Dhat}
\end{align}
		Truncating the singular value decomposition \eqref{eq:F_svd} and replacing $\hat V$ in \eqref{eq:Dhat} by $V \in \mathbb R^{n_\theta \times n_s}$, where $V$ is given by the columns of $\hat V$ corresponding to the $n_s$ largest singular values of $F_{cc}(\tf)$, we can obtain a lower-dimensional, diagonal, and positive definite approximation $D \in \mathbb{R}^{n_s \times n_s}$ of $F_{cc}(\tf)$:
		\begin{align}
			\begin{split}
			 D &= \sum_{i=1}^{n_y} \int_{0}^{\tf} (h_x^i(x(t)) G_c(t) V)^T (h_x^i(x(t)) G_c(t) V) \; \mathrm{d} t
			\end{split}
			\label{eq:D}
		\end{align}
		
The term $G_c(t) V$ can be interpreted as reduced sensitivities $G_\svdu{n_s}(t) \in \mathbb{R}^{n_x \times n_s}$. Property \eqref{eq_sensitivitySVD} follows from Lemma~\ref{lem:sensitivitiesA}, similarly to the proof argument in Lemma~\ref{lem:sensitivitiesLumped}.
\end{proof}
The approximation of $F_{cc}(\tf)$ with a diagonal matrix containing the $n_s$ largest eigenvalues is particularly interesting, because the A-, D-, and E-criteria for the objective $\phi(F(\tf))$ of \eqref{OED} mentioned in Section~\ref{sec:oed} can be formulated in terms of eigenvalues. This follows from basic linear algebra results, such as the trace being the sum of all eigenvalues, the determinant being the product of all eigenvalues, and the eigenvalues of an inverse matrix $A^{-1}$ being the inverse of those of the matrix $A$.
Also, singular value decompositions are known to provide the best rank-$n_s$ approximation of a matrix in the Frobenius norm \cite{Eckart1936}.
As $(w,u,x,G,F,z)$ change during optimization, for fixed matrices $V$ the expression $(\sum_{i=1}^{n_y} \int_{0}^{\tf} (h_x^i(x(t)) G_c(t) V)^T (h_x^i(x(t)) G_c(t) V) \; \mathrm{d} t)^{-1}$ is not necessarily diagonal any more, though.

% -------------
\subsection{Accuracy of the reduction approaches} \label{ref_accuracy}

A global error analysis is difficult due to the nonlinearity of \eqref{OED}. 
We start by looking at the objective function $\phi$ for fixed controls $u$ and sampling decisions $w$.
Obvisously, the values of $\phi(F_{ll}(\tf))$ and of
\begin{align*}
\phi(F_l(\tf)) =
\phi\left(
 \begin{pmatrix}
  F_{\hat p \hat p}(\tf) & F_{\hat p c}(\tf) \Xi_l \\
  \Xi_l^T F_{\hat p c}^T(\tf) & \Xi_l^T F_{c c}(\tf) \Xi_l
 \end{pmatrix}
\right)
\end{align*}
depend heavily on the weights $\theta$ of the ANN.
Even for the simpler case $n_{\hat p} = 0$ it is not clear how the scalar $\phi(\Xi_l^T F_{c c} \; \Xi_l) = \Xi_l^T F_{c c} \; \Xi_l$ and the value $\phi(F_{cc})$ are related.
The situation is clearer for the SVD approach. Here the decompositions 
$F(\tf) = F_{c c}(\tf) = \hat V \hat D \hat V^T = \sum_{i=1}^{n_\theta} \lambda_i v_i v_i^T$
and
$F_\svdu{n_s}(\tf) = V D V^T = \sum_{i=1}^{n_s} \lambda_i v_i v_i^T$
with eigenvalues $\lambda_i$ assumed to be positive can be used to derive
\begin{align}
 \phi_A(F(\tf)) - \phi_A(F_\svdu{n_s}(\tf)) &= \frac{1}{n_\theta} \sum_{i=1}^{n_\theta} \frac{1}{\lambda_i} - \frac{1}{n_s} \sum_{i=1}^{n_s} \frac{1}{\lambda_i} \label{eq_delta_phi_A} \\ 
 \phi_D(F(\tf)) - \phi_D(F_\svdu{n_s}(\tf)) &= \prod_{i=1}^{n_\theta} \frac{1}{\lambda_i^{1 / n_\theta}} - \prod_{i=1}^{n_s} \frac{1}{\lambda_i^{1 / n_s}} \label{eq_delta_phi_D} \\
 \phi_E(F(\tf)) - \phi_E(F_\svdu{n_s}(\tf)) &= \max_{1 \le i \le n_\theta} \frac{1}{\lambda_i} - \max_{1 \le i \le n_s} \frac{1}{\lambda_i} = \frac{1}{\lambda_{n_\theta}} - \frac{1}{\lambda_{n_s}} \ge 0 \label{eq_delta_phi_E}
\end{align}

As described in the literature, e.g., in \cite{barz2015nonlinear}, and plausible from the equations (\ref{eq_delta_phi_A}-\ref{eq_delta_phi_E}), the objective function $\phi_E$ is determined only by the smallest singular value $\lambda_{n_s}$, while the trace $\phi_A$ as the sum of the inverse singular values is dominated by several of the smallest eigenvalues, and the determinant $\phi_D$ is sensitive to the larger singular values. 
Given the specific situation of overparameterized ANN weights, this encourages the use of $\phi_D$ rather than $\phi_A$ or $\phi_E$.

These estimations, however, are based on identical $u^*$ and $w^*$, which is typically not the case due to the modification of the optimization problem. Already, the question how $\phi(F_\svdu{n_s}(\tf))$ changes with $w$ is nontrivial. Here, error bounds for the truncated singular value decomposition \cite{vu2021perturbation} or for eigenvalues of sums of Hermitian matrices as investigated by Knutson and Tao \cite{knutson2001honeycombs} might be an interesting line of research, yet beyond the scope of this paper.
We are also interested in the sensitivity of the resulting experiment as parameterized via $u$ and $w$.
Given the difficulty of this task, we focus on the sensitivity of $w^*$ with an analysis based on local necessary conditions of optimality. 
This is facilitated by the fact that in the continuous formulation well-posed OED control problems have bang-bang solutions in $w$, i.e., we have $w^*_i(t) \in \{0, 1\}$ for $t$ almost everywhere.
In \cite{Sager2013}, application of Pontryagin's maximum principle to \eqref{OED} resulted in the concept of information gain.

\begin{definition} {\bf{(Local and global information gain)}} \\%
The maps $P^i: \mathcal T \mapsto \in \R^{n_p \times n_p}$ for $i \in [n_y]$ defined with matrices
$$P^i(t) := \left( h_x^i(x(t)) G(t) \right)^T \left( h_x^i(x(t)) G(t) \right)
$$
are called {\emph{local information gain}}. Note that all $P^i(t)$ are positive semi-definite, and positive definite if the matrix $h_x^i(x(t)) G(t))$ has full rank $n_p$.

If $F^{-1}(\tf)$ exists, we call the matrices $\Pi: \mathcal T \mapsto \in \R^{n_p \times n_p}$ defined by
$$\Pi^i(t) := F^{-1}(\tf) P^i(t) F^{-1}(\tf)  \in \R^{n_p \times n_p}$$
the {\emph{global information gain}} of measurement function $i$.
\label{DInformationGain}
\end{definition}

\begin{lemma}{\bf{(Minimize trace of covariance matrix)}}
Let $\phi = \phi_A$ be the objective function of \eqref{OED}, let $w^*(\cdot)$ be an optimal control function, and let $\mu^*$ be the vector of Lagrange multipliers of the constraints $z_i(\tf) \le M_i$. 
If $w_i^*(t) = 1$, then
$$\frac{1}{n_p} \trace \left( \Pi^i(t) \right) \ge \mu_i^*.$$
\label{TCovarianceTrace}
\end{lemma}

\begin{lemma}{\bf{(Minimize determinant of covariance matrix)}}
Let $\phi = \phi_D = (\det(F^{-1}(\tf)))^{\frac{1}{n_p}}$ be the objective function of \eqref{OED}, let $w^*(\cdot)$ be an optimal control function, and let $\mu^*$ be the vector of Lagrange multipliers of the constraints $z_i(\tf) \le M_i$. 
If $w_i^*(t) = 1$, then
$$(\det(F^{-1}(\tf) ))^{\frac{1}{n_p}} \; \sum_{k,l = 1}^{n_p} (F(\tf))_{kl} (\Pi^i(t))_{kl} \ge \mu_i^*.$$
\label{TCovarianceDet}
\end{lemma}

The proofs follow from
$\frac{\partial \trace(A)}{\partial A} \Delta A = \trace( \Delta A )$ and
$\frac{\partial \det(A)}{\partial A} \Delta A = \det(A) \sum_{k,l = 1}^{n_p} A^{-1}_{kl} \Delta A_{kl}$
for symmetric, positive definite matrices $A \in \R^{n \times n}$ and from analysis of Pontryagin's maximum principle for the particular structures of \eqref{OED}, see \cite{Sager2013} for details and a similar result for $\phi_E$.

These results allow using the global information gains for an analysis of changes in $w^*$ for different approximations of sensitivities.
Of particular interest is the question of how the information gain changes with the number $n_s$ of singular values.
To apply Lemmata~\ref{TCovarianceTrace} and \ref{TCovarianceDet} for different values of $n_s$ in
\begin{align}
\Pi_\svdu{n_s}^i(t) := F_\svdu{n_s}^{-1}(\tf) \left( h_x^i(x(t)) G_\svdu{n_s}(t) \right)^T h_x^i(x(t)) G_\svdu{n_s}(t) F_\svdu{n_s}^{-1}(\tf), \label{eq_PISVD}
\end{align}
a re-scaling is convenient, though. Although the number of model parameters $n_p$ is considered in the OED criteria, e.g., 
$\phi_A(F(\tf)) = \frac{1}{n_p} \tr(F(\tf)^{-1})$ and
$\phi_D(F(\tf)) = \det(F(\tf)^{-1})^{\frac{1}{n_p}}$,
the objective function may change significantly when additional singular values are taken into account. 
It is well known that the Lagrange multiplier $\mu_i$ of the constraint $z_i(\tf) \le M_i$ represents the derivative of the objective function value with respect to small changes in $M_i$. In other words: how much more information would be obtained by measuring just a little bit more. As this value should be independent of the hyperparameter $n_s$ used to calculate the optimal design, it makes sense to define the scaling factor
\begin{align}
\gamma^{D}_{n_s} := \frac{\phi_D(F_{\svdu{n_t}}(\tf))}{\phi_D(F_{\svdu{n_s}}(\tf))} 
= \frac{\prod_{k=1}^{n_t} \frac{1}{\lambda_k}^{1/n_t}}{\prod_{k=1}^{n_s} \frac{1}{\lambda_k}^{1/n_s}}
= \prod_{k=1}^{n_s} \lambda_k^{\left(\frac{1}{n_s}-\frac{1}{n_t}\right)} \; \prod_{k=n_s+1}^{n_t} \lambda_k^{-\frac{1}{n_t}}
\end{align}
for the D-criterion and a fixed value $n_t > n_s$ of positive singular values.
However, as both sides of the inequalities in Lemmata~\ref{TCovarianceTrace} and \ref{TCovarianceDet} are multiplied, any scaling is valid for an analysis of the optimal sampling function.
Thus we can also use
\begin{align}
\gamma^{A}_{n_s} := 
\frac{n_s}{n_t}
\end{align}
for the linear A-criterion, which also provides a scaling that allows to infer the convergence of optimal samplings as $n_s \rightarrow n_t$. 
The main difference is that $\gamma^{A}_{n_s} \mu^*_i$ is going to converge towards the multiplier for $\phi_A(\svdu{n_t})$, while for all $n_s$ the value $\gamma^{D}_{n_s} \mu^*_i$ is an approximation of the multiplier for $\phi_D(\svdu{n_t})$.
With the scaling factors above we define the functions
\begin{align}
\Gamma^{D,i}_{n_s}(t) 
& := \gamma^{D}_{n_s} \; (\det(F_{\svdu{n_s}}(\tf)^{-1} ))^{\frac{1}{n_s}} \; \sum_{k,l = 1}^{n_s} (F_{\svdu{n_s}}(\tf))_{kl} (\Pi^i(t))_{kl} \label{eq_GammaDSVD}
\end{align}
and
\begin{align}
\Gamma^{A,i}_{n_s}(t) := \frac{\gamma^{A}_{n_s}}{n_s} \trace \left( \Pi_\svdu{n_s}^i(t) \right) \label{eq_GammaASVD}
\end{align}
and formulate the following result.
\begin{lemma}{\bf{(Global information gain for different numbers of singular values)}}
Let a SVD of the matrix $F(\tf)$ be given with decreasing singular values $\lambda_k > 0$ strictly positive for all $k \in [n_t]$. Let $1 \le n_s < n_t$. 
Then with the notation from above
\begin{align}
\Gamma^{D,i}_{n_t}(t) - \Gamma^{D,i}_{n_s}(t) = \phi_D(F_{\svdu{n_t}}) \displaystyle \sum_{k=n_s+1}^{n_t} P^i_{kk}(t) \lambda_k^{-1}
\end{align}
and
\begin{align}
\Gamma^{A,i}_{n_t}(t) - \Gamma^{A,i}_{n_s}(t) = \frac{1}{n_t} \displaystyle \sum_{k=n_s+1}^{n_t} P^i_{kk}(t) \lambda_k^{-2}.
\end{align}
\label{lem_convergence}
\end{lemma}
\begin{proof}
First we observe that
\begin{align*}
\Gamma^{D,i}_{n_s}(t) 
& = \gamma^{D}_{n_s} \; \phi_D(F_{\svdu{n_s}}(\tf)) \; \sum_{k,l = 1}^{n_s} (F_{\svdu{n_s}}(\tf))_{kl} (\Pi(t))_{kl} \\
& = \frac{\phi_D(F_{\svdu{n_t}}(\tf))}{\phi_D(F_{\svdu{n_s}}(\tf))} \phi_D(F_{\svdu{n_s}}(\tf)) \; \sum_{k,l = 1}^{n_s} (F_{\svdu{n_s}}(\tf))_{kl} (\Pi(t))_{kl} \\
& = \phi_D(F_{\svdu{n_t}}(\tf)) \; \sum_{k,l = 1}^{n_s} (F_{\svdu{n_s}}(\tf))_{kl} \lambda_k^{-1} P^i_{kl}(t) \lambda_l^{-1} \\
& = \phi_D(F_{\svdu{n_t}}(\tf)) \; \sum_{k = 1}^{n_s} \lambda_k^{-1} P^i_{kk}(t).
\end{align*}
This readily gives the claimed result
\begin{align*}
\Gamma^{D,i}_{n_t}(t) - \Gamma^{D,i}_{n_s}(t) 
& = \phi_D(F_{\svdu{n_t}}) \displaystyle \sum_{k=n_s+1}^{n_t} P^i_{kk}(t) \lambda_k^{-1}.
\end{align*}
For the A-criterion we calculate
\begin{align*}
\Gamma^{A,i}_{n_t}(t) - \Gamma^{A,i}_{n_s}(t) 
& = \frac{\gamma^{A}_{n_t}}{n_t} \trace \left( \Pi_\svdu{n_t}^i(t) \right) - \frac{\gamma^{A}_{n_s}}{n_s} \trace \left( \Pi_\svdu{n_s}^i(t) \right) \\
& = \frac{1}{n_t} \displaystyle \sum_{k=n_t}^{n_t} P^i_{kk}(t) \lambda_k^{-2} - \frac{1}{n_t} \displaystyle \sum_{k=1}^{n_s} P^i_{kk}(t) \lambda_k^{-2} \\
& = \frac{1}{n_t} \displaystyle \sum_{k=n_s+1}^{n_t} P^i_{kk}(t) \lambda_k^{-2}
\end{align*}
which finishes the proof.
\end{proof}

Lemma~\ref{lem_convergence} with positive summands shows a monotonic convergence of the global information gain criteria, as $n_s$ increases. The higher order $\lambda_k^{-2}$ for the A-criterion is an additional indication that the D-criterion might be better suited for SVD approximations with small singular values.

The functions $\Gamma^{D,i}_{n_s}(t)$ and the dual variables $\gamma^D_{n_s} \mu_i$ can be useful for numerical insight and for heuristics. 
For example, small values of $\| \Gamma^{A,i}_{n_s} - \Gamma^{A,i}_{n_s+1} \|$ in an appropriate function space norm might provide an alternative to the L-curve truncation criterion, as it implies that the global information gain for $a=\svdu{n_s}$ and $a=\svdu{n_s+1}$ is similar and results hence in similar optimal samplings $w^*$.
One line of future research beyond the scope of this paper is the question of how often and when during an optimization new SVD should be calculated, given that the SVD accuracy depends on the current optimization iterate. A comparison of \eqref{eq_GammaDSVD} for different values of $n_s$ might also be a promising heuristic for this purpose.
Another line of future work is that the information gain functions might be used for penalizations in the training of overfitted ANNs in sequential (online) OED for UDE, resulting in weights $\theta$ that perform similarly well in the training, but are regularized in a way such that an optimal sampling design is closer to that of more accurate but computationally more expensive approaches.

%\clearpage
% --------------
\section{Numerical results}\label{sec:results}

We provide implementation details of the novel methods and study illustrative UDE benchmark problems.

%will discuss numerical results showcasing the benefits of OED in general for the estimation of model parameters $\hat p$ and in the context of training hybrid models, i.e. estimating weights and biases $\theta$. 

%We will start with the well-known Lotka-Volterra  system before considering the chemical reaction of isocyanate and butanol to urethane. 

% --------------
\subsection{Implementation}

Our code is implemented in \texttt{Julia} \cite{bezanson2017julia}. It uses the packages \texttt{Lux} \cite{pal2023lux} for modeling, training, and evaluation of ANN and \texttt{DynamicOED.jl} \cite{Martensen2024c} for setting up the OED problems, e.g., augmenting the user-defined dynamical system by VDE for the sensitivities $G$ and the differential equations to evaluate the FIM $F(\tf)$, its inverse, and objective functions $\phi$. 
Several approaches for reducing computational effort have been implemented. For example, for fixed controls $u$ and initial values $x_0$, i.e. for OED problems where only the sampling decisions $w$ remain as degrees of freedom for the optimization, the states $x$ and the sensitivities $G$ are only computed once. 
The sampling decisions $w(t) \in \Omega = \{0,1\}^{n_y}$ can be relaxed to $w(t) \in \textrm{conv } \Omega = [0,1]^{n_y}$, allowing for an efficient solution via local nonlinear programming solvers such as \texttt{Ipopt} \cite{Waechter2006}. 
A theoretical justification for the relaxation was given in \cite{Sager2013}. Optimal solutions for relaxed sampling decisions on well-chosen time discretization grids typically already fulfill integrality due to a bang-bang property of $w$.

% The code is available at \texttt{GitHub}\footnote{\url{https://github.com/mathopt/DynamicOED.jl}}.

%Upcoming are steps to include efficient multiple shooting, lifting and condensing methods.

The implemented sequential procedure that mimics application of OED to obtain data from real experiments is identical and can be summarized as follows.
\begin{algo} {\bf{(General sequential evaluation procedure)}} \label{algo_sequential}%
\begin{enumerate}
\setlength\itemsep{-0.1em}
 \item Fix $\hat p$. Create synthetical data based on $\bar{f}(x, \hat p)$ and train $\myU(x, \theta)$ to obtain an initial guess $\bar \theta$
 \item Solve \eqref{OED} for $\bar p = (\hat p, \bar \theta)$ and obtain optimal controls and samplings $(u^*, w^*)$
 \item Create synthetical data based on the values of $w^*_i(t) \in [0,1]$ and on the observation functions $h^i(x^*(t))$  
 \item Solve a combined parameter estimation and ANN training problem to obtain updated values for $(\hat p, \theta)$
 \item Evaluate performance criterion
\end{enumerate}
\end{algo}

\noindent
In step 1., weights $\theta$ may already be available as the result of an initial training with existing real or synthetic data. In this study, we generate synthetic data from the ground truth model and perform initial training to obtain a function $\myU(x,\theta)$ such that the initial value problem \eqref{f_hat} has a unique solution. 

In some scenarios, the degrees of freedom in step 2. are restricted, e.g., to compare results for the cases where the control $u$ is fixed or the sampling is equidistant and not a degree of freedom. If not otherwise stated, we use the A-criterion $\phi_A$ for optimization.

In step 3. of Algorithm~\ref{algo_sequential}, we evaluate $h^i(x^*(t_j))$ at times $t_j$ determined as follows. In the direct approach we use to solve \eqref{OED}, $w(t)$ is discretized as $w_{i,j} \in [0,1]$ for $i \in n_y$ and $j \in [N]$  with $N \in \mathbb N$ being the number of discretization intervals. We normalize the optimal discretized solution $w^*_{i,j}$ by dividing by the corresponding bound $M_i$ on the measurements for each measurement function $h^i$, i.e.,
$\bar{w}^*_{i,j} = w^*_{i,j} / M_i, \; i \in [n_y], \, j \in [N]$ and interpret the elements of $\bar{w}^*$ as probabilities. For each measurement function we then sample the wanted number of measurement time points without replacement. A discrete measurement time point for measurement function $h^i$ is added at $\frac{t_{j-1}+t_j}{2}$ if $w_{i,j}$ was selected, where $t_{i-1}$ and $t_i$ are the start and end points of the $i$-th discretization interval, for $i \in [N]$.
The values $h^i(x^*(t_j))$ are then modified with normally distributed noise.

Concerning step 4., efficient algorithms for the concurrent estimation of model parameters and weights are an open research topic. For the small-scale illustrative hybrid benchmark examples in this paper, we apply an iterative heuristic procedure: we use the two different approaches alternatingly by fixing either $\hat p$ or $\theta$ to the most recent estimation, until the model parameters $\hat p$ do not change any more. This allows to use standard approaches for the estimation of $\hat p$ and the training of $\theta$, respectively. For all parameter estimation problems we used a standard Gauß-Newton algorithm, while for the training of the ANN we used the \texttt{Adam} algorithm \cite{Kingma2014}.
To compensate the stochastic nature of ANN training, we averaged over five training runs from the same initial $\theta$.

In step 5., the performance of a particular OED approach on the outcome of a follow-up parameter estimation / model training problem is evaluated. We did this using different approaches: a) by comparing parameter values and their uncertainty directly, b) by using the OED objective functions $\phi$ as an indicator for the uncertainty, c) based on simulations of the posterior distributions of the differential states $x$, and d) by comparing the absolute error between ANN $\myU(x, \theta)$ and the approximated ground-truth terms $\bar{f}(x, \hat p)$ of the differential equation models over a domain of interest $x \in \mathcal X$.

% --------------
% --------------
\subsection{Lotka-Volterra} \label{sec_lotka}

The Lotka-Volterra equations are commonly used to describe population dynamics of interacting species, typically in a predator-prey relationship. We will use an extended version involving a fishing control $u$ that was suggested to obtain a benchmark optimal control problem \cite{Sager2006}.
The ODE is given by
\begin{align}
	\begin{split}
		\dot x_1(t) &=  x_1(t) - \hat p_1 x_1(t) x_2(t) - \hat p_2 u(t) x_1(t), \\
		\dot x_2(t) &= -x_2(t) + \hat p_3 x_1(t) x_2(t) - \hat p_4 u(t) x_2(t), \\
		x(0)	&= x_0
	\end{split}
	\label{lotka}
\end{align}
with a time horizon $\mathcal T = [0, 12]$, model parameters $\hat p_1 = 1.0, \, \hat p_2 =0.4, \, \hat p_3=1.0, \, \hat p_4=0.2$, a control function $u(t) \in [0,1]$ and fixed initial values $x_0=(0.7,0.5)$. 
We assume that we can directly measure the states separately via $h^1(x(t)) = x_1(t), \; h^2(x(t))=x_2(t)$ with a limit of $M_1=M_2=4$ time units each.

To study hybrid modeling, we assume the interaction terms $x_1(t) x_2(t)$ to be unknown. We replace it in \eqref{lotka} with an ANN $\myU : \mathbb{R}^2 \times \mathbb{R}^{n_\theta} \mapsto \mathbb{R}$ and obtain
\begin{align}
	\begin{split}
		\dot x_1(t) &=  x_1(t) - \myU(x(t), \theta) - \hat p_2 u(t) x_1(t), \\
		\dot x_2(t) &= -x_2(t) + \myU(x(t), \theta) - \hat p_4 u(t) x_2(t), \\
		x(0)	&= x_0
	\end{split}
	\label{lotka_hybrid}
\end{align}

The ANN $\myU$ has two hidden layers with ten neurons each, using the \textit{tanh} activation function $\textrm{tanh}(x) = \frac{\exp(x)-\exp(-x)}{\exp(x)+\exp(-x)}$ in all hidden layers. The output layer is equipped with the \textit{softplus} function $\textrm{softplus}(x) = \ln(1+\exp(x))$. This way, the a-priori knowlegde that the correct term in the right-hand side is positive at all times for positive initial values $x_0$ can be easily incorporated. 
The total amount of weights and biases of this net is $n_\theta=151$.

In the following subsections we study OED problems based on (\ref{lotka}-\ref{lotka_hybrid}) with different combinations of model parameters $\hat p$ and weights $\theta$ by application of Algorithm~\ref{algo_sequential}.

% --------------
\subsubsection{Estimating mechanistic model parameters} \label{subsec:oed_model_parameters}

We start by considering OED for the ODE \eqref{lotka} to obtain a basis for comparison. In particular, we shall investigate the complementary impact of control functions $u$ and sampling decisions $w$ on uncertainty quantification.
As in \cite{Sager2013}, we assume the model parameters $\hat p_2$ and $\hat p_4$ to be known and calculate experiments that provide rich information for estimating the model parameters $(\hat p_1, \hat p_3)$.
We consider four different OED problems in step 2. of Algorithm~\ref{algo_sequential}, combining a) equidistant and free sampling $w$ with b) fixed controls $u \equiv 0$ and optimal controls $u \equiv u^*$.
In all cases, 
in step 1. of Algorithm~\ref{algo_sequential} we sample $\hat p_i \sim \mathcal N(1.0, 0.25^2), ~ i \in \lbrace 1 ,3 \rbrace$,
in step 3. we select three measurement times in $\mathcal T$ per differential state, and add normally distributed noise with mean $0$ and variance $\sigma^2 = 0.1^2$. 

The results of the parameter estimation problems are given in Table \ref{tab:default_lotka}. 
\begin{table}[h!!!]
	\centering
	\begin{tabular}{c|cc}
	 	Setting     & $\hat p_1$ & $\hat p_3$ \\ \toprule 
		$w0$-$u0$   & $0.998 \pm 0.057$ & $0.980 \pm 0.082$ \\
    $w^*$-$u0$  & $1.005 \pm 0.037$ & $1.008 \pm 0.034$ \\ %\midrule
		$w0$-$u^*$  & $1.013 \pm 0.055$ & $1.003 \pm 0.004$ \\
    $w^*$-$u^*$ & $0.995 \pm 0.017$ & $1.004 \pm 0.006$ \\ \bottomrule
	\end{tabular}
	\caption{Comparison of parameter estimates and their uncertainty for the four considered OED problems that vary in constraints on sampling function $w$ and control function $u$, see (\ref{wua_w}--\ref{wua_u}) for definitions.}\label{tab:default_lotka}
\end{table}

While three equidistant measurements suffice to determine the two parameters with a maximum error of $2\%$, a significant improvement in accuracy and in the reduction of uncertainty can be observed when timing the measurements optimally or stimulating the system with controls $u^*$. The effect is even stronger when both are combined. 
The more accurate and less uncertain parameter estimates directly carry over to more accurate predictions in the differential states and the interaction terms, as shown in Figure~\ref{fig:trajectories_default_lotka}.

\begin{figure}[h!!!]
	\centering
	\begin{subfigure}{.24\linewidth}
		\includegraphics[width=\linewidth]{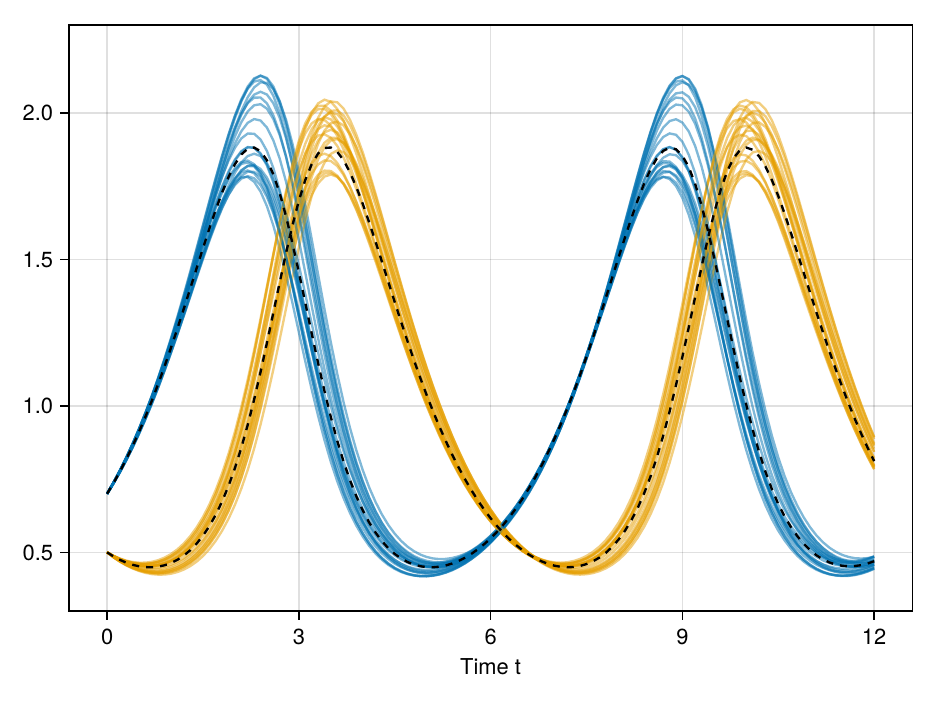}
	\end{subfigure}
	\begin{subfigure}{.24\linewidth}
		\includegraphics[width=\linewidth]{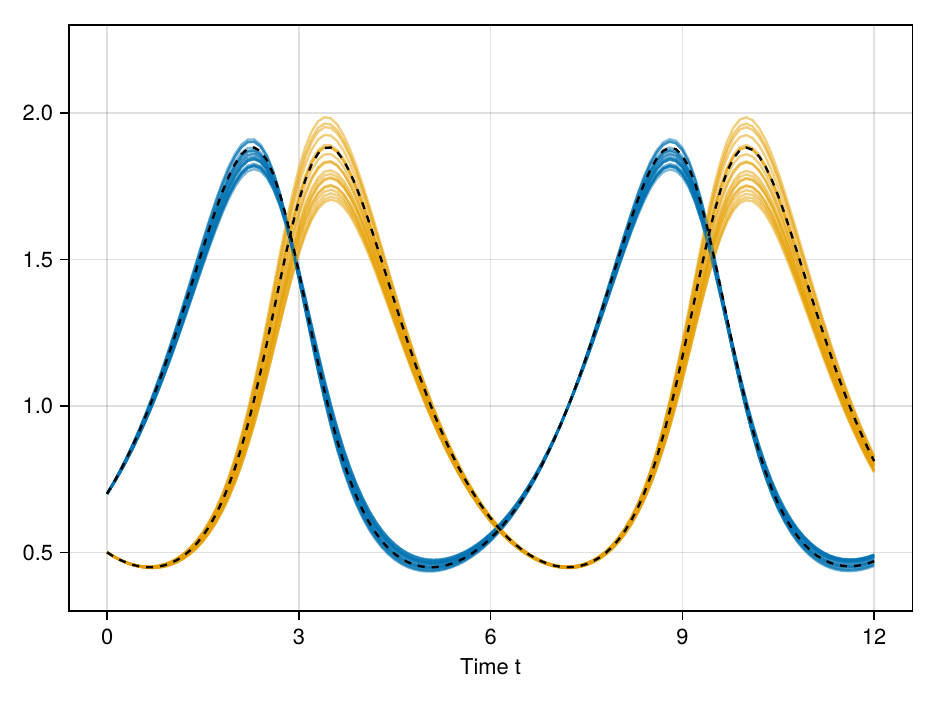}
	\end{subfigure}
	\begin{subfigure}{.24\linewidth}
		\includegraphics[width=\linewidth]{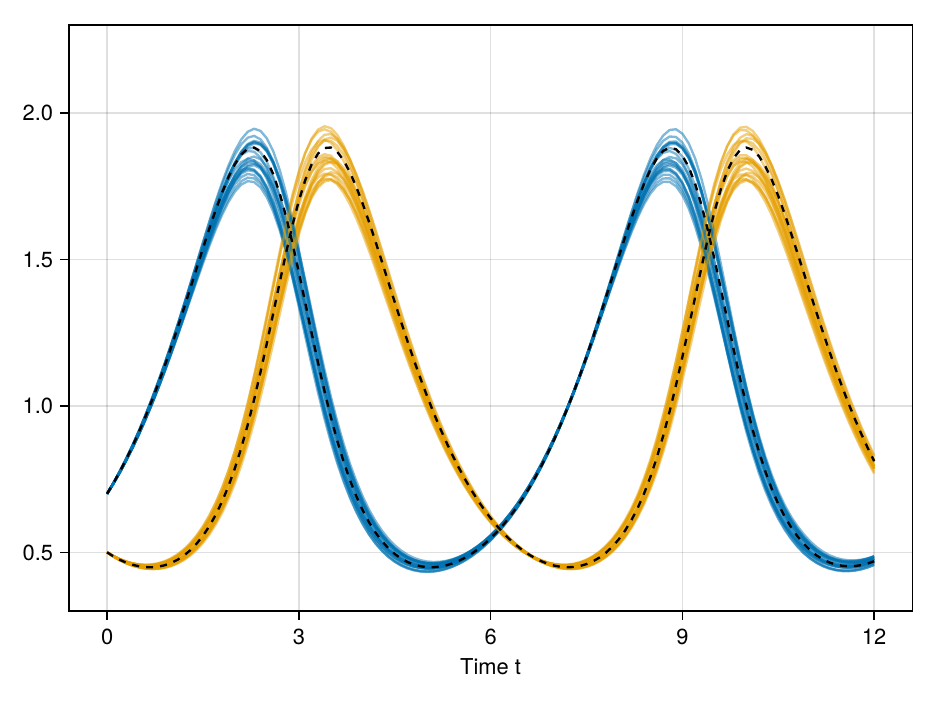}
	\end{subfigure}
	\begin{subfigure}{.24\linewidth}
		\includegraphics[width=\linewidth]{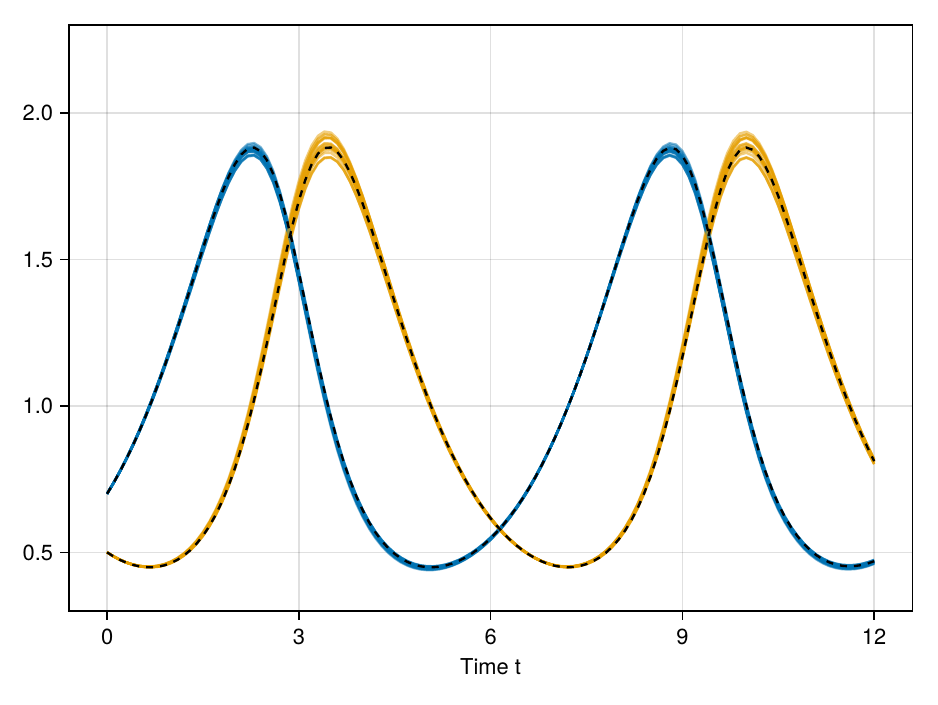}
	\end{subfigure}

	\begin{subfigure}{.24\linewidth}
		\includegraphics[width=\linewidth]{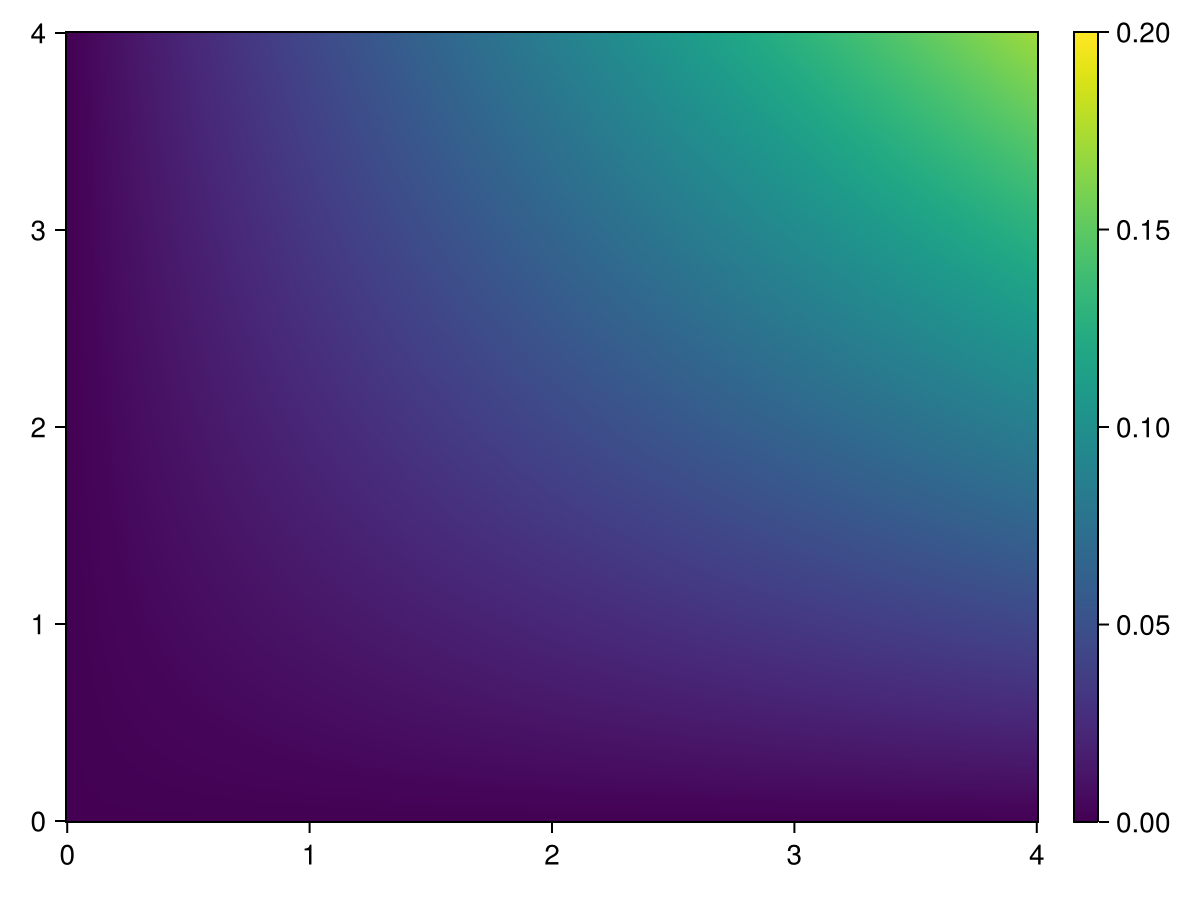}
		\subcaption{Setting $w0-u0$.}
	\end{subfigure}
	\begin{subfigure}{.24\linewidth}
		\includegraphics[width=\linewidth]{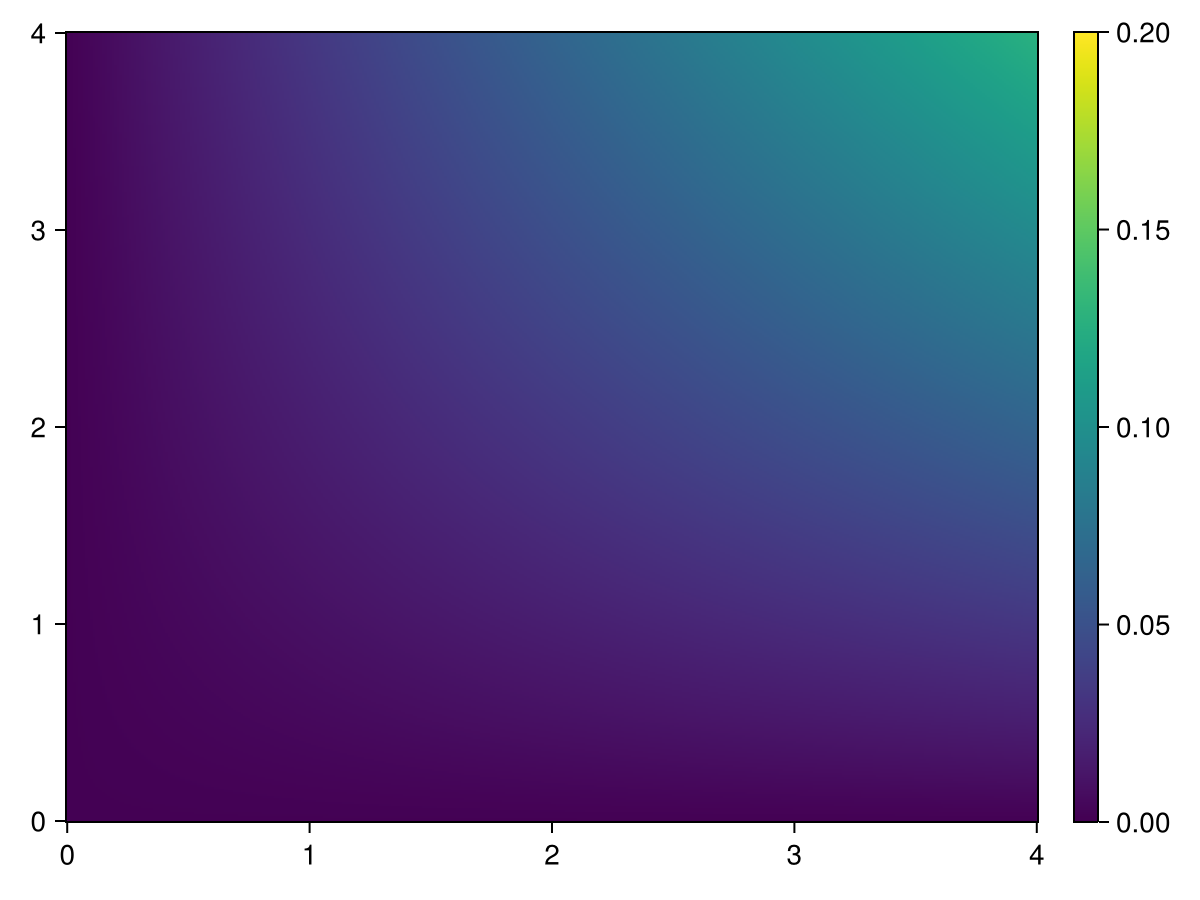}
		\subcaption{Setting $w0-u^*$.}
	\end{subfigure}
	\begin{subfigure}{.24\linewidth}
		\includegraphics[width=\linewidth]{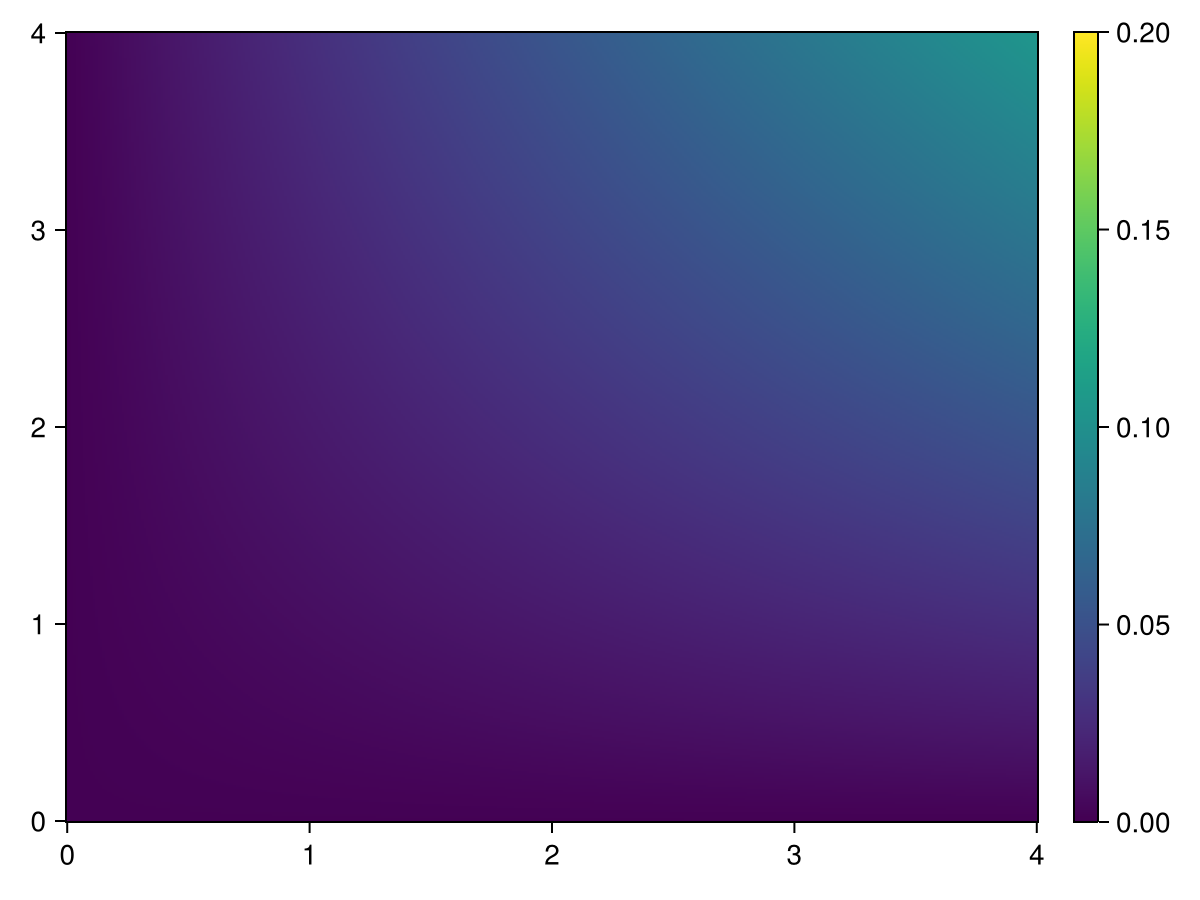}
		\subcaption{Setting $w^*-u0$}
	\end{subfigure}
	\begin{subfigure}{.24\linewidth}
		\includegraphics[width=\linewidth]{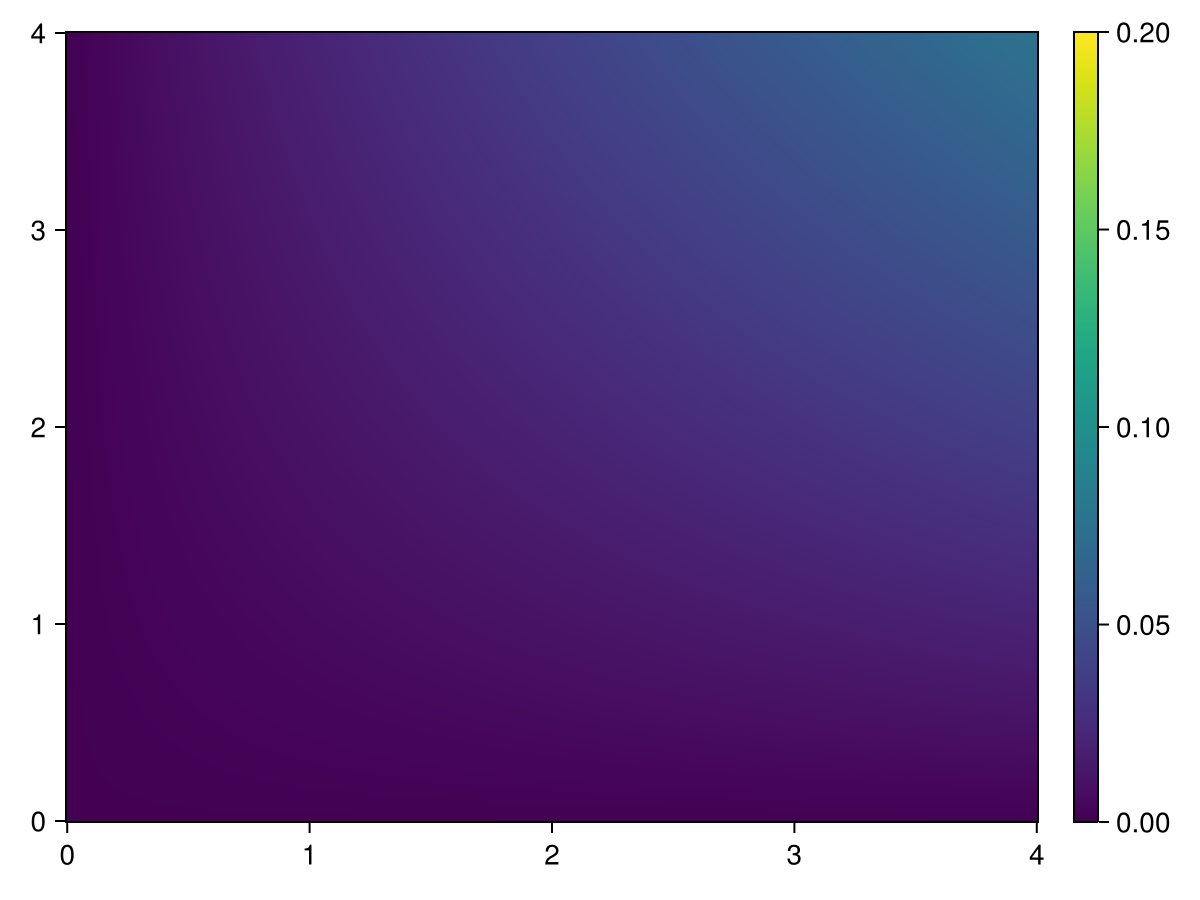}
		\subcaption{Setting $w^*-u^*$.}
	\end{subfigure}
	\caption{Top row: Monte-Carlo simulations of \eqref{lotka} using the parameter estimates and uncertainties from Table \ref{tab:default_lotka}. Bottom row: absolute errors of interaction term $\lvert \hat p_i \cdot x_1 x_2 - x_1 x_2\rvert$ averaged for $i=1,3$ on $\mathcal X = [0,4]^2$. Both rows (from left to right) show that solving \eqref{OED} reduces the uncertainty significantly.
	}
	\label{fig:trajectories_default_lotka}
\end{figure}

Figure~\ref{fig_gig_mechanistic} illustrates Lemma~\ref{TCovarianceTrace} and the concept of global information gain.
\begin{figure}[h!!!]
	\centering
 	\includegraphics[width=0.95\linewidth]{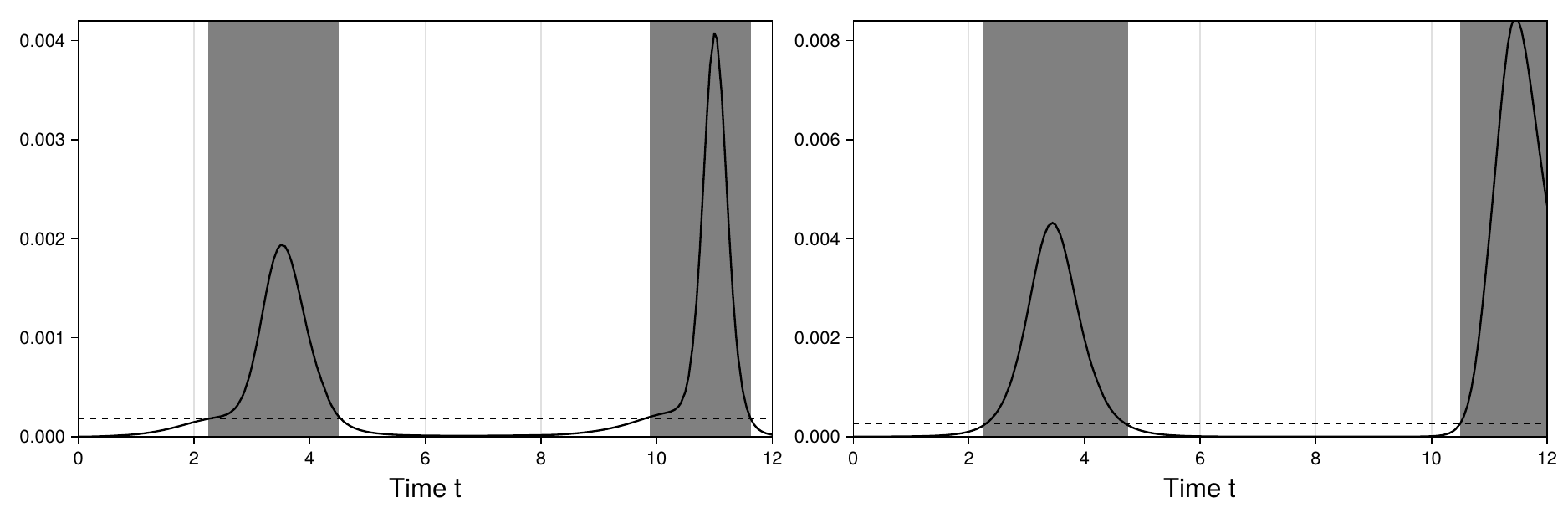}
	\caption{Visualization of global information gain for optimized $\phi_A$. The Lagrange multipliers $\mu^*$ of the constraints $z_i(\tf) \le M_1$ and the functions $\frac{1}{n_p} \trace \left( \Pi^i(t) \right)$ are plotted (dotted and solid, respectively) for the two measurement functions $h^i(x(t)) = x_i(t)$ together with optimal sampling functions $w_i^*(\cdot)$ for $i=1$ (left) and $i=2$ (right). As stated in Lemma~\ref{TCovarianceTrace}, if $w_i^*(t) = 1$ (gray background), then $\frac{1}{n_p} \trace \left( \Pi^i(t) \right) \ge \mu_i^*$.
}
	\label{fig_gig_mechanistic}
\end{figure}

% --------------------
\subsubsection{Training of ANN weights} 

Now we consider the UDE \eqref{lotka_hybrid}
for OED problems that differ in constraints on $(u,w)$ in step 2. of Algorithm~\ref{algo_sequential} as in Section~\ref{subsec:oed_model_parameters}, but also in the chosen approach $a$ in (\ref{wua_w}--\ref{wua_a}) for various numbers $n_s$ of singular values to approximate the FIM of the overparameterized ANN training problem.
%
%In all cases, 
%in step 1. of Algorithm~\ref{algo_sequential} we sample $\hat p_i \sim \mathcal N(1.0, 0.25^2), ~ i \in \lbrace 1 ,3 \rbrace$,
%in step 3. we select three measurement times in $\mathcal T$ per differential state, and
%in step 3. we add normally distributed noise with mean $0$ and variance $\sigma^2 = 0.1^2$. 
We use a procedure similar to the one in Section~\ref{subsec:oed_model_parameters}, but increase the maximum number of measurements per differential state in step 3. to 10.
A prototypical solution of problem \eqref{OED} is shown in Figure~\ref{fig:solution_OED_SVD} for $w^*$-$u^*$-$\svdu{2}$.
\begin{figure}[h!!!]
	\centering
	\includegraphics[width=.8\linewidth]{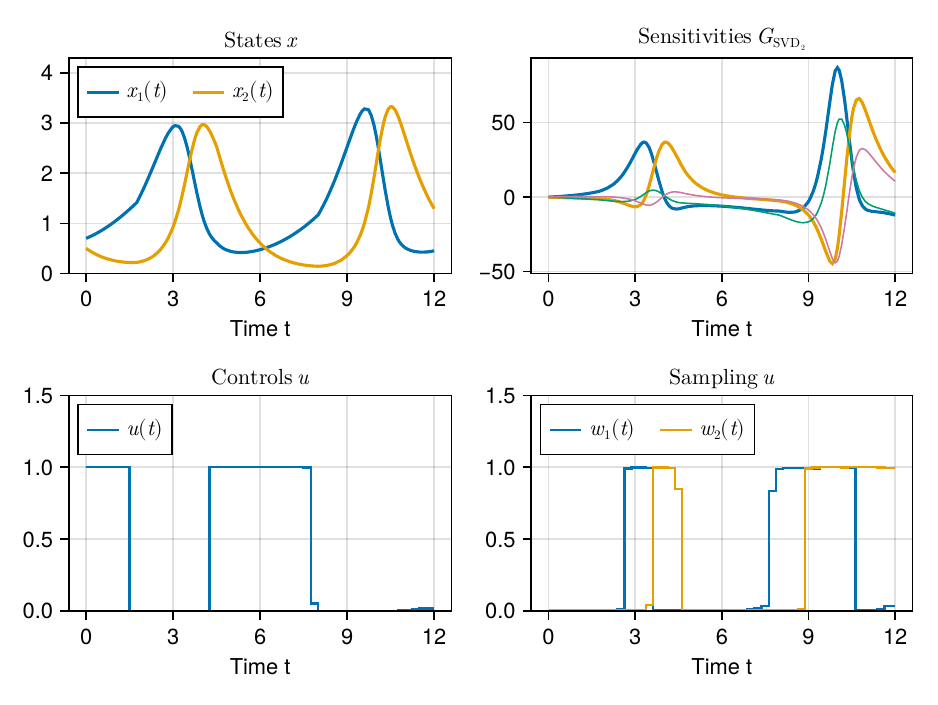}
	\caption{Solution of \eqref{OED} for \eqref{lotka_hybrid} with the embedded ANN using the $w^*$-$u^*$-$\svdu{2}$ approach.}
	\label{fig:solution_OED_SVD}
\end{figure}
As expected, the resulting optimal controls and states are qualitatively similar to those for the mechanistic ODE considered in Section~\ref{subsec:oed_model_parameters} and in \cite{Sager2013}. 
The optimal sampling $w^*$ consists of two time intervals each for both measurement functions, similar but not identical to the gray regions in Figure~\ref{fig_gig_mechanistic}.
The optimal control $u^*$ excites the differential states towards higher oscillations which go along with increased oscillations in the sensitivities. This way, more information can be collected in regions of the state-space where the model is highly sensitive towards the weights and biases $\theta$.

Looking in more detail into differences of the various approaches $a$, we can observe a decrease in uncertainty measured in the objective function $\phi_\svdu{2}$ if OED is applied, see Table~\ref{tab:criteria_with_controls}. The deviation of weights $\theta^*$ from the initial weights $\bar \theta$ is not a reliable indicator, though, because of the overparameterized ANN, and only shown for completeness.
\begin{table}[h!!!]
	\centering
	\begin{tabular}{r|cc || r|cc}
	 	Approach        & $\phi_\svdu{n_s}(w,u)$ & $\lvert \Delta\theta \rvert_2$  & Approach & $\phi_\svdu{n_s}(w,u)$ & $\lvert \Delta\theta \rvert_2$ \\ \toprule 
   $w0$-$u0$-$l$           & 5.25e-3 & 0.25 & $w0$-$u^*$-$l$         & 7.12e-4 & 1.91 \\ 
   $w^*$-$u0$-$o$          & 2.69e-3 & 0.83	& $w^*$-$u^*$-$o$         & 1.62e-4 & 1.31 \\
   $w^*$-$u0$-$l$          & 2.34e-3 & 0.73	& $w^*$-$u^*$-$l$         & 1.59e-4 & 2.31 \\
   $w^*$-$u0$-$ll$         & 2.40e-3 & 0.32	& $w^*$-$u^*$-$ll$        & 1.63e-4 & 1.85 \\
   $w^*$-$u0$-$\svdu{2}$   & 2.33-3  & 0.33	& $w^*$-$u^*$-$\svdu{2}$  & 1.53e-4 & 2.51 \\ 
   $w^*$-$u0$-$\tsvdu{2}$  & 1.69e-3 & 0.20	& $w^*$-$u^*$-$\tsvdu{2}$	& 1.61e-4 & 1.79 \\ \bottomrule
%		 \multirow{6}{*}{$u \equiv 0$} & Equidistant	& 5.25e-3 & 0.25  & \multirow{6}{*}{$u = u^*$} & 	Equidistant & 7.12e-4 & 1.91 \\ 
%		 	                             & OED-NN-ol 	      & 2.69e-3 & 0.83	&                        & OED-NN-ol 			  & 1.62e-4 & 1.31 \\
%		 	                             & OED-NN-l 	      & 2.34e-3 & 0.73	&                        & OED-NN-l 			  & 1.59e-4 & 2.31 \\
%		 	                             & OED-NN-ll    & 2.40e-3 & 0.32	&                            & OED-NN-ll 		& 1.63e-4 & 1.85 \\
%		 	                             & OED-SVD-w$_2$  &  1.69e-3  & 0.20	& 	                         & OED-SVD-w$_2$  & 1.61e-4 & 1.79 \\
%		 	                             & OED-SVD$_2$  & 2.33-3  & 0.33	& 	                         & OED-SVD$_2$  & 1.53e-4 & 2.51 \\ \midrule
	\end{tabular}
	\caption{D-Criterion $\phi_D$ of the SVD-based OED approach and averaged deviation of trained weights $\theta^*$ from the initial weights $\bar \theta$. Compared are the OED problems corresponding to fixed and free controls $u$, and for different dimension reduction approaches, respectively.}
	\label{tab:criteria_with_controls}
\end{table}
There are two trends observeable in the table. First, taking measurements at optimized time points compared to equidistant measurements yields a noticeable reduction in the objective, regardless whether controls are optimized along with the sampling decision or not. Second, applying optimized controls $u^*$ reduces the objective for almost all strategies by approximately one order of magnitude.

Another evaluation criterium is the comparison of absolute errors of ground truth model and ANN output $\myU(x,\theta^*)$ over a domain of interest $\mathcal X$.
Figure~\ref{fig:absolute_errors} shows the increase in accuracy due to OED. One observes advantages for the SVD approach, and in general for all OED problems with additional degrees of freedom.
Of particular interest is the correspondence between OED solutions with optimal $u^*$ resulting in a larger training domain (see also Figure~\ref{fig:solution_OED_SVD} for the increased oscillations of $x$), and the increased accuracy for the extrapolated domain $\mathcal X = [0,4]^2$ (as compared to the originally considered state-space domain $\mathcal X = [0,2]^2$).

Finally, we have a look at the global information gain function. Figure~\ref{fig_gig} shows how the functions differ for selected reduced OED approaches $a$, resulting in different sampling intervals in which $w_1^*(t) = 1$.
\begin{figure}[h!!!]
	\centering
  \includegraphics[width=0.48\linewidth]{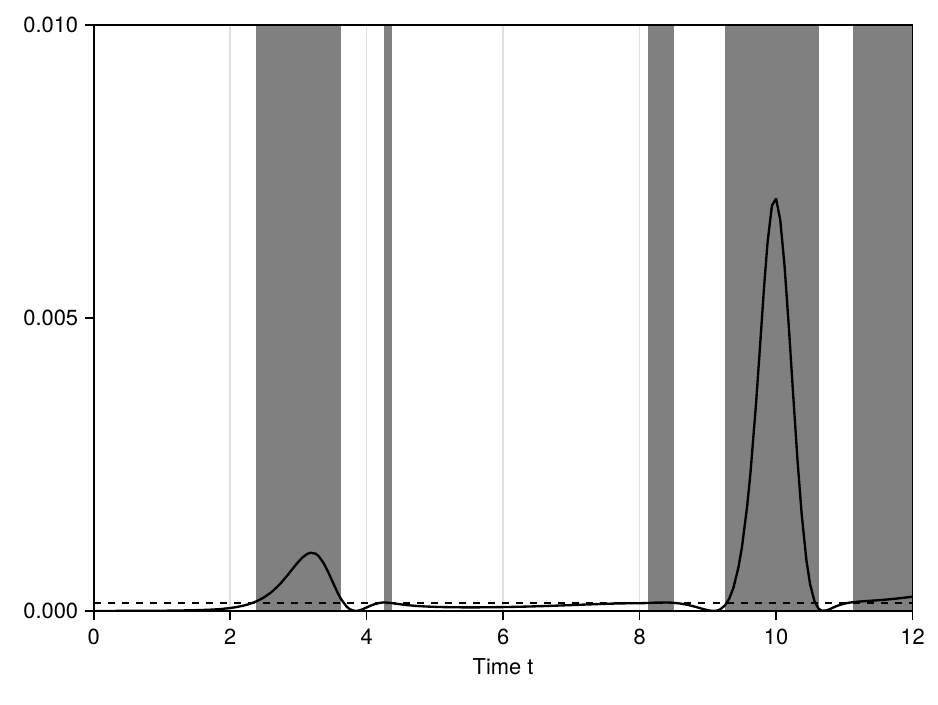}
  \includegraphics[width=0.48\linewidth]{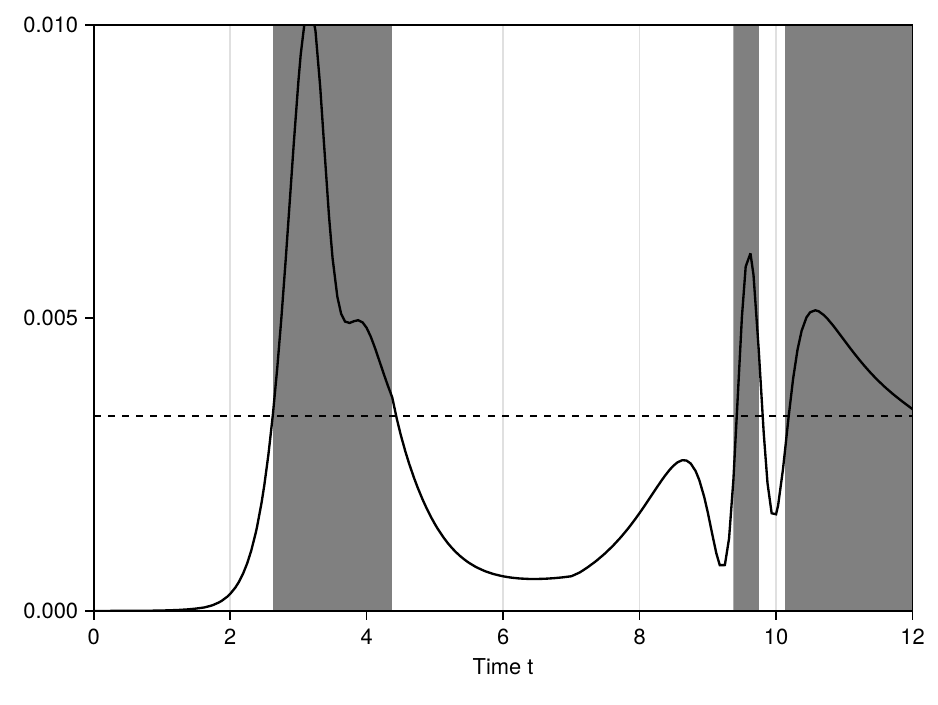}

  \includegraphics[width=0.48\linewidth]{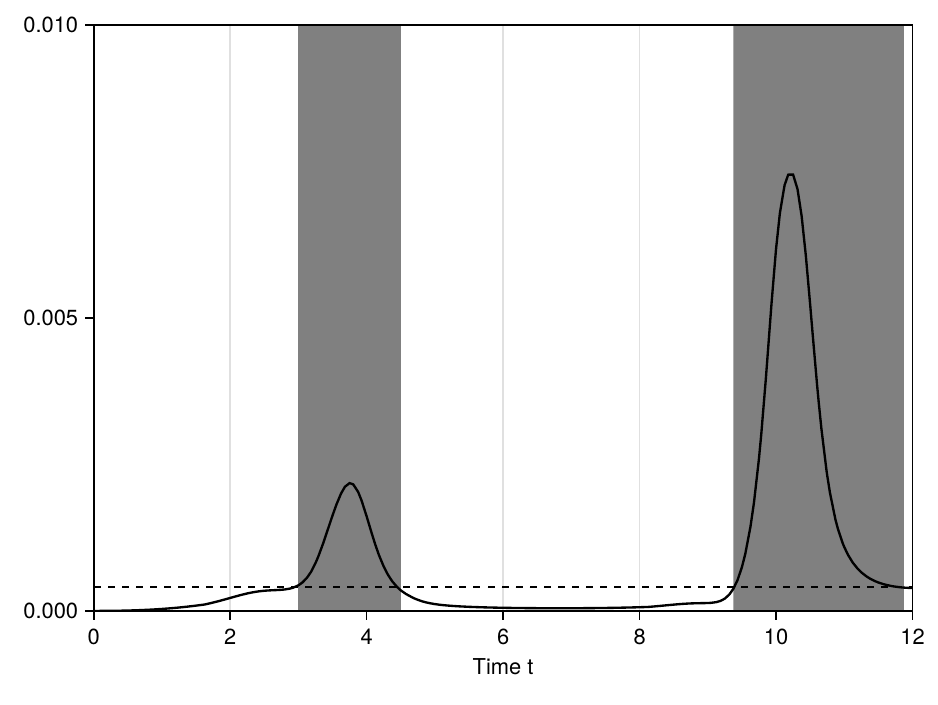}
  \includegraphics[width=0.48\linewidth]{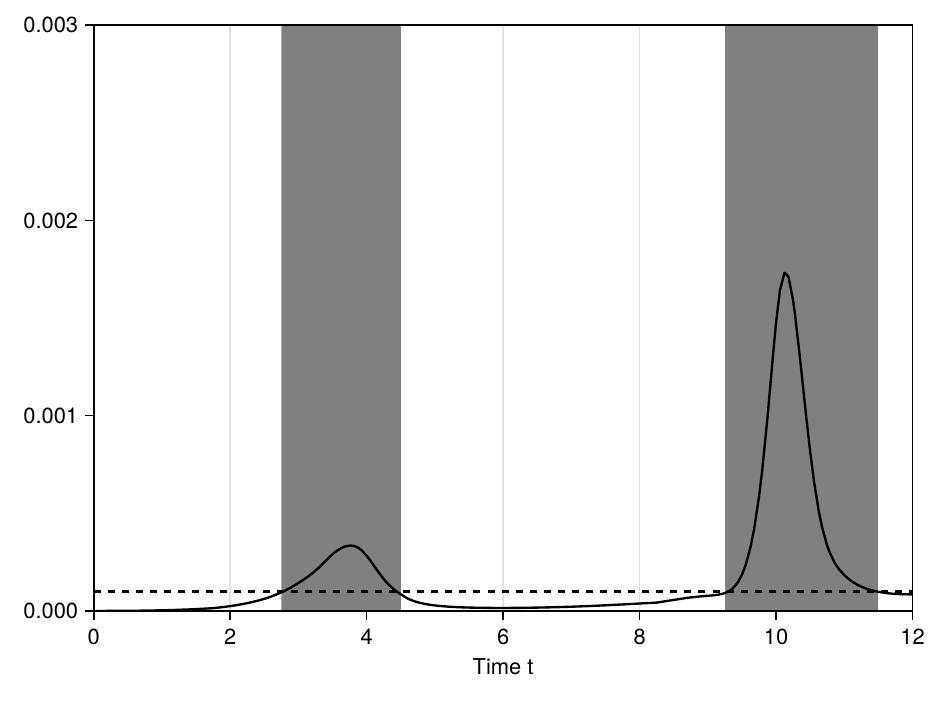}

% This effect is even stronger in case of the $\phi_D$ criterion (bottom row).
%  \includegraphics[width=0.32\linewidth]{plots/information_gain/lotka_svd_ns7_DCriterion_information_gain_vspan_h1.pdf}
%  \includegraphics[width=0.32\linewidth]{plots/information_gain/lotka_svd_ns8_DCriterion_information_gain_vspan_h1.pdf}
%  \includegraphics[width=0.32\linewidth]{plots/information_gain/lotka_svd_ns9_DCriterion_information_gain_vspan_h1.pdf}
	\caption{Visualization of Lemma~\ref{TCovarianceTrace} for $w^*$-$u^*$-$l$, $w^*$-$u^*$-$ll$, and $w^*$-$u^*$-$\svdu{7}$ and the A-criterion, and for Lemma~\ref{TCovarianceDet} and $w^*$-$u^*$-$\svdu{7}$ and the D-criterion, from top left to bottom right. Analogous to Figure~\ref{fig_gig_mechanistic} left, the functions $\frac{1}{n_p} \trace \left( \Pi^1(t) \right)$ and $(\det(F^{-1}(\tf) ))^{\frac{1}{n_p}} \; \sum_{i,j = 1}^{n_p} (F(\tf))_{ij} (\Pi^1(t))_{ij}$ (bottom right) in solid, the multipliers $\mu^*_1$ in dotted, and the sampling intervals in gray are shown for the first measurement function $h^1(x(t))=x_1(t)$. While the optimal samplings $w^*$ for $w^*$-$u^*$-$l$ and $w^*$-$u^*$-$ll$ (top row) differ from the optimal sampling in Figure~\ref{fig_gig_mechanistic} left, the $\svdu{7}$ samplings are similar, both for the A- and the D-criterion (bottom row).}
	\label{fig_gig}
\end{figure}
% remember: ll: 10 trace PI and all traces times 1/n_p

\begin{figure}[t]
	\centering
    \begin{subfigure}{.3\linewidth}
		\includegraphics[width=\linewidth]{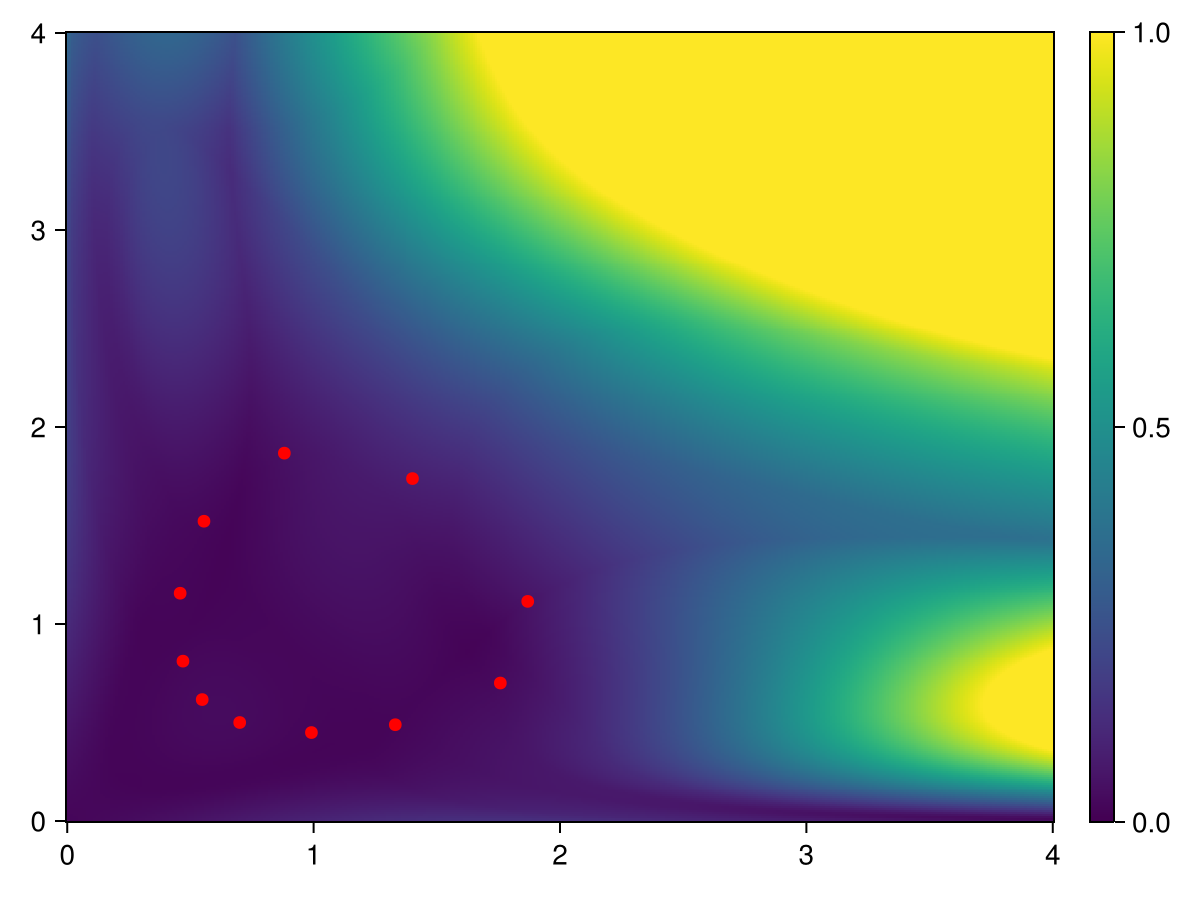}
		\subcaption{$w0$-$u0$-$l$, $\delta = 0.52$}
	\end{subfigure}
	\begin{subfigure}{.3\linewidth}
		\includegraphics[width=\linewidth]{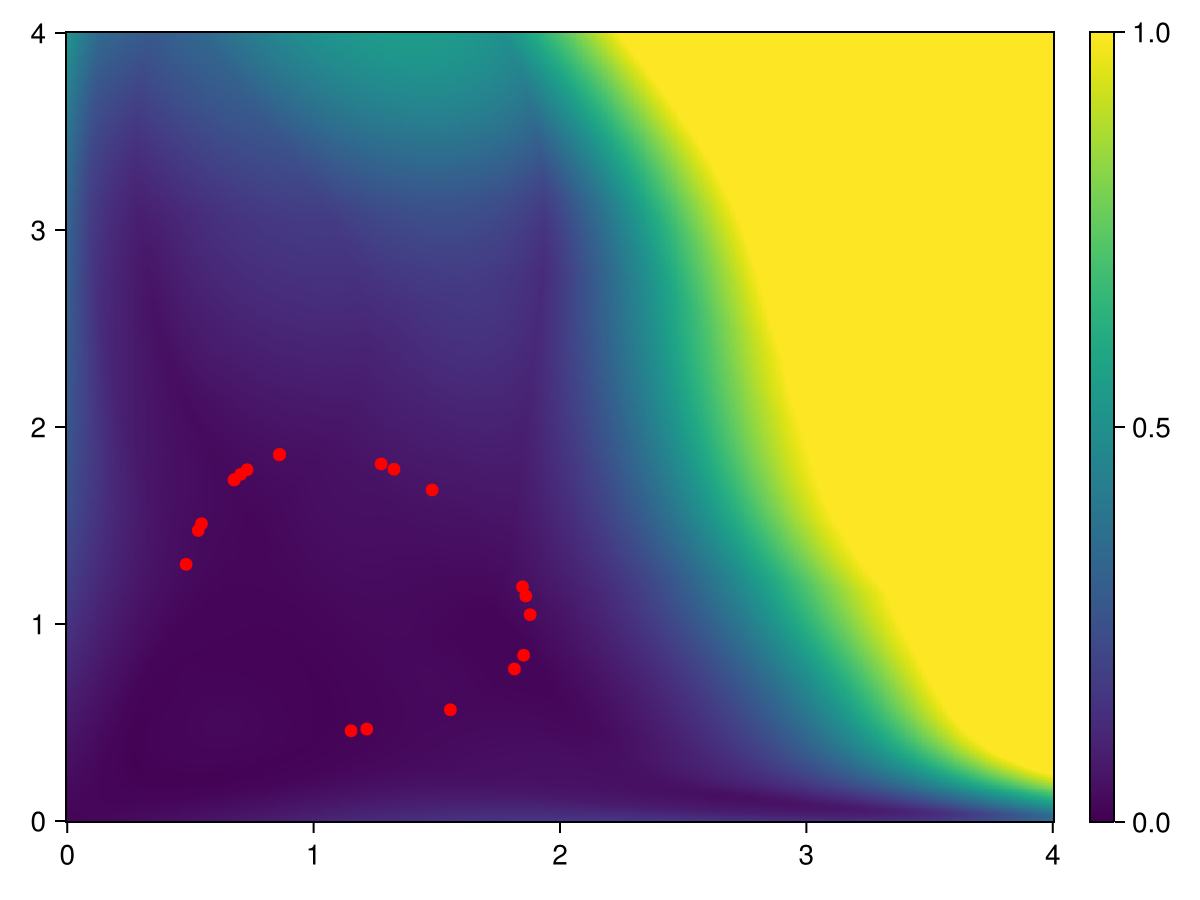}
		\subcaption{$w^*$-$u0$-$o$, $\delta = 0.52$}
	\end{subfigure}
	\begin{subfigure}{.3\linewidth}
		\includegraphics[width=\linewidth]{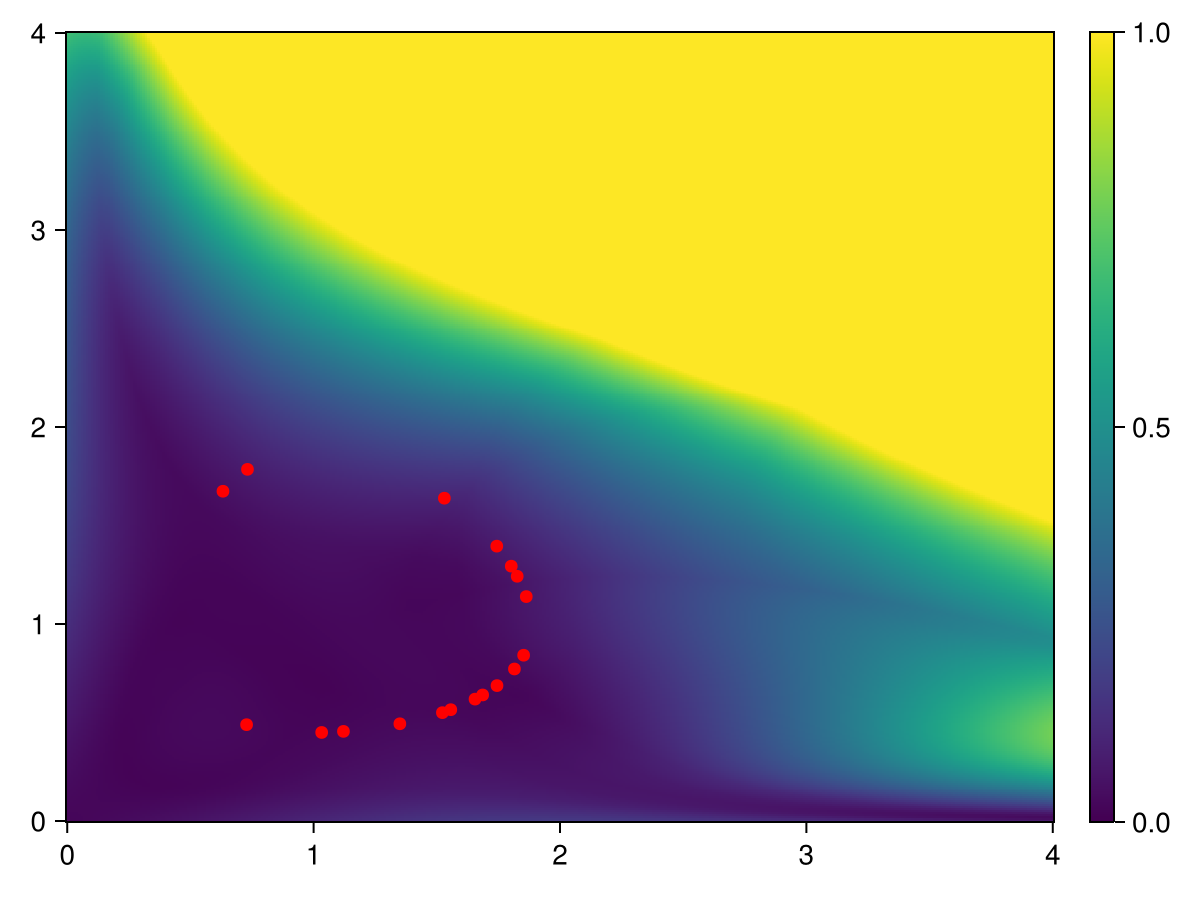}
		\subcaption{$w^*$-$u0$-$l$, $\delta = 1.11$}
	\end{subfigure}

	\begin{subfigure}{.3\linewidth}
		\includegraphics[width=\linewidth]{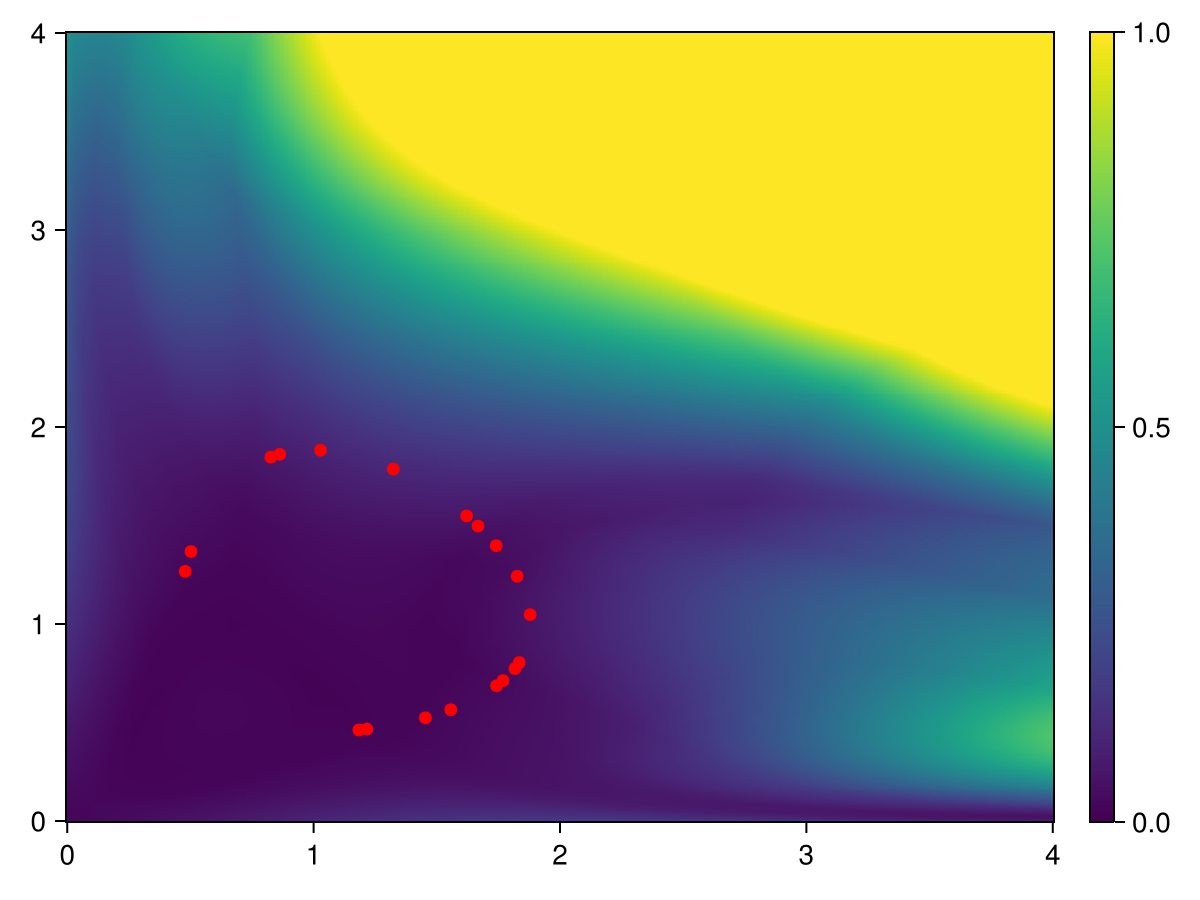}
		\subcaption{$w^*$-$u0$-$ll$, $\delta = 0.64$}
	\end{subfigure}
	\begin{subfigure}{.3\linewidth}
		\includegraphics[width=\linewidth]{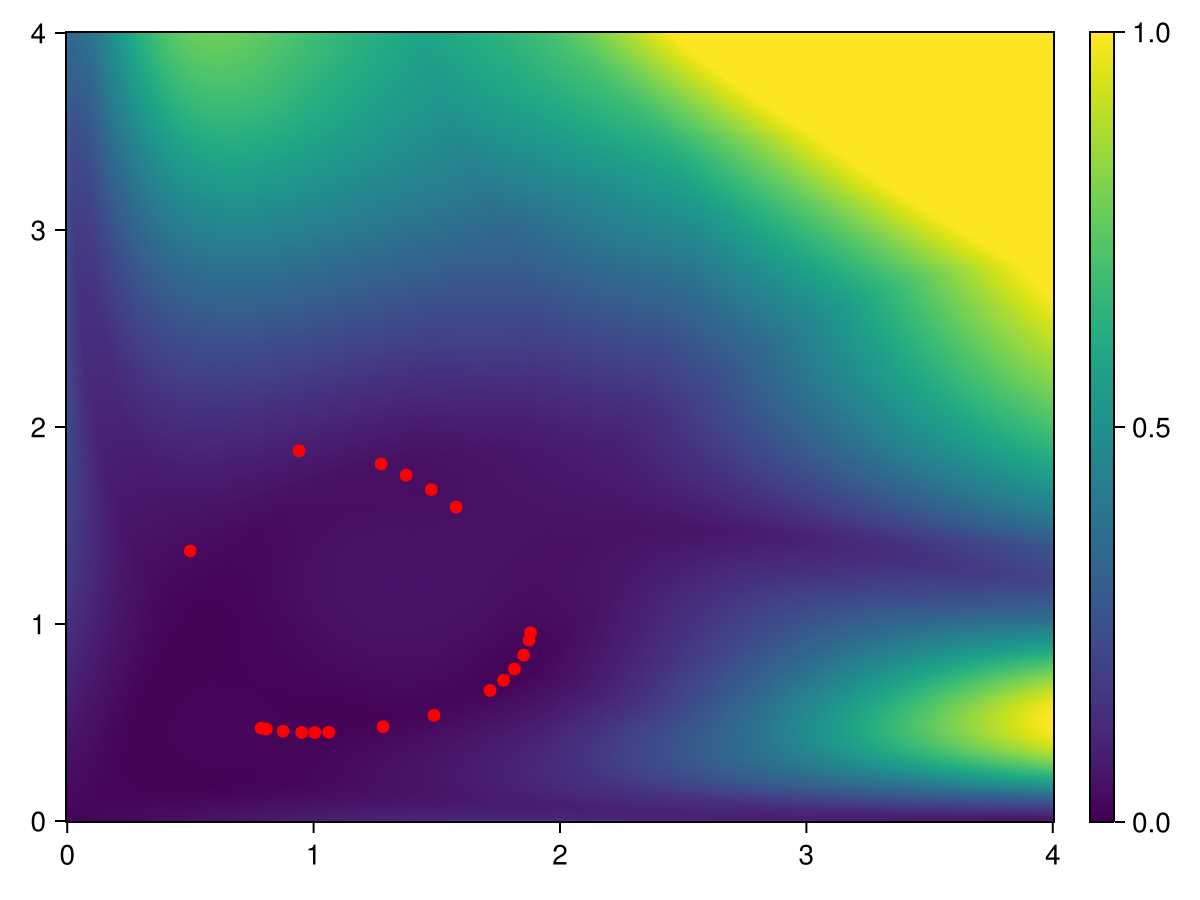}
		\subcaption{$w^*$-$u0$-$\svdu{2}$, $\delta = 0.37$}
	\end{subfigure}
	\begin{subfigure}{.3\linewidth}
		\includegraphics[width=\linewidth]{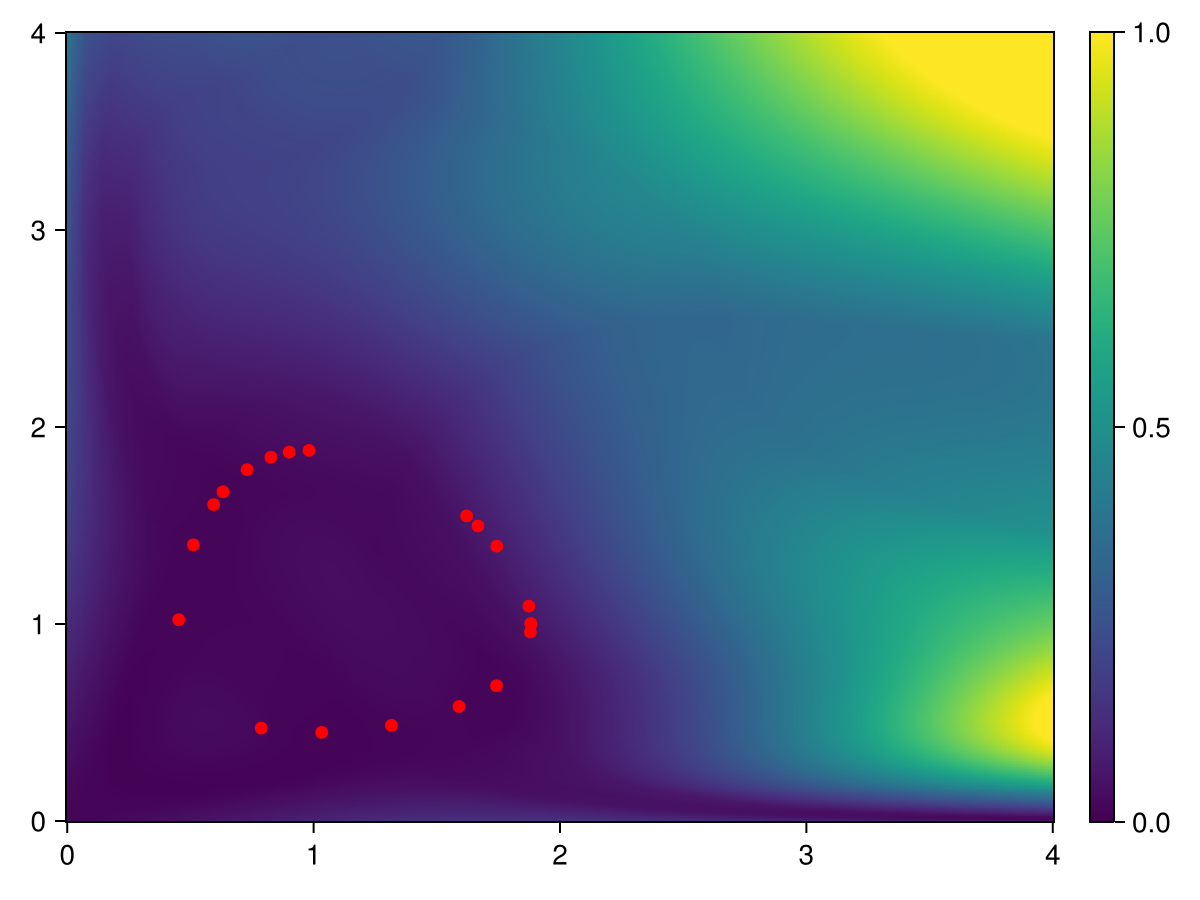}
		\subcaption{$w^*$-$u0$-$\tsvdu{2}$, $\delta = 0.30$}
	\end{subfigure}

	\begin{subfigure}{.3\linewidth}
		\includegraphics[width=\linewidth]{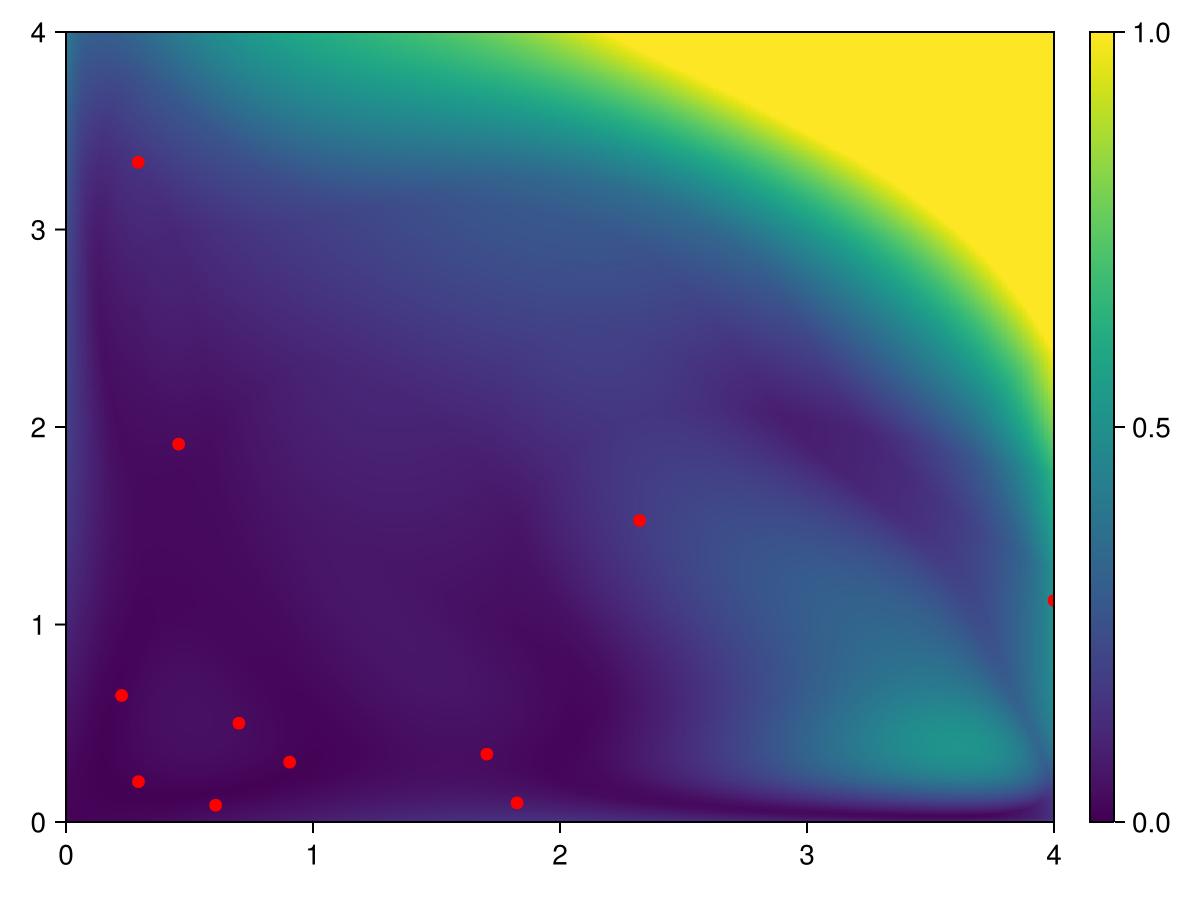}
		\subcaption{$w0$-$u^*$-$l$, $\delta = 0.32$}
	\end{subfigure}
	\begin{subfigure}{.3\linewidth}
		\includegraphics[width=\linewidth]{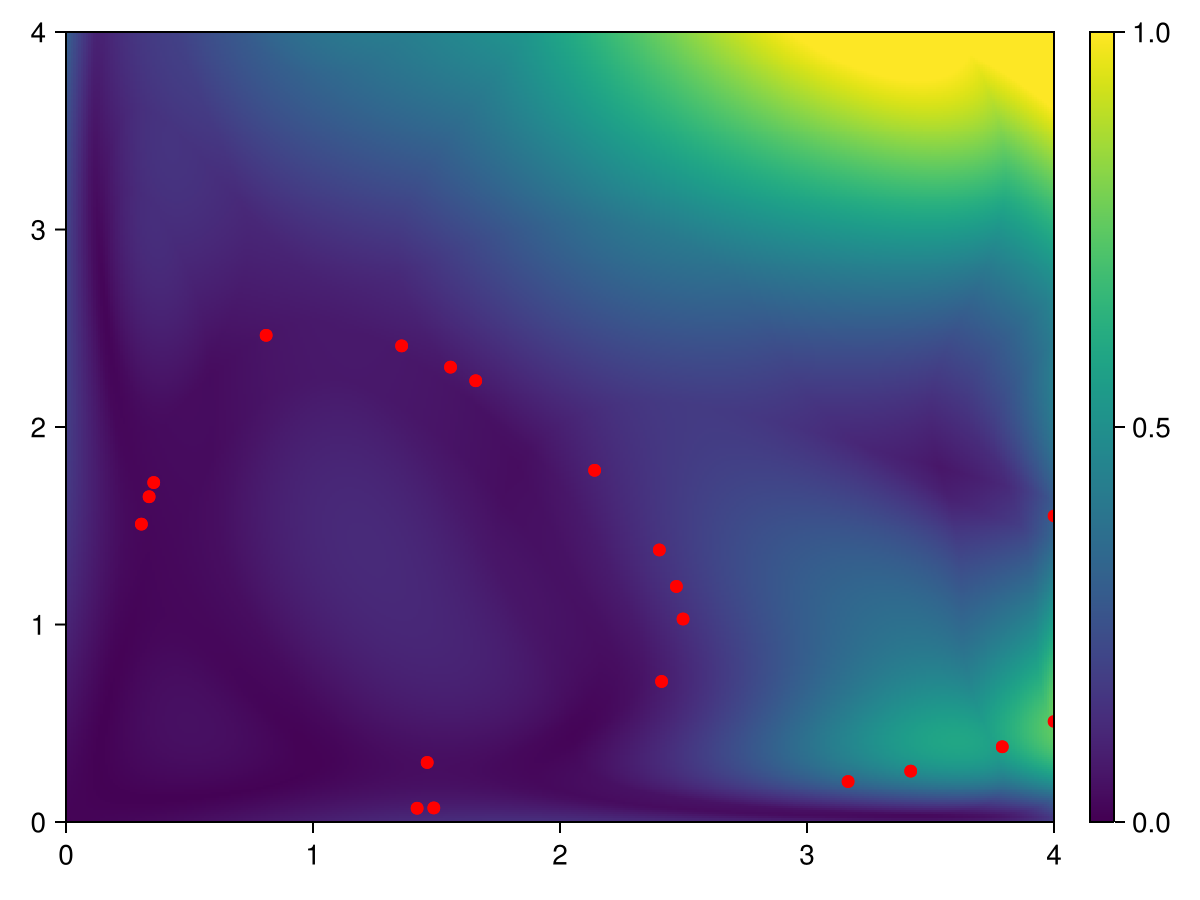}
		\subcaption{$w^*$-$u^*$-$o$, $\delta = 0.24$}
	\end{subfigure}
	\begin{subfigure}{.3\linewidth}
		\includegraphics[width=\linewidth]{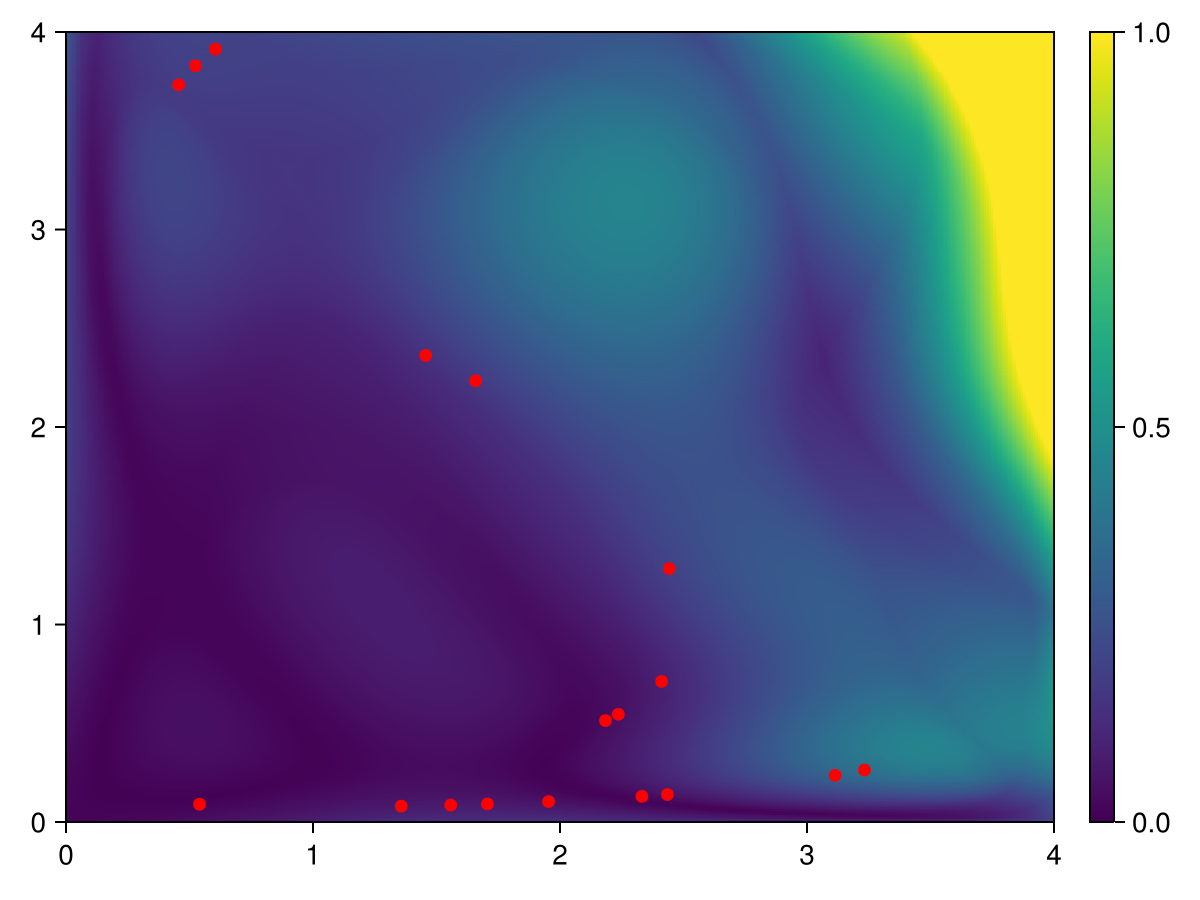}
		\subcaption{$w^*$-$u^*$-$l$, $\delta = 0.24$}
	\end{subfigure}
	
	\begin{subfigure}{.3\linewidth}
		\includegraphics[width=\linewidth]{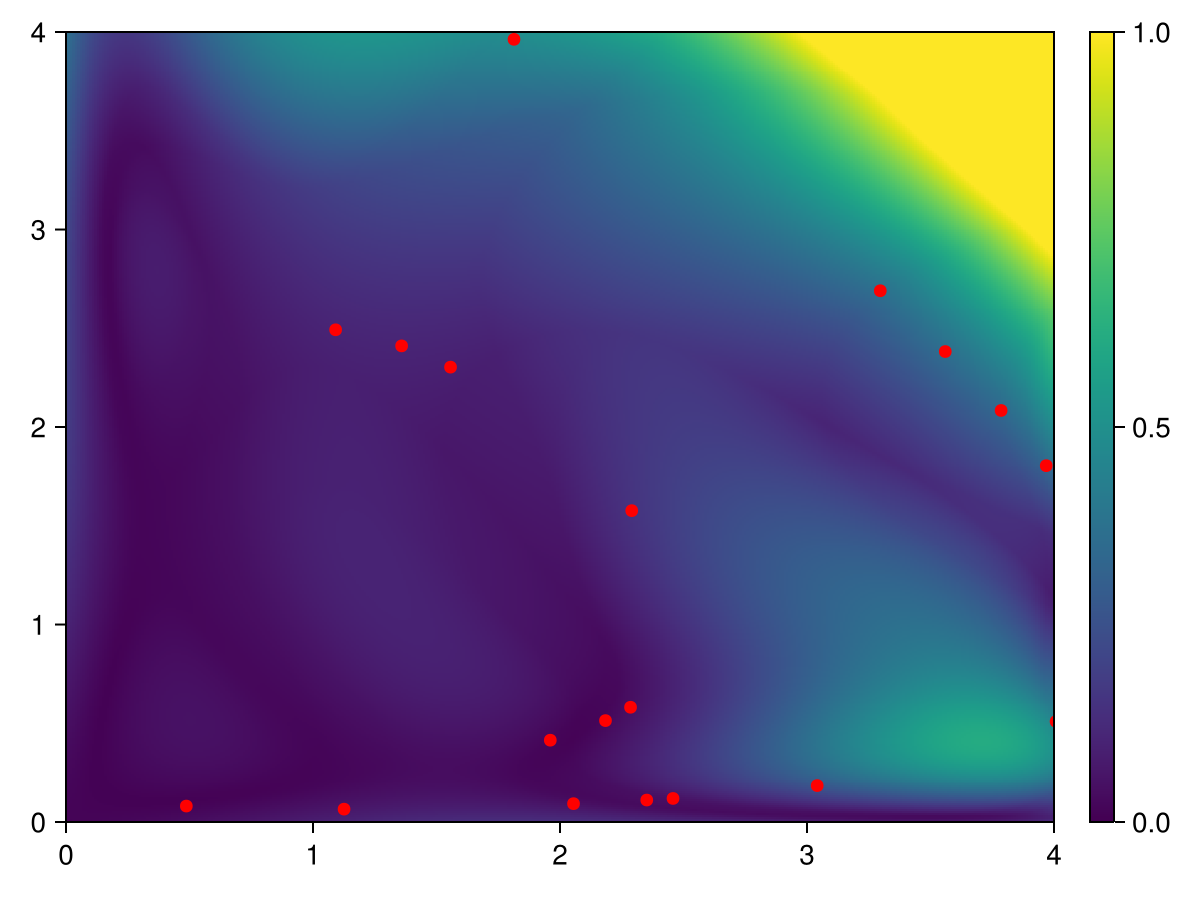}
		\subcaption{$w^*$-$u^*$-$ll$, $\delta = 0.25$}
	\end{subfigure}
	\begin{subfigure}{.3\linewidth}
		\includegraphics[width=\linewidth]{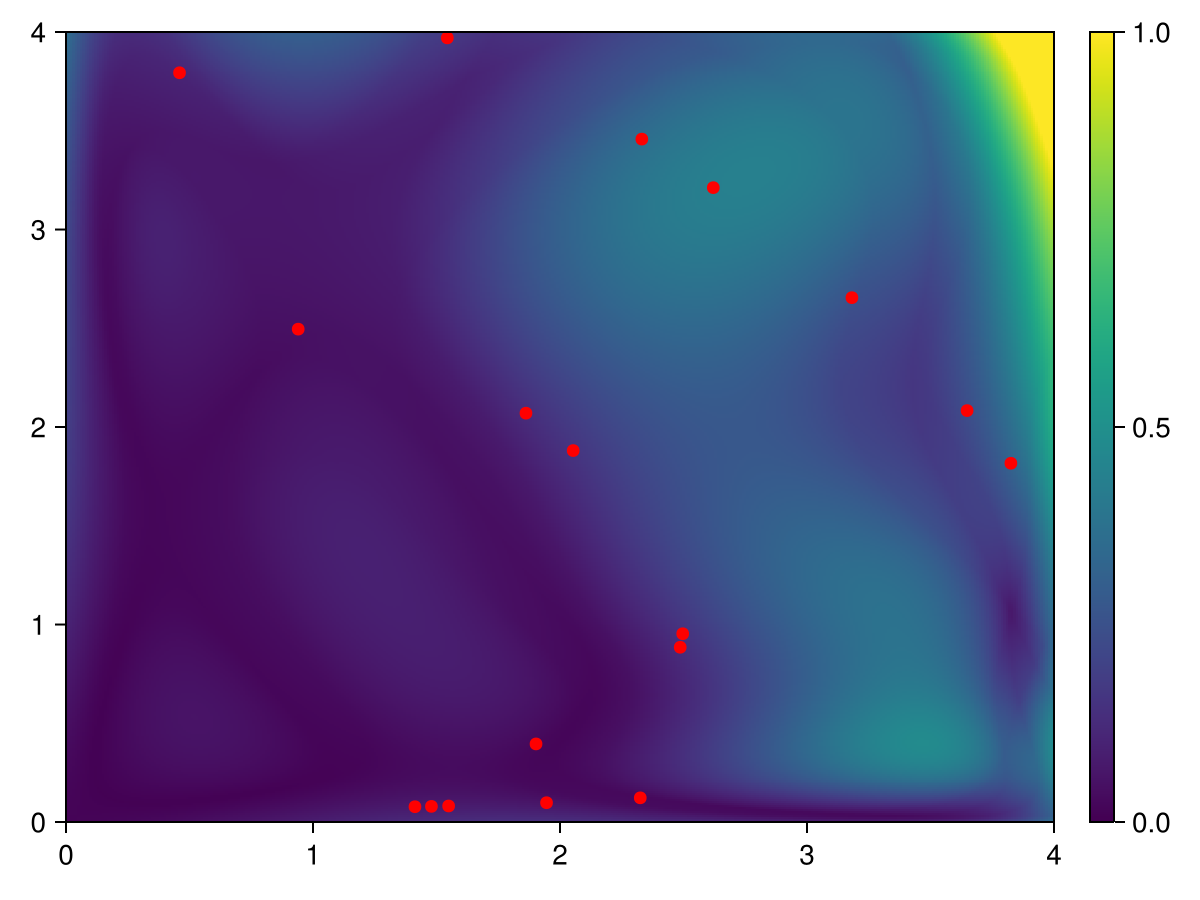}
		\subcaption{$w^*$-$u^*$-$\svdu{2}$, $\delta = 0.18$}
	\end{subfigure}
	\begin{subfigure}{.3\linewidth}
		\includegraphics[width=\linewidth]{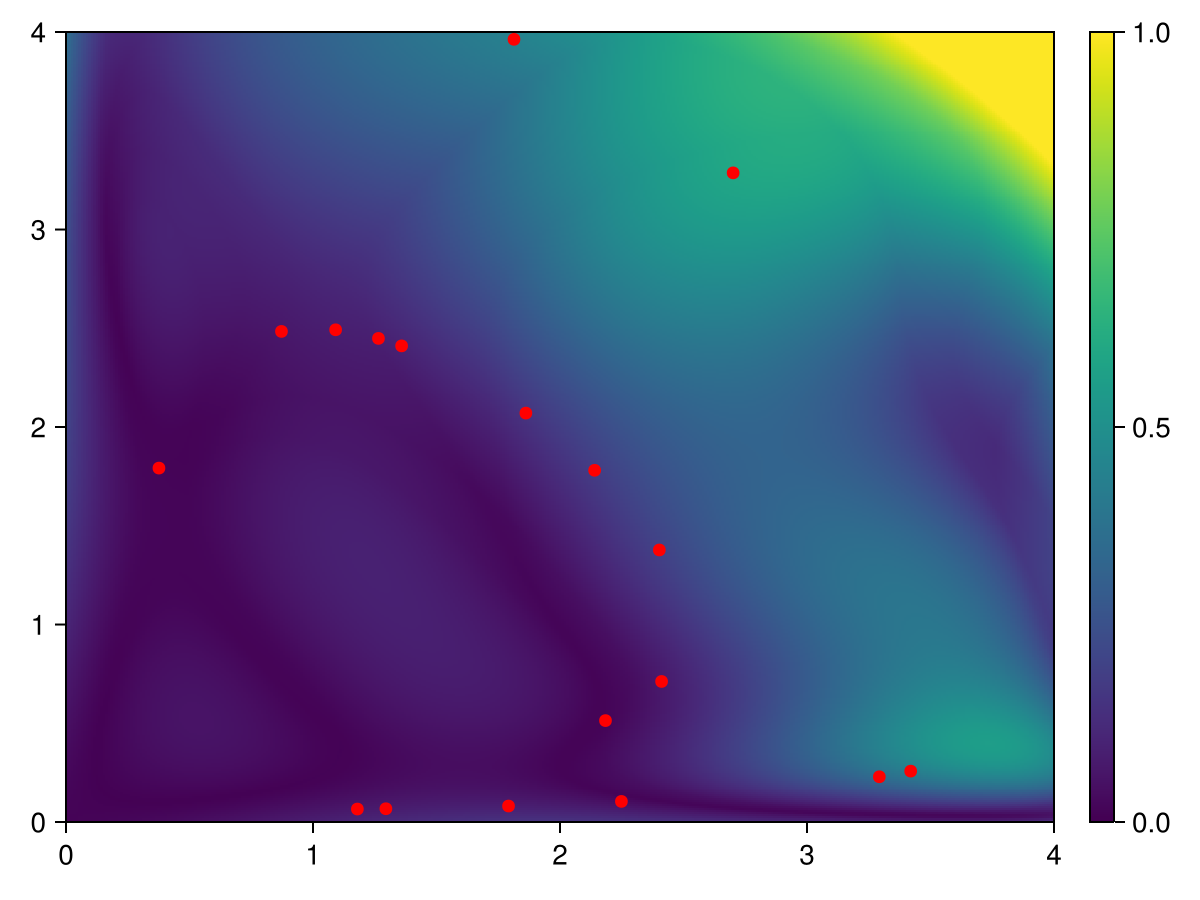}
		\subcaption{$w^*$-$u^*$-$\tsvdu{2}$, $\delta = 0.24$}
	\end{subfigure}
	\caption{Absolute errors of trained neural networks $\lvert \myU(x,y) - x\cdot y\rvert$ averaged over five training runs starting from initial $\theta$. The captions denote the applied strategy and the average absolute error $\delta$. Shown in red dots are 10 measurements per differential state which have been chosen according to the calculated OED. 
First two rows: Errors on domain $[0,4]^2$ for OED with fixed $u(t)=0$, corresponding to left columns of Table~\ref{tab:criteria_with_controls}.
Bottom rows: as above, but with $u(t) = u^*(t)$, corresponding to right columns of Table~\ref{tab:criteria_with_controls}.. 
    While the topology of uncertainty depends on the sampling $w$ and the chosen dimension reduction (first two rows), an additional excitation $u^*$ of the system (bottom rows) leads to increased accuracy.} 
	\label{fig:absolute_errors}
\end{figure}

\clearpage
% --------------------
\subsubsection{Concurrent training of ANN weights and estimation of model parameters} 

We now consider the situation where both model parameters $\hat p$ and ANN weights $\theta$ shall be inferred from experimental data.
We consider the case in which the parameter $\hat p_2$ is uncertain in addition to the weights of the ANN $U(x,\theta)$ studied in the previous section.

In Figure~\ref{fig_mixed_sensitivities} selected corresponding sensitivities are plotted, illustrating the relation between lumping and singular value decomposition and the impact of including the mechanistic model parameter $\hat{p}_2$ in the SVD.
\begin{figure}[h!!!]
	\centering
  \includegraphics[width=0.35\linewidth]{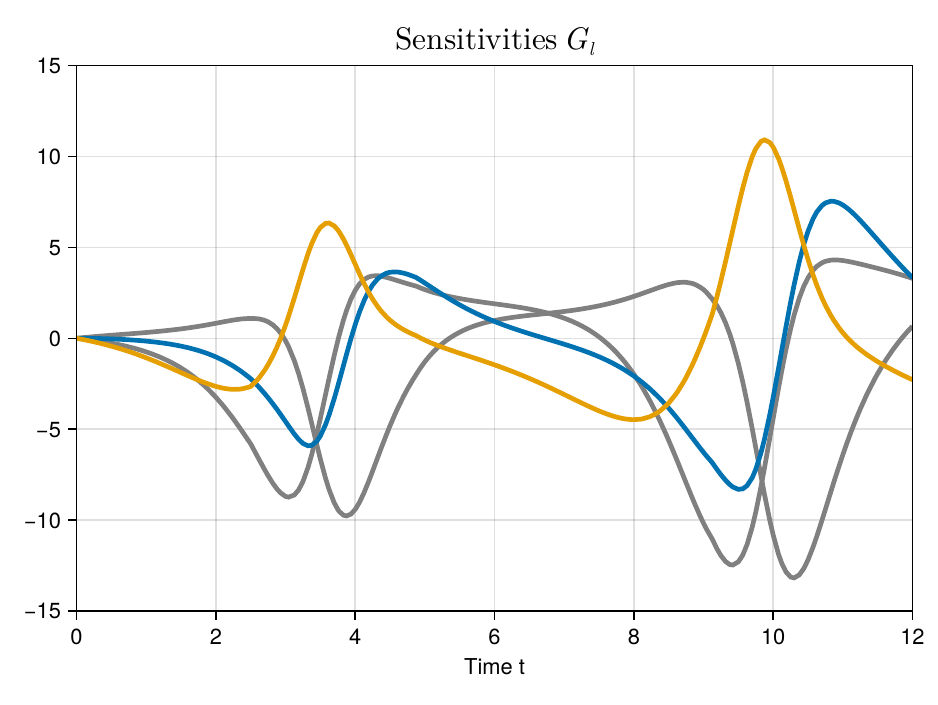}
  \includegraphics[width=0.35\linewidth]{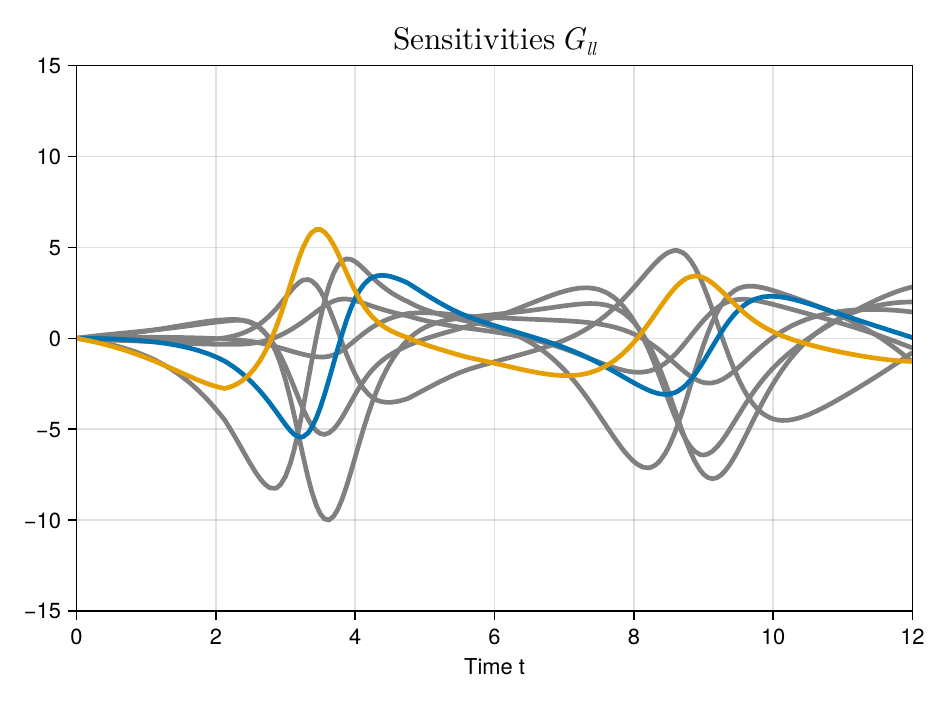}

  \includegraphics[width=0.35\linewidth]{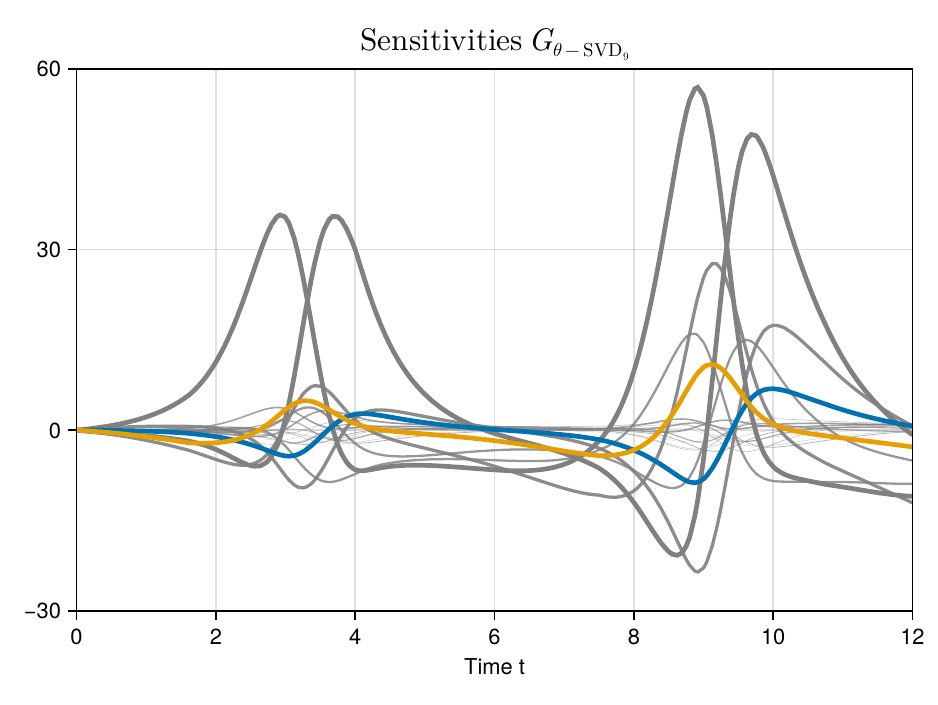}
  \includegraphics[width=0.35\linewidth]{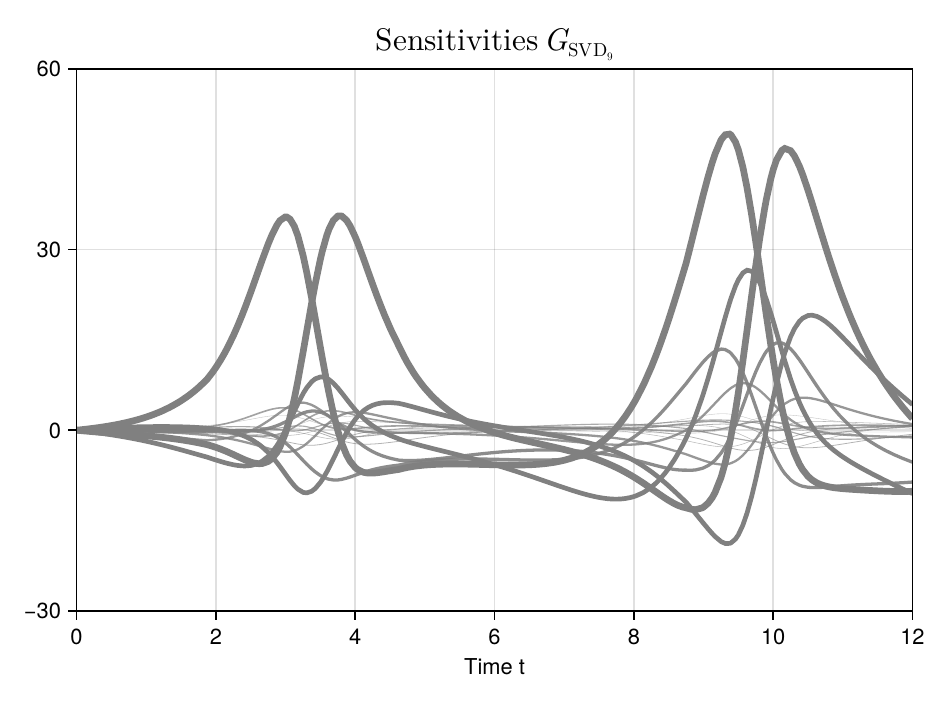}
	\caption{Comparison of sensitivities for the dimension reduction approaches $a \in \{l, ll, \psvdu{10}, \svdu{10} \}$ for the Lotka-Volterra UDE. The line width of trajectories indicates the size of the corresponding singular values. The two sensitivities $\frac{\mathrm{d} x}{\mathrm{d} \hat p_2}$ are shown in color.}
	\label{fig_mixed_sensitivities}
\end{figure}

Table~\ref{tab:criteria_with_controls_concurrent} shows the objective function values of the different OED approaches.
\begin{table}[h!!!]
	\centering
	\begin{tabular}{lr|lr|lr}
	 	Approach 			            & $\phi_{A,\svdu{10}}(w,u)$ & Approach 			  & $\phi_{A,\svdu{10}}(w,u)$ & Approach     & $\phi_{A,\svdu{10}}(w,u)$ \\ \toprule 
%		$w0$-$u^*$-$l$           	&  1.071e-01 &	-	&  - & - & -\\ 
		$w^*$-$u^*$-$o$		       	&  4.711e-02 &	$w^*$-$u^*$-$l$          	&  5.727e-02 &  $w^*$-$u^*$-$ll$	       	&  1.119e-01\\ \hline
		$w^*$-$u^*$-$\svdu{1}$   	&  4.323e-02 &	$w^*$-$u^*$-$\psvdu{1}$  	&  5.473e-02 &	$w^*$-$u^*$-$\tsvdu{1}$  	&  4.394e-02 \\
		$w^*$-$u^*$-$\svdu{2}$   	&  4.497e-02 &	$w^*$-$u^*$-$\psvdu{2}$  	&  5.206e-02 &	$w^*$-$u^*$-$\tsvdu{2}$  	&  4.945e-02 \\ 
		$w^*$-$u^*$-$\svdu{3}$   	&  6.451e-02 &	$w^*$-$u^*$-$\psvdu{3}$  	&  5.253e-02 &	$w^*$-$u^*$-$\tsvdu{3}$  	&  4.806e-02 \\
		$w^*$-$u^*$-$\svdu{4}$   	&  4.721e-02 &	$w^*$-$u^*$-$\psvdu{4}$  	&  4.871e-02 &	$w^*$-$u^*$-$\tsvdu{4}$  	&  4.710e-02\\
		$w^*$-$u^*$-$\svdu{5}$   	&  4.755e-02 &	$w^*$-$u^*$-$\psvdu{5}$  	&  4.716e-02 &	$w^*$-$u^*$-$\tsvdu{5}$  	&  4.337e-02\\
		$w^*$-$u^*$-$\svdu{6}$   	&  5.867e-02 &	$w^*$-$u^*$-$\psvdu{6}$  	&  4.381e-02 &	$w^*$-$u^*$-$\tsvdu{6}$  	&  4.293e-02\\
		$w^*$-$u^*$-$\svdu{7}$   	&  4.083e-02 &	$w^*$-$u^*$-$\psvdu{7}$  	&  4.518e-02 &	$w^*$-$u^*$-$\tsvdu{7}$  	&  4.348e-02\\
		$w^*$-$u^*$-$\svdu{8}$   	&  3.769e-02 &	$w^*$-$u^*$-$\psvdu{8}$  	&  4.454e-02 &	$w^*$-$u^*$-$\tsvdu{8}$  	&  4.348e-02\\ 
		$w^*$-$u^*$-$\svdu{9}$   	&  3.723e-02 &	$w^*$-$u^*$-$\psvdu{9}$  	&  4.571e-02 &	$w^*$-$u^*$-$\tsvdu{9}$  	&  4.348e-02\\ 
		$w^*$-$u^*$-$\svdu{10}$  	&  3.705e-02 &	$w^*$-$u^*$-$\psvdu{10}$ 	&  4.682e-02 &	$w^*$-$u^*$-$\tsvdu{10}$ 	&  4.327e-02\\
		\bottomrule
	\end{tabular}
	\caption{Objective function values corresponding to different approaches $w^*$-$u^*$-$a$. The evaluation of the objective function was done for the largest singular value by simulation of optimal solutions $(w^*, u^*)$. As expected, the objective function values improve when more singular values are considered. Interestingly, the lumping approaches $w^*$-$u^*$-$o$ and $w^*$-$u^*$-$l$ perform well when compared to SVD approaches. No clear improvements for $a=\psvdu{n_s}$ or $a=\tsvdu{n_s}$ when compared to $a=\svdu{n_s}$ are apparent.}
	\label{tab:criteria_with_controls_concurrent}
\end{table}
We finish this section by looking at the behavior of optimal samplings $w^*$ in dependence of the SVD truncation index $n_s$. 
Figure~\ref{fig_gig_lotka} shows the distribution of the ten largest singular values $\lambda_i$, the difference between optimal samplings for increasing number $n_s$ of considered singular values, and the convergence of the functions $\Gamma^{D,i}_{n_s}(t)$ defined in \eqref{eq_GammaDSVD} that explain this behavior.
\begin{figure}[h!!!]
	\centering
  \includegraphics[width=0.48\linewidth]{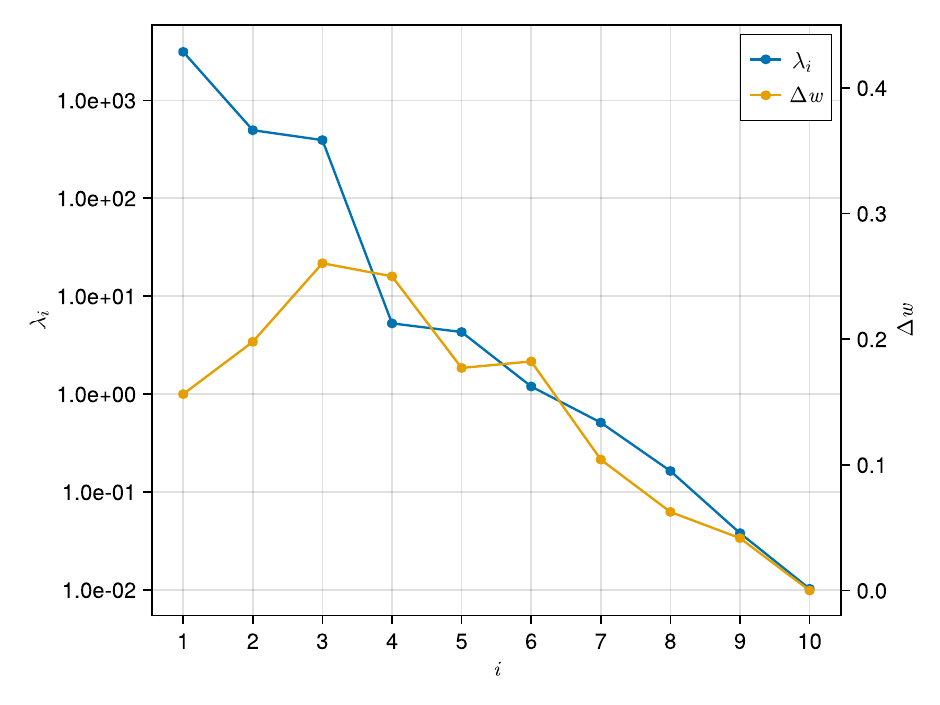}
  \includegraphics[width=0.48\linewidth]{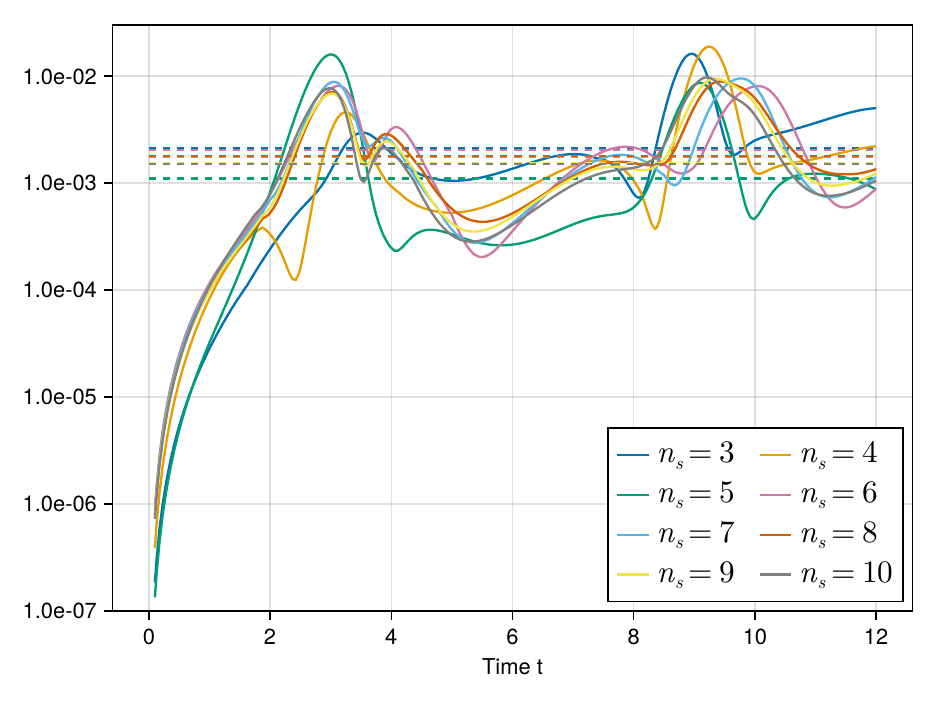}
	\caption{Left: distribution of singular values $\lambda_i$ of the SVD of the Lotka-Volterra UDE in blue on a logarithmic scale. The values $\Delta w := \frac{1}{N} \sum_{j=1}^{N} \lvert  w^*_j - w^*_{\svdu{10},j} \rvert$ with $N$ as the number of equidistant discretization intervals indicate how much the calculated optimal samplings for the problem with $n_s$ singular values differs from the solution for $n_t=10$ singular values. Right: functions $\Gamma^{D,i}_{n_s}(t)$ as defined in \eqref{eq_GammaDSVD} in solid and scaled Lagrange multipliers $\gamma^{D}_{n_s} \mu_i$ for $i=1$. The trajectories and multipliers get closer as $n_s \rightarrow n_t$, explaining via Lemma~\ref{lem_convergence} why also the optimal samplings converge.}
	\label{fig_gig_lotka}
\end{figure}

% --------------
\subsection{Urethane reaction}

We consider the chemical reaction of urethane investigated in \cite{Koerkel2002,Koerkel2004}. The reaction is given by the scheme 

\begin{align}
	\begin{split}
		A + B &\rightarrow C\\
		A + C & \rightleftarrows D\\
		3 A & \rightarrow E
	\end{split}
\end{align}
with educts isocyanate $A$, butanol $B$, value product urethane $C$, downstream product allophanate $D$, by-product isocyanurate $E$ and solvent dimethylsulfoxide $L$.
It can be described 
%by the index-1 DAE system
%%
%\begin{align}
%	\begin{split}
%		\dot n_c &= V \cdot \left( r_1 - r_2 + r_3\right) \\
%		\dot n_D &= V \cdot \left( r_2 - r_3\right) \\
%		\dot n_E &= V \cdot \left( r_4 \right) \\
%		0 &= n_A + n_C + 2n_D + 3n_E - n_{A0} - feed_1(t) \\
%		0 &= n_B + n_C + n_D - n_{B0} - feed_2(t) \\
%		0 &= n_L - n_{L0} - feed_1(t) - feed_2(t),\\
%		\dot T(t) 			&= u_T(t),\\
%		\dot{feed_1}(t) 	&= u_1(t),\\
%		\dot{feed_2}(t) 	&= u_2(t),
%	\end{split}
%\end{align}
%
%or equivalently 
by the system of ODEs
\begin{align}
	\begin{split}
		\dot n_C &= V \cdot \left( r_1 - r_2 + r_3\right) \\
		\dot n_D &= V \cdot \left( r_2 - r_3\right) \\
		\dot n_E &= V \cdot \left( r_4 \right) \\
		\dot n_A &= V \cdot (-r_1 - r_2 + r_3 - 3r_4) + u_1(t) \\
		\dot n_B &= V \cdot (- r_1) + u_2(t) \\
		\dot n_L &= u_1(t) + u_2(t),\\
		\dot T(t) 			&= u_T(t).
	\end{split} \label{eq_urethane}
\end{align}
The reaction rates $r_i$ are given by  
$r_1 = k_1 \frac{n_A}{V} \frac{n_B}{V}$,
$r_2 = k_1 \frac{n_A}{V} \frac{n_C}{V}$,
$r_3 = k_3 \frac{n_D}{V}$, and
$r_4 = k_4 \left(\frac{n_l}{V}\right)^2$, respectively,
%
%\begin{align*}
%	\begin{split}
%		r_1 &= k_1 \frac{n_A}{V} \frac{n_B}{V}, \\
%		r_2 &= k_1 \frac{n_A}{V} \frac{n_C}{V}, \\
%		r_3 &= k_3 \frac{n_D}{V}, \\
%		r_4 &= k_4 \left(\frac{n_l}{V}\right)^2, \\
%	\end{split}
%\end{align*}
%
using the reaction volume
\begin{align*}
	V &= \frac{n_A \cdot M_A}{\rho_A} + \frac{n_B \cdot M_B}{\rho_B} + \frac{n_C \cdot M_C}{\rho_C} +\frac{n_D \cdot M_D}{\rho_D} +\frac{n_E \cdot M_E}{\rho_E} +\frac{n_L \cdot M_L}{\rho_L},
\end{align*}
and Arrhenius equations for the dependency of the reaction rate on the temperature $T$
\begin{align*}
	\begin{split}
		k_1 &= k_{ref1} \cdot \exp \left( -\frac{E_{a1}}{R} \cdot \left( \frac{1}{T} - \frac{1}{T_{ref}}\right)\right), \\
		k_2 &= k_{ref2} \cdot \exp \left( -\frac{E_{a2}}{R} \cdot \left( \frac{1}{T} - \frac{1}{T_{ref}} \right)\right), \\
		k_4 &= k_{ref4} \cdot \exp \left( -\frac{E_{a4}}{R} \cdot \left( \frac{1}{T} - \frac{1}{T_{ref}}\right)\right), \\
		k_C &= k_{C2} 	\cdot \exp \left( -\frac{\Delta H_2}{R} \cdot \left( \frac{1}{T} - \frac{1}{T_{ref}}\right)\right), \\
		k_3 &= \frac{k_2}{k_c}
	\end{split}
\end{align*}
with steric factors $k_{ref1},k_{ref2},k_{ref4}$, activation energies $E_{a1}, E_{a2}, E_{a4}$, equilibrium constant $K_{C2}$, reaction enthalpy $\Delta H_2$, and reference reaction temperature $T_{ref}$ as parameters. The initial values of the products are set to zero, i.e.,
\begin{align*}
	n_C(t_0) = n_D(t_0) = n_E(t_0) = 0.
\end{align*}
Time-independent control values are given by the initial values of the educts
\begin{align*}
	n_A(t_0) = n_{A0}, n_B(t_0)=n_{B0}, n_L(t_0)=n_{L0}.
\end{align*}
Furthermore, the control functions $u_1(t), u_2(t), u_T(t), \ t \in [t_0, t_f]$ influence the evolution of the two feeds 
and the temperature $T(t)$, respectively. We assume that the reaction rate for the reaction of isocyanate to isocyanurate is unknown and needs to be learned from experimental data and thus replace $r_4$ by an ANN. Moreover, we assume that the steric factor $k_{ref_1}$ and the activation energy $E_{a1}$ involved in the reaction rate $r_1$ are uncertain parameters $\hat p$. Hence, we solve the problem \eqref{OED} for the concurrent estimation of parameters $k_{ref_1}$ and $E_{a1}$ and the weights $\theta$ with the hybrid model
\begin{align}
	\label{neural_urethan}
	\begin{split}
		\dot n_C &= V \cdot \left( r_1 - r_2 + r_3\right) \\
		\dot n_D &= V \cdot \left( r_2 - r_3\right) \\
		\dot n_E &= V \cdot \left( \myU((n_C, n_D, n_E, n_A, n_B, n_L), \theta) \right) \\
		\dot n_A &= V \cdot (-r_1 - r_2 + r_3 - 3\myU((n_C, n_D, n_E, n_A, n_B, n_L), \theta)) + u_1(t) \\
		\dot n_B &= V \cdot (- r_1) + u_2(t) \\
		\dot n_L &= u_1(t) + u_2(t),\\
		\dot T(t) 			&= u_T(t).
	\end{split}
\end{align}
which we consider on the time horizon $\mathcal T = [0, 80h]$. To generate experimental data three measurement functions are available. They are given by the mass percentages of the species isocyanate, urethane and isocyanurate, i.e.,
\begin{align*}
	\begin{split}
		h^1((n_C, n_D, n_E, n_A, n_B, n_L)) = \frac{100 \cdot n_A M_A}{\sum_{i \in \{A, B, C, D, E, L\}} n_i M_i}, \\		
		h^2((n_C, n_D, n_E, n_A, n_B, n_L)) = \frac{100 \cdot n_C M_C}{\sum_{i \in \{A, B, C, D, E, L\}} n_i M_i}, \\		
		h^3((n_C, n_D, n_E, n_A, n_B, n_L)) = \frac{100 \cdot n_E M_E}{\sum_{i \in \{A, B, C, D, E, L\}} n_i M_i}.		
	\end{split}
\end{align*}
 
For illustration purposes we restrict ourselves in the problem \eqref{OED} to the calculation of optimal feeds $u_1(t), u_2(t)$ with fixed initial values $n_{A0} = 0.1 \,\textrm{mol}, n_{B0} = 0.05 \,\textrm{mol}, n_{L0} = 0.0 \,\textrm{mol}$, and fixed temperature profile given by a linear increase from $T(0 h) = 300 K$ to $T(80 h)=400K$.

We focus directly on the interesting case of a concurrent estimation of model parameters and ANN weights.
Figure~\ref{fig_solution_urethan} shows parts of an optimal solution of \eqref{OED} for $w^*$-$u^*$-$\svdu{10}$.
The optimal controls $u^*$ are of bang-bang type and bring the system into a different area in state space. 
\begin{figure}[h!!!]
	\centering
  \includegraphics[width=0.48\linewidth]{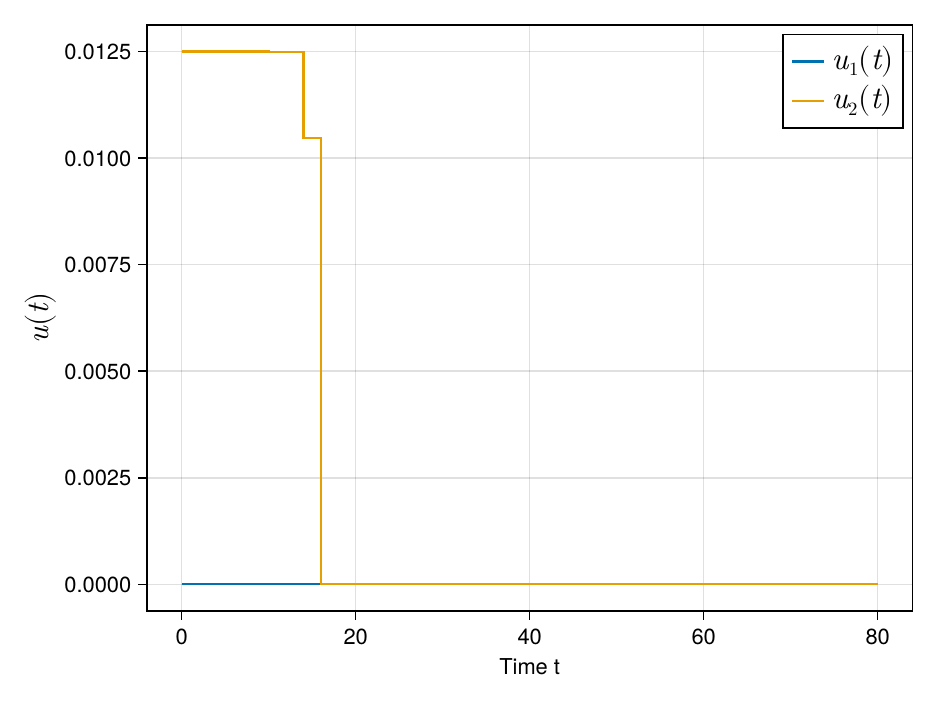}
  \includegraphics[width=0.48\linewidth]{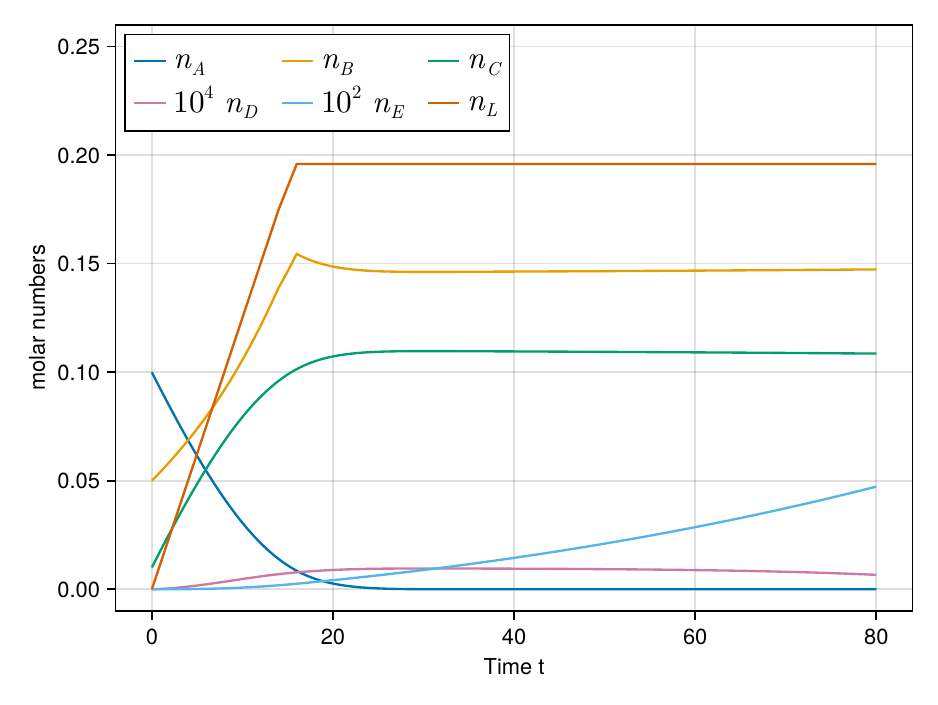}
	\caption{Optimal controls and differential states for $w^*$-$u^*$-$\svdu{10}$.}
	\label{fig_solution_urethan}
\end{figure}

The sensitivities for different approaches are shown in Figure~\ref{fig_mixed_sensitivities_urethan}.
\begin{figure}[h!!!]
	\centering
  \includegraphics[width=0.32\linewidth]{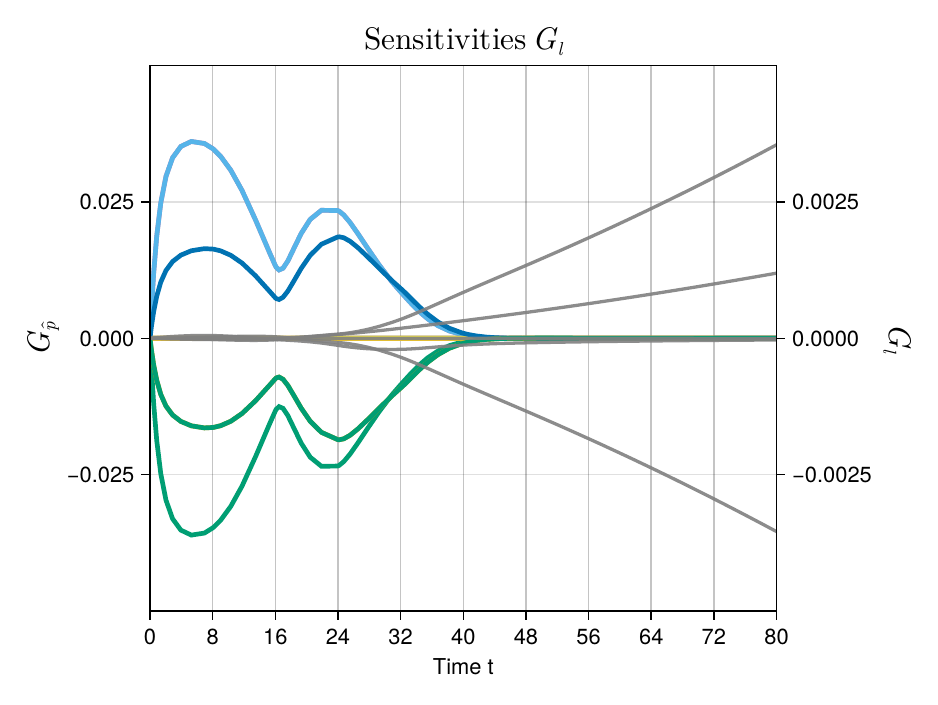}
  \includegraphics[width=0.32\linewidth]{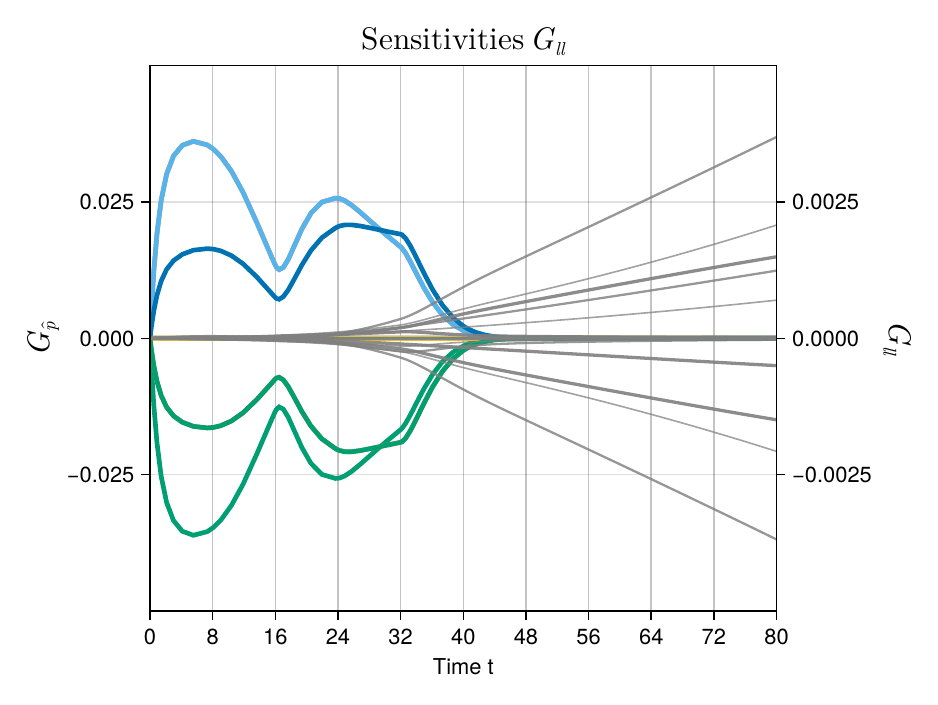}
  \includegraphics[width=0.32\linewidth]{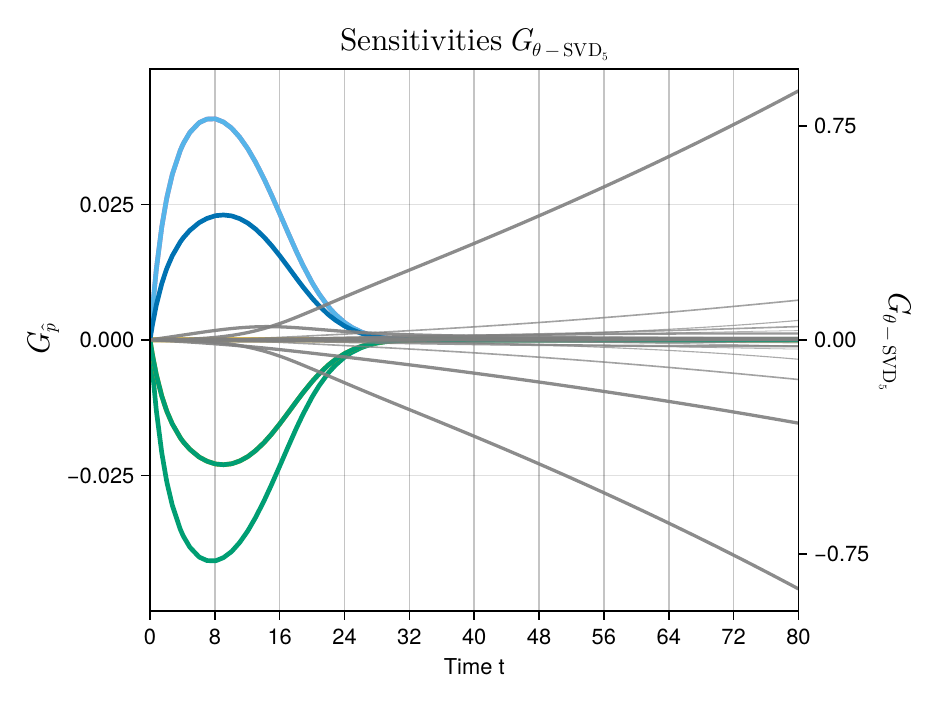}

	\caption{Visualization of the sensitivities $G_l$, $G_{ll}$, and $G_\tsvdu{5}$. The thickness of lines corresponds to the value of singular values. One observes that $G_l$ and $G_{ll}$ have a qualitatively similar behavior compared to $G_\tsvdu{5}$ for the sensitivities with respect to ANN weights (gray), while the sensitivities with respect to model parameters $\hat{p}$ (highlighted in color) are different.}
	\label{fig_mixed_sensitivities_urethan}
\end{figure}
Here one can observe that the sensitivities of the differential states with respect to $\hat{p} = (k_{ref_1}, E_{a1})$ are very small for almost all states. The four visible sensitivities differ, most apparently for the approach $a=\tsvdu{5}$, because different optimal controls $u^*$ were calculated that influence $x$ and hence also $G$. 

Table~\ref{tab:urethan_criteria_with_controls_concurrent} shows the OED objective function values for different approaches. Here, the approaches $a \in \{ o, l, ll \}$ do not perform as well as the SVD approaches. 
\begin{table}[h!!!]
	\centering
	\begin{tabular}{lr|lr|lr}
	 	Approach 			            & $\phi_{D,\svdu{10}}(w,u)$ & Approach 			  & $\phi_{D,\svdu{10}}(w,u)$ & Approach     & $\phi_{D,\svdu{10}}(w,u)$ \\ \toprule 
%		$w0$-$u^*$-$l$           	&  3.967e-02 &	-	&  - & - & -\\ 
		$w^*$-$u^*$-$o$		       	&  1.201e-02 &	$w^*$-$u^*$-$l$          	&  2.072e-02 &  $w^*$-$u^*$-$ll$	       	&  1.883e-02\\ \hline
		$w^*$-$u^*$-$\svdu{1}$   	&  1.704e-04 &	$w^*$-$u^*$-$\psvdu{1}$  	&  2.299e-04 &	$w^*$-$u^*$-$\tsvdu{1}$  	&  5.146e-03 \\
		$w^*$-$u^*$-$\svdu{2}$   	&  1.728e-04 &	$w^*$-$u^*$-$\psvdu{2}$  	&  1.652e-02 &	$w^*$-$u^*$-$\tsvdu{2}$  	&  1.537e-03 \\ 
		$w^*$-$u^*$-$\svdu{3}$   	&  1.623e-04 &	$w^*$-$u^*$-$\psvdu{3}$  	&  3.255e-03 &	$w^*$-$u^*$-$\tsvdu{3}$  	&  4.763e-02 \\
		$w^*$-$u^*$-$\svdu{4}$   	&  2.688e-03 &	$w^*$-$u^*$-$\psvdu{4}$  	&  1.863e-04 &	$w^*$-$u^*$-$\tsvdu{4}$  	&  2.119e-02\\
		$w^*$-$u^*$-$\svdu{5}$   	&  2.781e-03 &	$w^*$-$u^*$-$\psvdu{5}$  	&  3.179e-03 &	$w^*$-$u^*$-$\tsvdu{5}$  	&  2.647e-04\\
		$w^*$-$u^*$-$\svdu{6}$   	&  1.623e-04 &	$w^*$-$u^*$-$\psvdu{6}$  	&  3.246e-03 &	$w^*$-$u^*$-$\tsvdu{6}$  	&  2.260e-04\\
		$w^*$-$u^*$-$\svdu{7}$   	&  2.779e-03 &	$w^*$-$u^*$-$\psvdu{7}$  	&  3.041e-03 &	$w^*$-$u^*$-$\tsvdu{7}$  	&  2.049e-04\\
		$w^*$-$u^*$-$\svdu{8}$   	&  2.913e-03 &	$w^*$-$u^*$-$\psvdu{8}$  	&  4.019e-03 &	$w^*$-$u^*$-$\tsvdu{8}$  	&  2.047e-04\\ 
		$w^*$-$u^*$-$\svdu{9}$   	&  2.788e-03 &	$w^*$-$u^*$-$\psvdu{9}$  	&  4.000e-03 &	$w^*$-$u^*$-$\tsvdu{9}$  	&  1.895e-04\\ 
		$w^*$-$u^*$-$\svdu{10}$  	&  2.680e-03 &	$w^*$-$u^*$-$\psvdu{10}$ 	&  2.799e-03 &	$w^*$-$u^*$-$\tsvdu{10}$ 	&  9.720e-03\\
			\midrule
	\end{tabular}
	\caption{Same setting as in Table~\ref{tab:criteria_with_controls_concurrent}, but for the urethane problem \eqref{neural_urethan}.}
	\label{tab:urethan_criteria_with_controls_concurrent}
\end{table}
The non-monotonicity with respect to increasing numbers $n_s$ of singular values indicates the effect of noise on the results. 
But also local optimality plays a role. Application of the local NLP solver \texttt{ipopt} to the discretized, nonconvex optimal control problem \eqref{OED} for \eqref{neural_urethan} only provides candidates for locally optimal solutions. For $a=\svdu{n_s}$ and $n_s \in \{1, 2, 3, 6\}$ slightly different (better) local optima were found.

Finally, Figure~\ref{fig_gig_urethan} visualizes the distribution of singular values, the Hamming distance of optimal samplings, and the scaled information gain functions resulting from the analysis of the necessary conditions of optimality for optimal samplings. As in the previous section, it can be observed that functions and multipliers are more and more similar for $n_s \rightarrow n_t$, with the exception of the outliers $n_s=3$ and $n_s=6$ due to local optima.
\begin{figure}[t!!!]
	\centering
  \includegraphics[width=0.4\linewidth]{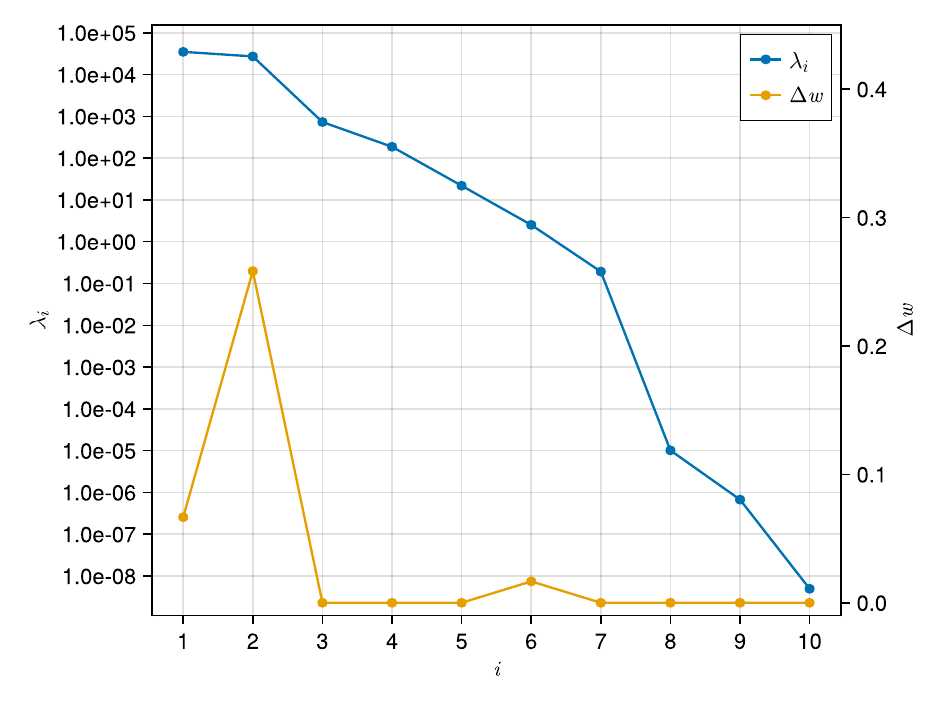}
  \includegraphics[width=0.4\linewidth]{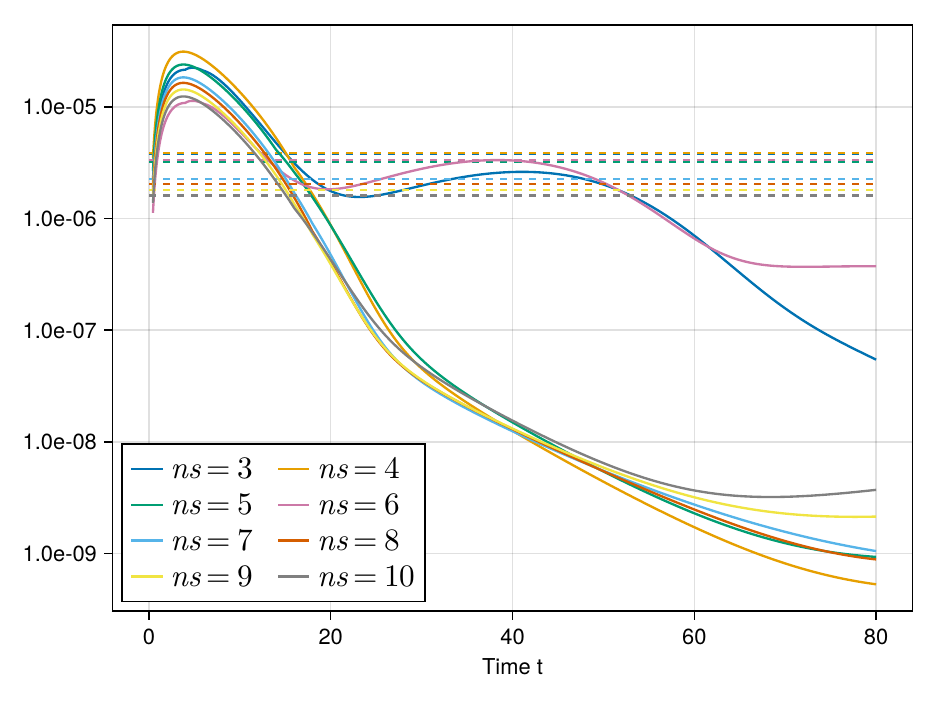}

  \includegraphics[width=0.4\linewidth]{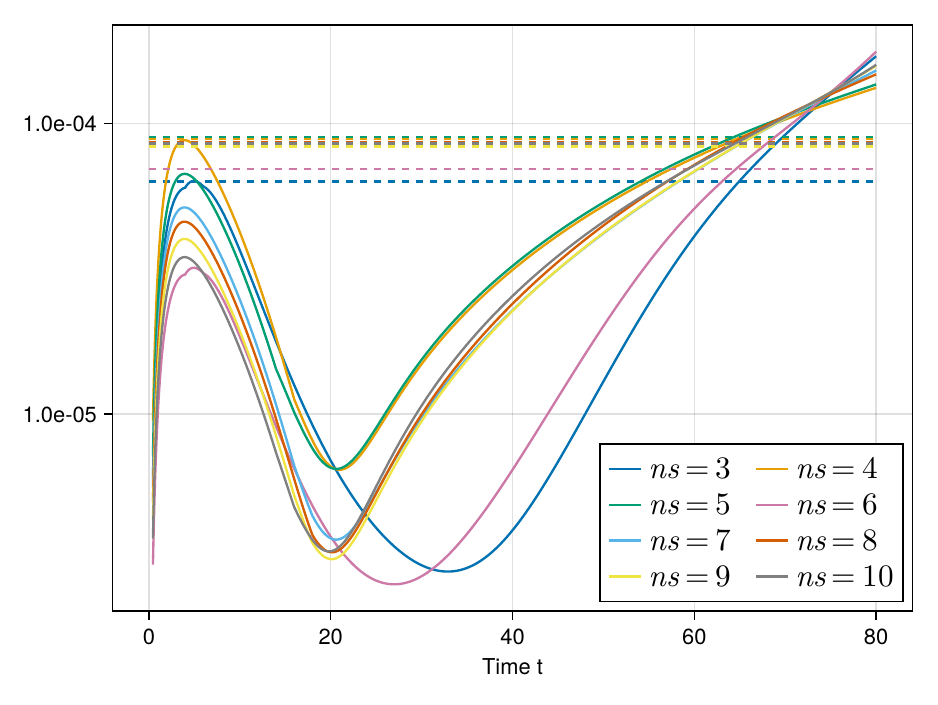}
  \includegraphics[width=0.4\linewidth]{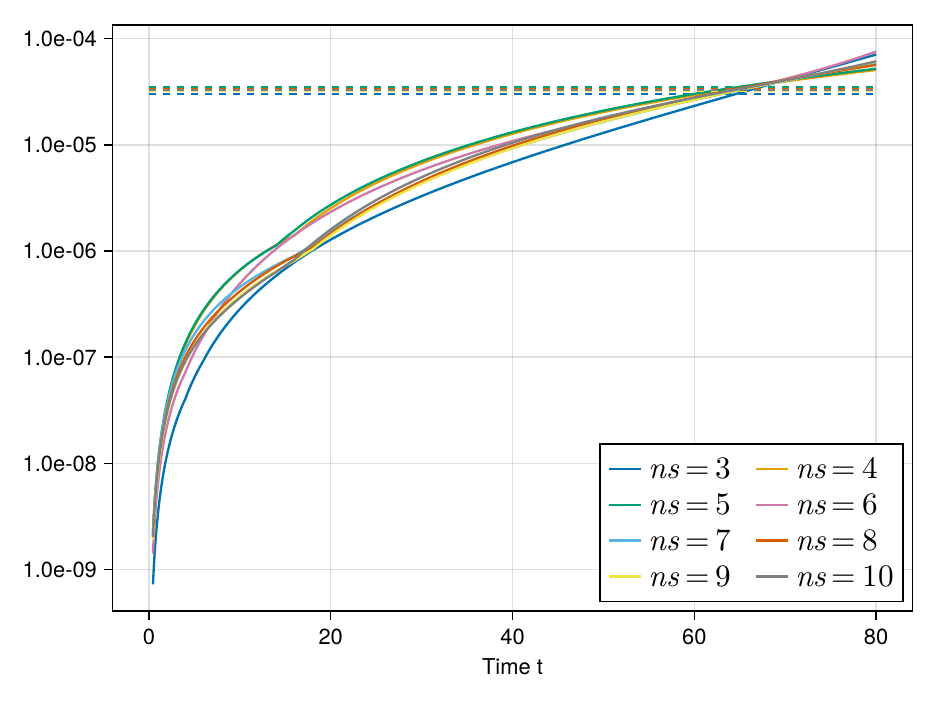}
	\caption{Same setting as in Table~\ref{tab:criteria_with_controls_concurrent}, but for the urethane problem \eqref{neural_urethan}. The top left plot shows the distribution of singular values and the Hamming distance of optimal samplings for different values of $n_s$ from the one for $n_t=10$. The other plots show the functions $\Gamma^{D,i}_{n_s}(t)$ as defined in \eqref{eq_GammaDSVD} in solid and scaled Lagrange multipliers $\gamma^{D}_{n_s} \mu_i$ for $i \in \{1,2,3\}$. The convergence of Lemma~\ref{lem_convergence} is visible, with the exception of local optima outliers $n_s=3$ and $n_s=6$.}
	\label{fig_gig_urethan}
\end{figure}

\section{Discussion and outlook} \label{sec:discussion}

We discussed optimal experimental design for universal differential equations. In particular, we compared a novel sensitivity lumping approach as an alternative to the established truncated SVD method. It uses the trick of an artificial scaling parameter with value 1.0 to derive a simple multiplication formula for the dimension reduction of the sensitivity equations that uses the current ANN weight estimate. 

By implementing all algorithms based on the julia package \texttt{dynamicOED.jl} we created a modular platform that allows to solve challenging OED problems and compare novel algorithmic ideas. Active learning for hybrid problems can thus be treated by efficient numerical algorithms for optimal control. 

Part of this comparison is the analysis of global information gain, which yields necessary conditions of optimality for the sampling decisions $w$. We motivated that this could be used in the future as a criterion for deciding the truncation point of SVD, for determining when to update SVD decompositions during an optimization algorithm, or to derive penalizations for the training of ANN weights resulting in better sampling decisions.

In summary, lumping approaches might become a feasible alternative to established reduction approaches for OED of UDE problems, especially when singular value decompositions become computationally expensive due to the size of the involved ANN.

\paragraph{Acknowledgements} 
This project has received funding from the European Regional Development Fund (grants timingMatters and IntelAlgen) under the European Union's Horizon Europe Research and Innovation Program, from the German Research Foundation DFG within GRK 2297 'Mathematical Complexity Reduction' and priority program 2331 ’Machine Learning in Chemical Engineering’ under grants KI 417/9-1, SA 2016/3-1, SE 586/25-1, which we gratefully acknowledge.

\bibliographystyle{plain}
\bibliography{literature,agsager}

\appendix

%\section{Omitted Proof in Section~\ref{sec:methodology}}
%\label{app:1}
%

\end{document}